\documentclass[a4paper,11pt,reqno,twoside]{amsart}
\usepackage{amsthm,amssymb}
\usepackage{url}
\usepackage{mathtools}
\usepackage{mathdots}
\newcommand{\BigFig}[1]{\parbox{12pt}{\Huge #1}}
\newcommand{\BigZero}{\BigFig{0}}

\numberwithin{equation}{section}

\newtheorem{theorem}{Theorem}[section]
\newtheorem{proposition}{Proposition}[section] 
\newtheorem{lemma}[proposition]{Lemma}
\newtheorem{corollary}[proposition]{Corollary}
\newtheorem{remark}[proposition]{Remark}

\newcommand*{\C}{\mathbb{C}}
\newcommand*{\R}{\mathbb{R}}
\newcommand*{\Q}{\mathbb{Q}}
\newcommand*{\Z}{\mathbb{Z}}

\newcommand{\comment}[1]{}

\title[Inverse problems for canonical systems]%
      {An inverse problem for a class of canonical systems and its applications to self-reciprocal polynomials} 
\author[M. Suzuki]{Masatoshi Suzuki}
\subjclass[2000]{30C15, 34A55, 34L40}
\keywords{}
\AtBeginDocument{%
\mathtoolsset{showonlyrefs,mathic = true}
\begin{abstract}
A canonical system is a kind of first-order system of ordinary differential equations on an interval of the real line parametrized by complex numbers. 
It is known that any solution of a canonical system generates an entire function of the Hermite-Biehler class. 
In this paper, we deal with the inverse problem to recover a canonical system from a given entire function of the Hermite-Biehler class satisfying appropriate conditions. 
This inverse problem was solved by de Branges in 1960s. However his results are often not enough to investigate a Hamiltonian of recovered canonical system. 
In this paper, we present an explicit way to recover a Hamiltonian from a given exponential polynomial belonging to the Hermite-Biehler class. After that, we apply it to study distributions of roots of self-reciprocal polynomials.
\end{abstract}
\maketitle
}
\begin{document}

\section{Introduction} \label{section_1}

Let $H(a)$ be a $2\times 2$ real symmetric matrix-valued function 
defined almost everywhere on a finite interval $I=[a_1,a_0)$ $(0<a_1<a_0 < \infty)$ 
with respect to Lebesgue measure $da$. 
We refer to a first-order system of differential equations 
\begin{equation} \label{can_0}
-a\frac{d}{da}
\begin{bmatrix}
A(a,z) \\ B(a,z)
\end{bmatrix}
= z 
\begin{bmatrix}
0 & -1 \\ 1 & 0
\end{bmatrix}
H(a)
\begin{bmatrix}
A(a,z) \\ B(a,z)
\end{bmatrix}, 
\quad 
\lim_{a \nearrow a_0}
\begin{bmatrix}
A(a,z) \\ B(a,z)
\end{bmatrix}
= 
\begin{bmatrix}
1 \\ 0
\end{bmatrix}
\end{equation}
on $I$ parametrized by $z \in \C$ as a {\bf quasi-canonical system} (on $I$). 
A column vector-valued function $(A(\cdot,z),B(\cdot,z)):I \to \C^{2\times 1}$ is called a {\bf solution} 
if it consists of absolutely continuous functions and satisfies \eqref{can_0} almost everywhere on $I$ for every fixed $z \in \C$. 

A quasi-canonical system \eqref{can_0} is called a {\bf canonical system} if 
\begin{enumerate}
\item[(H1)] $H(a)$ is a real positive-semidefinite symmetric matrix for almost every $a \in I$, 
\item[(H2)] $H(a)\not\equiv 0$ on any subset of $I$ with positive Lebesgue measure,  
\item[(H3)] $H(a)$ is locally integrable on $I$ with respect to $da/a$,  
\end{enumerate}
and $H(a)$ is called the {\bf Hamiltonian} of the system. 
Abusing language, if it cause no confusion, we often call a matrix-valued function $H(a) $ a Hamiltonian 
if a quasi-canonical system \eqref{can_0} is not a canonical system. 

A number of different second-order differential equations,  
such as Schr{\"o}dinger equations and Sturm--Liouville equations of appropriate form, 
and systems of first-order differential equations 
such as Dirac type systems of appropriate form 
are reduced to a canonical system. 
Fundamental results on the spectral theory of canonical systems were established in works of 
Gohberg--Kre{\u\i}n~\cite{GoKr67}, de Branges~\cite{deBranges68}  
and many other authors; see the survey articles Winkler~\cite{Win14},  Woracek~\cite{Wo} and references there in 
for historical details on canonical systems. 
Note that the variable $a$ in \eqref{can_0} is the variable on the multiplicative group $\R_{>0}$. 
By the change of variable $a=e^{-x}$, \eqref{can_0} is transformed into the equation for the variable $x$ on the additive group $\R$, 
and the right endpoint of $a$ for the initial value is transformed into the left endpoint of $x$ for the initial value. 
The transformed equation is the one treated by de Branges, Kre{\u\i}n, Kac, and others; 
see the final paragraph of the introduction for the reason we use a multiplicative variable.
\medskip

The subject of the present paper is an inverse spectral problem for canonical systems. 
In order to state the problem, we review the theory of the Hermite--Biehler class. 
\smallskip

We use the notation $F^\sharp(z)=\overline{F(\bar{z})}$ for functions of the complex variable $z$,  
and denote by $\C_+$ the (open) upper half-plane $\{z=x+iy \in \C\,:\, y>0\}$. 
An entire function $E(z)$ satisfying 
\begin{equation} \label{HB}
|E^\sharp(z)| < |E(z)| \quad \text{for every $z \in \C_+$}
\end{equation}
and having no real zeros 
is said to be a function of the {\bf Hermite--Biehler class}, or the {\bf class HB}, for short. 
(This definition of the class HB is equivalent to the definition of Levin~\cite[Section 1 of Chapter VII]{Levin80} 
if ``the upper half-plane'' is replaced by ``the lower half-plane'', 
because \eqref{HB} implies that $E(z)$ has no zeros in $\C_+$. 
We adopt the above definition for the convenience of using the theory of canonical systems via the theory of de Branges spaces.) 
\smallskip

Suppose that the system \eqref{can_0} is a canonical system 
endowed with a solution $(A(a,z)$, $B(a,z))$. 
Then $E(a,z):=A(a,z)-iB(a,z)$ is a function of the class HB 
for every fixed point $a \in I$. 
In particular, $\lim_{a \nearrow a_0}E(a,z)=1$ and $E(a_1,z)$ is a function of the class HB. 
Therefore, an inverse problem for canonical systems 
is to recover their Hamiltonians from given entire functions of the class HB satisfying appropriate conditions. 
Usually, such an inverse problem is difficult to solve in general, 
because it is an inverse spectral problem (\cite[Section 2]{Lagarias06}, \cite[Section 7]{Remling02}). 
However it was already solved by the theory of de Branges in 1960s (\cite{deBranges68}). 
For example, for fixed $I$ and an entire function $E(z)$ of exponential type 
belonging to the class HB 
and satisfying $\int_{-\infty}^{\infty}(1+x^2)^{-1}|E(x)|^{-2}\,dx< \infty$ and $E(0)=1$,  
there exists a canonical system (that is, there exists a Hamiltonian $H(a)$ on $I$) having 
\[
A(a_1,z)=\frac{1}{2}(E(z)+E^\sharp(z)), \quad 
B(a_1,z)=\frac{i}{2}(E(z)-E^\sharp(z)) 
\]
as its solution 
(see \cite[Theorem 7.3]{Remling02} with \cite[Lemma 3.3]{Dym70}). 
Moreover, such a canonical system is uniquely determined by $E(z)$ and $I$ under appropriate normalizations. 
More general situation is treated in Kac~\cite{Kats07}; see also \cite{Win14} and \cite{Wo}. 

As above, de Bragnes's theory ensures the existence of canonical systems or Hamiltonians for given functions of the class HB, 
but it does not provide explicit or useful expressions of Hamiltonians. 
In fact, an explicit form of $H(a)$ is not known except for a few examples of $E(z)$, 
as in \cite[Chapter 3]{deBranges68}, \cite[Section 8]{Dym70},  
and some additional examples constructed from such known examples 
using transformation rules for Hamiltonians and Weyl functions \cite{Win95}.  
\medskip

In this paper, we deal with the above inverse problem for a special class of exponential polynomials, 
together with the problem of explicit constructions of $H(a)$,  
and apply the results to the study of the distribution of roots of self-reciprocal polynomials. 
\medskip

Let $\R^\ast:=\R \setminus \{0\}$, $g \in \Z_{>0}$, and $q>1$. 
We denote by $\underline{C}$ a vector of length $2g+1$ of the form
$
\underline{C}= 
(C_g,C_{g-1},\cdots,C_{-g}) \in \R^\ast \times \R^{2g-1} \times \R^\ast
$, 
and consider exponential polynomials  
\begin{equation} \label{0418_2}
E(z):=E_q(z;\underline{C}):=\sum_{m=-g}^{g}C_m q^{imz}
\end{equation}
along with the associated functions 
\begin{equation} \label{def_A}
A(z) := A_q(z;\underline{C}) := \frac{1}{2}(E_q(z;\underline{C})+E_q^\sharp(z;\underline{C})), 
\end{equation}
\begin{equation} \label{def_B}
B(z) := B_q(z;\underline{C}) := \frac{i}{2}(E_q(z;\underline{C})-E_q^\sharp(z;\underline{C})).
\end{equation}
A basic fact (see, e.g., \cite[Chapter VII, Theorem 6]{Levin80}) 
is that an exponential polynomial $E(z)$ of \eqref{0418_2} belongs to the class HB 
if and only if it has no zeros in the closed upper half-plane $\C_+ \cup \R=\{z=x+iy \in \C \,:\, y \geq 0\}$. 
Therefore, there exists a Hamiltonian $H(a)$ of a canonical system corresponding to $E(z)$ if $E(z)$ has no zeros in $\C_+ \cup \R$. 
Beyond such results standing on a general theory, 
we show that there exists a real symmetric matrix-valued function $H(a)$ of a quasi-canonical system corresponding to an exponential polynomial $E(z)$ 
if $E(z)$ satisfies a weaker condition assumed in Theorem \ref{thm_01}. 
It is constructed explicitly as follows.  

We define lower triangular matrices $E^+$ and $E^-$ of size $2g+1$ by  
\[
E^{\pm}=
E^{\pm}(\underline{C}) := 
\left[
\begin{array}{lllllll}
C_{\mp g} &   &  &  &  &  &  \\
C_{\mp (g-1)} & C_{\mp g} & & & & & \\
\vdots & \ddots & \ddots & & & & \\
C_{0} & C_{\mp 1} & \ddots & C_{\mp g} & & & \\
\vdots & \ddots & \ddots & \ddots & \ddots & &  \\
C_{\pm (g-1)} & C_{\pm (g-2)} & \ddots & \ddots & \ddots & C_{\mp g} & \\
C_{\pm g}   & C_{\pm (g-1)} & \cdots & C_{0} & \cdots & C_{\mp (g-1)} & C_{\mp g} \\
\end{array}
\right]
\] 
and define square matrices $J_n$ of size $2g+1$  by
\[
J_n=J_{n}^{(2g+1)}:=
\left[ 
\begin{array}{cc|cc}
 & & &\\
 \multicolumn{2}{c|}{\raisebox{1.5ex}[0pt]{$J^{(n)}$}}& \multicolumn{2}{c}{\raisebox{1.5ex}[0pt]{\BigZero}}\\ \hline
 &  & & \\
\multicolumn{2}{c|}{\raisebox{1.5ex}[0pt]{\BigZero}} & \multicolumn{2}{c}{\raisebox{1.5ex}[0pt]{\BigZero}}  \\
\end{array}
\right], \quad 
J^{(n)} := 
\begin{bmatrix}
 & & 1 \\
 & \iddots & \\
1 & & 
\end{bmatrix}
\]
for $1 \leq n \leq 2g$, where $J^{(n)}$ is the antidiagonal matrix of size $n$ with $1$'s on the antidiagonal line. 
Using the above matrices, we define the numbers $\Delta_n$ by 
\begin{equation} \label{0423_7}
\Delta_n=\Delta_n(\underline{C}) := 
\frac{\det(E^{+}(\underline{C})+E^{-}(\underline{C})J_{n})}
{\det(E^{+}(\underline{C})-E^{-}(\underline{C})J_{n})} 
\end{equation}
for $1 \leq n \leq 2g$ and $\Delta_0=\Delta_0(\underline{C}) := 1$. 
We write $\Delta_n=\infty$ if $\det(E^+-E^-J_n)=0$. 
In addition, we define the real-valued locally constant function $\gamma:[1,q^g) \to \R$ by 
\begin{equation} \label{0425_2}
\gamma(a) =\gamma(a;\underline{C}) := \Delta_{n-1}(\underline{C})\Delta_n(\underline{C}) 
\quad \text{for} \quad q^{(n-1)/2} \leq a < q^{n/2}.
\end{equation}
Then, we obtain the following results 
for the inverse problem associated with the exponential polynomial \eqref{0418_2}.
\begin{theorem} \label{thm_01}
Let $q>1$ and $\underline{C} \in \R^\ast \times \R^{2g-1} \times \R^\ast$. 
Let $E(z)$ be the exponential polynomial defined by \eqref{0418_2}. 
Define the square matrix of size $2n$ 
\begin{equation*} 
D_n(\underline{C}) :=\left[
\begin{array}{llll|llll}
C_{-g} & & & & C_{g} & & & \\
C_{-g+1} & C_{-g} & & & C_{g-1} & C_{g} & & \\ 
\vdots & \ddots & \ddots & & \vdots & \ddots & \ddots & \\
C_{-g+n-1} & C_{-g+n-2} & \cdots & C_{-g} & C_{g-n+1} & C_{g-n+2} & \cdots & C_{g} \\ \hline 
C_{g} & C_{g-1} & \cdots & C_{g-n+1} & C_{-g} & C_{-g+1} & \cdots & C_{-g+n-1} \\
 & C_{g} &  \cdots & C_{g-n+2} &  & C_{-g} &  \cdots & C_{-g+n-2}\\ 
 & & \ddots & \vdots &  & & \ddots &  \vdots \\
  & & & C_{g} &   & & & C_{-g} \\
\end{array}\right].
\end{equation*}
Suppose that $\det D_n(\underline{C})\not=0$ for every $1 \leq n \leq 2g$.
Then
\begin{enumerate}
\item $\det(E^{+} \pm E^{-}J_{n}) \not=0$ for every $1 \leq n \leq 2g$; 
\item the pair of functions $(A(a,z),B(a,z))$ defined in \eqref{0330_10} below satisfies 
\begin{equation} \label{0330_1}
-a\frac{d}{da}
\begin{bmatrix}
A(a,z) \\ B(a,z)
\end{bmatrix}
= z 
\begin{bmatrix}
0 & -1 \\ 1 & 0
\end{bmatrix}
H(a)
\begin{bmatrix}
A(a,z) \\ B(a,z)
\end{bmatrix} \quad (a \in [1,q^g),\,z \in \C)
\end{equation}
together with the boundary conditions 
\begin{equation} \label{0418_4}
\aligned 
\begin{bmatrix}
A(1,z) \\ B(1,z)
\end{bmatrix}
=
\begin{bmatrix}
A(z) \\ B(z)
\end{bmatrix}, \\ \quad 
\lim_{a \nearrow  q^g}
\begin{bmatrix}
A(a,z) \\ B(a,z)
\end{bmatrix}
= 
\begin{bmatrix}
E(0) \\ 0
\end{bmatrix},
\endaligned 
\end{equation}
where $A(z)$ and $B(z)$ are functions of \eqref{def_A} and \eqref{def_B}, respectively, and 
\begin{equation} \label{def_Hq}
H(a) = H(a;\underline{C}) := 
\begin{bmatrix}
\gamma(a;\underline{C})^{-1} & 0 \\ 0 & \gamma(a;\underline{C})
\end{bmatrix}.
\end{equation}
\end{enumerate}
\end{theorem}
We mention another way of constructing $(\gamma(a),A(a,z),B(a,z))$ in Section \ref{section_7}. 

The positive-definiteness of the Hamiltonian $H(a)$ in \eqref{def_Hq} 
is characterized by the following to-be-expected way.
\begin{theorem} \label{thm_02}
Let $q>1$ and $\underline{C} \in \R^\ast \times \R^{2g-1} \times \R^\ast$, 
let $E(z)$ be the exponential polynomial defined by \eqref{0418_2}, 
and let $H(a)$ be the matrix-valued function defined by \eqref{def_Hq}. 
\begin{enumerate}
\item
Suppose that $E(z)$ belongs to the class HB. 
Then $H(a)$ is well-defined and is positive definite for every $1 \leq a <q^g$. 
Hence the quasi-canonical system attached to \eqref{0330_1} and \eqref{0418_4} 
is a canonical system, and $(A(a,z),B(a,z))/E(0)$ is its solution. 
\item
Suppose that $H(a)$ is well-defined and is positive definite for every $1 \leq a <q^g$. 
Then $E(z)$ belongs to the class HB. 
\end{enumerate}
\end{theorem}
As mentioned above, $E(z)$ of \eqref{0418_2} belongs to the class HB 
if and only if it has no zeros in $\C_+ \cup \R$. 
On the other hand, the following conditions 
are equivalent to each other by $\Delta_0=1$ and definitions \eqref{0423_7}, \eqref{0425_2}, and \eqref{def_Hq}:  
\begin{enumerate}
\item[(i)] $H(a)$ is positive definite for every $1 \leq a <q^g$,
\item[(ii)] $\gamma(a)>0$ for every $1 \leq a <q^g$,
\item[(iii)] $0<\Delta_n<\infty$ for every $1 \leq n \leq 2g$.
\end{enumerate}
Therefore, we obtain the following corollary. 
\begin{corollary} \label{cor_01} 
An exponential polynomial $E(z)$ of \eqref{0418_2} 
has no zeros in $\C_+ \cup \R$ 
if and only if $0<\Delta_n<\infty$ 
for every $1 \leq n \leq 2g$. 
\end{corollary}

The converse of Theorem \ref{thm_01} is the direct problem for quasi-canonical systems \eqref{can_0} with the Hamiltonians of the form \eqref{def_Hq}. 
It is easier than the inverse problem, because the Hamiltonians have a simple form: $H(a)={\rm diag}(\gamma(a)^{-1},\gamma(a))$, 
and $\gamma(a)$ is a (special) locally constant function. 

\begin{theorem}  \label{thm_02_1} Let $q>1$, $g \in \Z_{>0}$, and let $\gamma(a)$ be a locally constant function on $[1,q^g)$ 
such that $\gamma(a)=\gamma_n \in \R^\ast$ for $q^{(n-1)/2} \leq a < q^{n/2}$ $(1 \leq n \leq 2g)$. 
Then the quasi-canonical system \eqref{can_0} with $H(a)={\rm diag}(\gamma(a)^{-1},\gamma(a))$ on $[1,q^g)$ 
has the unique solution $(A(a,z),B(a,z))$ whose components have the forms 
\begin{equation} \label{0829_1}
\aligned  
A(a,z) &= \sum_{m=-g+n}^{g} \alpha_m(n)\left[\left(\frac{q^{m}}{a}\right)^{iz} + \left(\frac{q^{m}}{a}\right)^{-iz} \right], \\
-iB(a,z) &= \sum_{m=-g+n}^{g} \beta_m(n)\left[\left(\frac{q^{m}}{a}\right)^{iz} - \left(\frac{q^{m}}{a}\right)^{-iz} \right] \\
\endaligned
\end{equation}
for $q^{(n-1)/2} \leq a < q^{n/2}$ and $1 \leq n \leq 2g$, where 
$\alpha_{m}^{\pm}(n)$ 
and 
$\beta_{m}^{\pm}(n)$
are real constants depending only on the set $\{\gamma_n\}$ of values of $\gamma(a)$. 
Therefore, for $q^{(n-1)/2} \leq a < q^{n/2}$, $E(a,z):=A(a,z)-iB(a,z)$ is the exponential polynomial
\[
E(a,z) = 
\sum_{m=-g+n}^{g} \left[ (\alpha_m(n)+\beta_m(n)) \left(\frac{q^{m}}{a}\right)^{iz} + (\alpha_m(n)-\beta_m(n)) \left(\frac{q^{m}}{a}\right)^{-iz} \right]. 
\]
Moreover, $E(a,0)=1$ and $E(a,z)$ has no real zeros for any fixed $1 \leq a \leq q^g$. 

In particular, each Hamiltonian of the form \eqref{def_Hq} with \eqref{0425_2} yields an exponential polynomial 
$E(1,z)=A(1,z)-iB(1,z)$ having no real zeros. 
\end{theorem}

Theorem \ref{thm_02_1} does not guarantee that $E(1,z)$ has the form \eqref{0418_2}. 
In fact, $H(a)$ on $[0,q)$ with $\gamma(a)=1$ for $1 \leq a < q^{1/2}$ and $\gamma(a)=-1$ for $q^{1/2} \leq a < q$ 
yields the constant function $E(1,z)=1$, 
and $H(a)$ on $[0,q)$ with $\gamma(a)=1$ for $1 \leq a < q$  
yields the function $E(1,z)=q^{-iz}$. In these examples, $E(1,z)$ does not have the form \eqref{0418_2}. 
A sufficient condition for $E(1,z)$ to have the form \eqref{0418_2} 
is the following.  

\begin{theorem} \label{thm_02_2}
With the notation of Theorem \ref{thm_02_1}, 
$E(1,z)$ is an exponential polynomial of the form \eqref{0418_2} if $\gamma_n>0$ for every $1 \leq n \leq 2g$ and $\gamma_1 \not=1$.  
\end{theorem}

If the hypothesis of Theorem \ref{thm_02_2} is satisfied, one proves that $E(1,z)$ belongs to the class HB 
in a way similar to the proof of Theorem \ref{thm_02} (2) in Section \ref{section_5_2}. 
On the other hand, $\gamma_n>0$ for every $1 \leq n \leq 2g$, and $\gamma_1 \not=1$ 
if we start from the exponential polynomial of the form \eqref{0418_2},  
by Theorem \ref{thm_02} and \eqref{0503_3} below. 
Hence, as a consequence of the above theorems, the exponential polynomials \eqref{0418_2} belonging to the class HB 
are characterized in terms of positive-definiteness of Hamiltonians. 
\medskip

Now we turn to an application of Corollary \ref{cor_01}.
A nonzero polynomial $P(x)=c_0 x^n + c_1 x^{n-1} + \cdots + c_{n-1} x + c_n $ 
with {\it real coefficients} 
is called a {\bf self-reciprocal polynomial} of degree $n$ 
if $c_0\not=0$ and $P$ satisfies the self-reciprocal condition $P(x)=x^n P(1/x)$; 
equivalently, $c_0\not=0$ and $c_{k}=c_{n-k}$ for every $0\leq k \leq n$. 
The roots of a self-reciprocal polynomial either lie on the unit circle $T=\{z \in \C\,:\, |z|=1\}$  
or are distributed symmetrically with respect to $T$. 
Therefore, a basic problem 
is to find a ``nice'' condition on coefficients of a self-reciprocal polynomial under which all roots of $P$ lie on $T$. 
Quite a few results for this problem can be found in the literature; see, e.g., 
books of Marden~\cite{Marden66}, 
Milovanovi\'c--Mitrinovi\'c--Rassias \cite{MMR94}, 
Takagi~\cite[Section 10]{Takagi} and the survey paper of Milovanovi\'c--Rassias~\cite{MiloRass00}
for several systematic treatments of roots of polynomials. 

As an application of Corollary \ref{cor_01}, we study roots 
self-reciprocal polynomials of {\it even degree}. 
The restriction on the degree is not essential, because 
if $P(x)$ is a self-reciprocal polynomial of odd degree, 
there exists a self-reciprocal polynomial $\tilde{P}(x)$ of even degree 
and an integer $r \geq 1$ 
such that $P(x)=(x+1)^{r}\tilde{P}(x)$. 
In contrast, the reality of coefficients is essential.   
We denote by $P_g(x)$ a self-reciprocal polynomial of degree $2g$ of the form
\begin{equation} \label{def_Pg}
P_g(x) = \sum_{k=0}^{g-1} c_k(x^{(2g-k)} + x^{k}) + c_g x^{g}, \quad c_0 \not=0, 
\end{equation}
and identify the polynomial with the vector $\underline{c}=(c_0,c_1,\cdots,c_g) \in \R^\ast \times \R^{g}$ 
consisting of its coefficients. 
For a vector $\underline{c} \in \R^\ast \times \R^{g}$
and a real number $q>1$, 
we define the numbers $\delta_n(\underline{c})$ ($1 \leq n \leq 2g$) by 
\begin{equation} \label{0423_8}
\delta_n(\underline{c}) := 
\frac{\det(E^{+}(\underline{C})+E^{-}(\underline{C})J_{n})}
{\det(E^{+}(\underline{C})-E^{-}(\underline{C})J_{n})} 
\times 
\begin{cases}
~1 & \text{if $n$ is even}, \\
~g\log q & \text{if $n$ is odd},
\end{cases}
\end{equation}
where $\underline{C}$ is the vector defined by  
\begin{equation} \label{0418_1}
\sum_{m=-g}^{g} C_m T^{m} := \sum_{m=0}^{g} c_{g-m}(1-\log q^m) T^{m} 
+ \sum_{m=1}^{g} c_{g-m}(1+ \log q^m) T^{-m}.
\end{equation}

\begin{theorem} \label{thm_1} 
Let $g \in \Z_{>0}$, $q>1$. 
Let $\underline{c} \in \R^\ast \times \R^{g}$ be coefficients of a self-reciprocal polynomial $P_g(x)$ 
of the form \eqref{def_Pg}. 
Define $\underline{C} \in \R^\ast \times \R^{2g-1} \times \R^\ast$  by \eqref{0418_1}. 
Then, for every $1 \leq n \leq 2g$, it is independent of $q$ 
whether $\det(E^{+}\pm E^{-}J_{n})$ are zero, 
and the numbers $\delta_n(\underline{c})$ of \eqref{0423_8} are independent of $q$ 
if $\det(E^{+} \pm E^{-}J_{n})$ are not zero.  

Moreover, a necessary and sufficient condition for all roots of $P_g(x)$ 
to be simple roots on $T$ 
is that $0<\delta_n(\underline{c}) <\infty$ for every $1 \leq n \leq 2g$. 
\end{theorem}
\begin{remark}
As a function of indeterminate elements $(c_0,\cdots,c_g)$, 
$\delta_n(\underline{c})$ is a rational function of $(c_0,\cdots,c_g)$ over $\Q$. 
\end{remark}

The criterion of Theorem \ref{thm_1} may not be new,  
because it seems that the quantity $\delta_n(\underline{c})$ in Theorem \ref{thm_1} 
probably essentially coincides with quantities in a classical theory on the roots of polynomials in Section \ref{section_8_5}. 
The author has not yet considered rigorously whether $\delta_n(\underline{c})$ and the classical quantities 
indeed have the same meaning. 
However, the theory of this paper at least provides a new bridge between the study
of the roots of polynomials and the theory of canonical systems.
\smallskip

In order to deal with the case that all roots of $P_g(x)$ lie on $T$ 
but $P_g(x)$ may have a multiple zero, 
we modify the above definition of $\delta_n(\underline{c})$ as follow. 
\medskip 

For a vector $\underline{c} \in \R^\ast \times \R^{g}$ and real numbers $q>1$, $\omega>0$, 
we define the numbers $\delta_n(\underline{c}\,;q^\omega)$ ($1 \leq n \leq 2g$) by 
\begin{equation} \label{0503_9}
\delta_n(\underline{c}\,;q^\omega) 
= 
\frac{\det(E^{+}(\underline{C})+E^{-}(\underline{C})J_{n})}
{\det(E^{+}(\underline{C})-E^{-}(\underline{C})J_{n})} 
\times 
\begin{cases}
~1 & \text{if $n$ is even}, \\
~\displaystyle{\frac{q^{g\omega}-q^{-g\omega}}{q^{g\omega}+q^{-g\omega}}}& \text{if $n$ is odd}, 
\end{cases}
\end{equation}
where $\underline{C}$ is the vector defined by 
\begin{equation} \label{0503_8}
\sum_{m=-g}^{g} C_m T^{m} 
:= \sum_{m=0}^{g} c_{g-m}q^{-m \omega} T^{m} 
+ \sum_{m=1}^{g} c_{g-m}q^{m \omega} T^{-m}.
\end{equation}
\begin{theorem} \label{thm_2} 
Let $g \in \Z_{>0}$, $q>1$. 
Let $\underline{c} \in \R^\ast \times \R^{g}$ be coefficients of a self-reciprocal polynomial $P_g(x)$ 
of the form \eqref{def_Pg}. 
Then a necessary and sufficient condition for all roots of $P_g(x)$ to lie on $T$ is that 
$0<\delta_n(\underline{c}\,;q^\omega)<\infty$ for every $1 \leq n \leq 2g$ and $\omega>0$. 
\end{theorem}
Furthermore, the quantities $\delta_n(\underline{c})$ and $\delta_n(\underline{c}\,;q^\omega)$ 
are related as follows. 
\begin{theorem} \label{thm_5}
Let $\delta_n(\underline{c})$ and $\delta_n(\underline{c}\,;q^\omega)$ be as above. Then  
\[
\lim_{q^\omega \searrow  1} \delta_n(\underline{c}\,;q^\omega)
= \delta_{n}(\underline{c})
\]
as a rational function of $\underline{c}=(c_0,\cdots,c_g)$ over $\Q$.  
Suppose that all roots of a self-reciprocal polynomial \eqref{def_Pg} lie in $T$ and are simple. Then 
\[
\delta_n(\underline{c}\,;q^\omega)
= \delta_{n}(\underline{c})+O(\log q^\omega) \quad \text{as} \quad q^\omega \searrow  1, 
\]
where the implied constants depend only on  $\underline{c}$. 
\end{theorem}

Finally, we comment on the reason we use a multiplicative variable on $\R_{>0}$ and the right endpoint for the initial value 
in \eqref{can_0}.  
It comes from the author's personal motivation for the work of this paper: 
the inverse problem for entire functions of the class HB obtained by 
Mellin transforms of functions on $[1,\infty)$. 
This problem was stimulated by Burnol \cite{Burnol2011}, in particular, Sections 5--8, 
and was partially treated in \cite{Su2012_1} in the number theoretic setting.   
In \cite{Burnol2011}, Burnol often uses a multiplicative variable on $\R_{>0}$, and the left endpoint ($=+\infty$) corresponds to the initial value. 
In order to apply the method of \cite{Burnol2011} to the above inverse problem, the author noted on the formulas 
\[
\int_{1}^{\infty} f(x) x^{iz} \frac{dx}{x} = \lim_{X \to \infty}\int_{1}^{X} f(x) x^{iz} \frac{dx}{x} 
\]
and
\[
\int_{1}^{X} f(x) x^{iz} \frac{dx}{x} = \lim_{q \to 1+}\frac{1}{\log q} \sum_{n=0}^{\lfloor \log X/\log q \rfloor} f(q^n) q^{inz}. 
\]
Here exponential polynomials appear in the left-hand side of the second formula. 
If we fix $X>1$, $q>1$ and put $2g=\lfloor \log X/\log q \rfloor$, 
we obtain an exponential polynomial of degree $2g$. 
Recall that the results of the paper corresponds the exponential polynomial to the canonical system on $[1,q^g)$. 
Then, if we imagine the limiting situation $q \to 1$ and $X \to +\infty$, 
it is expected (but not proved rigorously) that a canonical system corresponds to 
the Mellin transform $\int_{1}^{\infty} f(x) x^{iz} \frac{dx}{x}$ 
should be a system on $[1,\infty)$. 

For a realizations of the above heuristic discussion on the original motivation, 
the right endpoint may be more useful and convenient for the initial value than the left endpoint, 
although the left endpoint is useful for the initial value in the usual theory of canonical systems. 
In any case, the above number theoretic aspect of de Branges’ theory is an important aspect. 
\medskip

The paper is organized as follows. 
In Section \ref{section_2}, we describe the outline of the proof of Theorem \ref{thm_01}. 
In Section \ref{section_3}, we prepare several lemmas and the notation 
used to prove statements of Section \ref{section_2} and Theorem \ref{thm_01} in Section \ref{section_4}. 
In Section \ref{section_5}, we prove Theorems \ref{thm_02}, \ref{thm_02_1}, and \ref{thm_02_2}.  
In Section \ref{section_7}, we mention a way of constructing $(\gamma(a),A(a,z),B(a,z))$ 
which is different from the way of Sections \ref{section_2} and \ref{section_4}. 
In Section \ref{section_8}, we prove Theorems \ref{thm_1}, \ref{thm_2}, and \ref{thm_5} 
and compare them with classical results 
on the roots of self-reciprocal polynomials. 

\medskip

\noindent
{\bf Acknowledgments}~
The author thanks Shigeki Akiyama for suggesting the book of Takagi (\cite[Section 10]{Takagi}). 
The author also thanks the referee 
for a number of helpful suggestions for improvement in the article and for careful reading. 
In particular, Theorems \ref{thm_02_1} and \ref{thm_02_2} were added by the suggestion of the referee. 
This work was supported by KAKENHI (Grant-in-Aid for Young Scientists (B)) No. 21740004 
and No. 25800007. 
Also, this work was partially supported 
by French-Japanese Projects ``Zeta Functions of Several Variables and  Applications''
in Japan-France Research Cooperative Program supported by JSPS and CNRS. 

\section{Outline of the proof of Theorem \ref{thm_01}} \label{section_2}

\subsection{Hilbert spaces and operators.} 
Let $q>1$, 
and let $L^2(\mathbb{T}_q)$ be the completion of the space of $(2\pi/\log q)$-periodic continuous functions on $\R$ 
with respect to the $L^2$-norm 
$\Vert f \Vert_{L^2(\mathbb{T}_q)}^2 := \langle f, f \rangle_{L^2(\mathbb{T}_q)}$, 
where 
\[
\langle f, g \rangle_{L^2(\mathbb{T}_q)} := \frac{\log q}{2\pi}\int_{I_q} f(z)\overline{g(z)}\, dz 
\quad \text{and} \quad  I_q=[0,2\pi/(\log q)).
\] 
Every $f \in L^2(\mathbb{T}_q)$ has the Fourier expansion 
\[
f(z) = \sum_{k \in \Z}u(k)q^{ikz}
\] 
with $\{u(k)\}_{k \in \Z} \in l^2(\Z)$ 
and $\Vert f \Vert_{L^2(\mathbb{T}_q)}^2 = \sum_{k \in \Z}|u(k)|^2$,  
where $l^2(\Z)$ is the Hilbert space of sequences $\{u(k) \in \C \,:\, k \in \Z\}$ satisfying $\sum_{k \in \Z}|u(k)|^2 < \infty$ (\cite[Chapter 4]{Rudin}). 

For $0 < a \not\in q^{\Z/2}=\{q^{n/2}\,:\, n \in \Z\}$, we define the vector space 
\[
V_a : = \left\{\left. \phi(z)= a^{-iz}f(z) + a^{iz}g(z)
\,\right|~f,\,g \in L^2(\mathbb{T}_q) \right\}
\]
of functions of $z \in \R$. 
As a vector space, $V_a$ is isomorphic to the direct sum $L^2(\mathbb{T}_q) \oplus L^2(\mathbb{T}_q)$, 
since $a^{-iz}f(z) + a^{iz}g(z) =0$ if and only if $(f,g) =(0,0)$. 
In fact, if one of $f$ and $g$ is zero and $a^{-iz}f(z) + a^{iz}g(z) =0$, the other one is also zero. 
On the other hand, if $f\not=0$, $g\not=0$, and $a^{-iz}f(z) + a^{iz}g(z) =0$, 
then $a^{2iz}=-f(z)/g(z)$, and hence $a \in q^{\Z/2}$. 
The maps $p_1:(a^{-iz}f(z) + a^{iz}g(z)) \mapsto a^{-iz}f(z)$ 
and $p_2: (a^{-iz}f(z) + a^{iz}g(z)) \mapsto a^{iz}g(z)$ 
are projections from $V_a$ to the first and the second components 
of the direct sum, respectively.
We define the inner product on $V_a$ by 
\[
\langle \phi_1, \phi_2 \rangle 
= \langle f_1, g_1 \rangle_{L^2(\mathbb{T}_q)} + \langle g_1, g_2 \rangle_{L^2(\mathbb{T}_q)},
\]
where $\phi_j(z) = a^{-iz}f_j(z) + a^{iz}g_j(z)$ ($j=1,2$). 
Then $V_a$, with this inner product is a Hilbert space and is isomorphic to 
the (orthogonal) direct sum $L^2(\mathbb{T}_q) \oplus L^2(\mathbb{T}_q)$ of Hilbert spaces (\cite[Chapter 5]{Lax}). 
We put 
\begin{equation} \label{0313_1}
X(k):=(q^k/a)^{iz}, \quad Y(l):=X(l)^{-1}=(q^l/a)^{-iz}
\end{equation}
for $k,l \in \Z$ and $a>0$. 
We regard $X(k)$ and $Y(l)$ as functions of $z$, functions of $(a,z)$, or symbols, depending on the situation. 
For a fixed $0<a \not\in q^{\Z/2}$, the countable set consisting of all $X(k)$ and $Y(l)$ is linearly independent over $\C$ as a set of functions of $z$, 
since the linear dependence of $\{X(k),\,Y(l)\}_{k,l \in \Z}$ implies the existence of a nontrivial pair of functions 
$f,g \in L^2(\mathbb{T}_q)$ satisfying  $a^{-iz}f(z) + a^{iz}g(z) =0$. 
Using these vectors, we have 
\[
\aligned 
V_a 
& = \left\{ \phi=\sum_{k \in \Z}u(k)X(k) + \sum_{l \in \Z}v(l)Y(l)
\,:~ \{u(k)\}_{k \in \Z}, \, \{v(l)\}_{l \in \Z} \in l^2(\Z) \right\} \\
& \simeq  L^2(\mathbb{T}_q) \oplus L^2(\mathbb{T}_q),
\endaligned 
\]
\[
\aligned 
p_1V_a & = \left\{ \sum_{k \in \Z}u(k)X(k) 
\,:~ \{u(k)\}_{k \in \Z} \in l^2(\Z) \right\} \simeq  L^2(\mathbb{T}_q), \\
p_2V_a & = \left\{ \sum_{l \in \Z}v(l)Y(l)
\,:~ \{v(l)\}_{l \in \Z} \in l^2(\Z) \right\}\simeq  L^2(\mathbb{T}_q) 
\endaligned 
\]
and 
\[
\langle \phi, \phi' \rangle 
= 
\sum_{k \in \Z}u(k)\overline{u'(k)} + \sum_{l \in \Z} v(l) \overline{v'(l)} \quad(\phi,\phi' \in V_a), 
\]
\begin{equation} \label{0421_1}
\Vert \phi \Vert^2 = \langle \phi, \phi \rangle =  \sum_{k \in \Z}|u(k)|^2 + \sum_{l \in \Z}|v(l)|^2 \quad(\phi \in V_a), 
\end{equation}
where 
\[
\phi=\sum_{k \in \Z}u(k)X(k) + \sum_{l \in \Z}v(l)Y(l)\quad  \text{and} \quad \phi'=\sum_{k \in \Z}u'(k)X(k) + \sum_{l \in \Z}v'(l)Y(l).
\] 
On the other hand, we have 
\begin{equation} \label{0421_1_1}
\Vert \phi \Vert^2 
= \frac{\log q}{2\pi}\int_{I_q} p_1\phi(z) \overline{p_1 \phi(z)} \, dz +  \frac{\log q}{2\pi}\int_{I_q} p_2\phi(z) \overline{p_2 \phi(z)} \, dz
\end{equation}
for $\phi \in V_a$, since 
\[
\frac{\log q}{2\pi}\int_{I_q}
(a^{\pm iz}f_1(z)) 
\overline{(a^{\pm iz}f_2(z))} \, dz
= \sum_{k \in \Z}u_1(k)\overline{u_2(k)} = \langle f_1, f_2  \rangle_{L^2(\mathbb{T}_q)}
\]
for $f_j(z)=\sum_{k \in \Z}u_j(k)q^{ikz} \in L^2(\mathbb{T}_q)$ ($j=1,2$). 
Note that, for $\phi \in V_a$, $p_1\phi$ and $p_2 \phi$ are not periodic functions of $z$, 
but the integrals $\int_I p_j\phi(z) \overline{p_j \phi'(z)} \, dz$ ($j=1,2$, $\phi, \phi' \in V_a$)
are independent of the intervals $I=[\alpha,\alpha+2\pi/\log q]$ ($\alpha \in \R$). 
We write $\phi \in V_a$ as $\phi(z)$ (respectively, $\phi(a,z)$) to emphasize that $\phi$ is a function of $z$ (respectively, $(a,z)$). 
If we regard $X(k)$ and $Y(l)$ as symbols, $V_a$, 
endowed with the norm defined by \eqref{0421_1},
is an abstract Hilbert space and is isomorphic to $l^2(\Z) \oplus l^2(\Z)$.  
\medskip

For each nonnegative integer $n$, we define the closed subspace $V_{a,n}$ of $V_{a}$ by 
\[
V_{a,n} = \left\{ \phi_n=\sum_{k=0}^{\infty} u_n(k)X(k) + \sum_{k=-\infty}^{n-1} v_n(l)Y(l)
\,:~ \{u_n(k)\}_{k=0}^{\infty}, \, \{v_n(l)\}_{l=-\infty}^{n-1} \in l^2(\Z) \right\}
\]  
and denote by ${\mathsf P}_n^\ast$ the projection operator $V_a \to V_{a,n}$. 
Then, for the conjugate operator 
${\mathsf P}_n^{\ast\ast}:={\mathsf J}{\mathsf P}_n^\ast{\mathsf J}:V_a \to V_a$ 
of ${\mathsf P}_n^\ast$ by the involution ${\mathsf J}: V_a \mapsto V_a$ 
defined by 
\[
{\mathsf J} \phi = \sum_{k=-\infty}^{\infty} v(k)X(k) + \sum_{l=-\infty}^{\infty} u(l)Y(l) \quad (\phi \in V_a), 
\]
we have 
\[
{\mathsf P}_n^{\ast\ast} \phi ={\mathsf J}{\mathsf P}_n^\ast{\mathsf J} \phi= \sum_{k=-\infty}^{n-1} u(k)X(k) + \sum_{l=0}^{\infty} v(l)Y(l)\quad (\phi \in V_a).
\]
%
%
Therefore, $\mathsf{P}_n:=\mathsf{P}_n^{\ast\ast}\mathsf{P}_n^\ast$ maps $V_{a}$ into $V_{a,n}$  
for every nonnegative integer $n$. 
In addition, ${\mathsf J}{\mathsf P}_n$ also maps $V_{a,n}$ into $V_{a,n}$ 
for every nonnegative $n$,  
because 
\begin{equation} \label{0421_2}
{\mathsf J}{\mathsf P}_n \phi_n = \sum_{l=0}^{n-1} v_n(l)X(l) + \sum_{k=0}^{n-1} u_n(k)Y(k)
\end{equation}
for
$\phi_n \in V_{a,n}$. 
Note that ${\mathsf P}_0|_{V_{a,0}}={\mathsf J}{\mathsf P}_0|_{V_{a,0}}$, 
since ${\mathsf P}_0^\ast{\mathsf J}|_{V_{a,0}}=0$, by definition. 
\medskip

For an exponential polynomial $E(z)$ of \eqref{0418_2}, we have
\[
E^\sharp(z)=\sum_{m=-g}^{g}C_{-m} q^{imz}=E_q(z;\underline{C}^\sharp)
\] 
for $\underline{C}^\sharp :=(C_{-g},C_{-(g-1)},\dots,C_{g}) \in \R^\ast \times \R^{2g-1} \times \R^\ast$, 
by a simple calculation. 
Using $E(z)$ and $E^\sharp(z)$, we define the multiplication operators
\begin{equation}\label{0611_1}
{\mathsf E}: \phi(z) \mapsto E(z) \phi(z), \quad 
{\mathsf E}^\sharp:  \phi(z) \mapsto E^\sharp(z) \phi(z).
\end{equation}
These operators map $V_a$ into $V_a$, 
because ${\mathsf E}$ and ${\mathsf E}^\sharp$ are expressed as
\[
{\mathsf E} = \sum_{m = -g}^{g} C_m {\mathsf T}_m, \qquad
{\mathsf E}^{\sharp}  = \sum_{m = -g}^{g} C_{-m} {\mathsf T}_m
\]
by using shift operators ${\mathsf T}_m:V_a \to V_a$ ($m \in \Z$) defined by 
\[
{\mathsf T}_m v=\sum_{k=-\infty}^{\infty} u(k)X(k+m) + \sum_{l=-\infty}^{\infty} v(l)Y(l-m). 
\]
Both ${\mathsf E}$ and ${\mathsf E}^\sharp$ are bounded on $V_a$, 
since 
$\Vert {\mathsf E} \Vert \leq \sum_{m=-g}^{g}|C_m|\cdot \Vert \mathsf{T}_m \Vert \leq M$ 
and 
$\Vert {\mathsf E}^\sharp \Vert \leq \sum_{m=-g}^{g}|C_m|\cdot \Vert \mathsf{T}_m \Vert \leq M$ 
for $M=\max\{C_m\,|\, -g \leq m \leq g\}$. 

Suppose that $E(z)$ has no zeros on the real line. 
Then the operator ${\mathsf E}$ is invertible on $V_a$ (Lemma \ref{0330_6}\,(1)). 
Thus the operator 
\begin{equation} \label{0418_3}
\Theta := {\mathsf E}^{-1}{\mathsf E}^{\sharp}
\end{equation}
is well-defined. We have  
\begin{equation} \label{0421_3}
(\Theta \phi)(z)=\Theta(z)\phi(z) \quad \text{for} \quad 
\Theta(z):=\frac{E^\sharp(z)}{E(z)},~\phi \in V_a 
\end{equation}
and $\Theta\mathsf{J}\mathsf{P}_n(W_{a,n}) \subset W_{a,n}$ 
for  each nonnegative integer $n$, where  
\[
W_{a,n} := V_{a,n} + \Theta\mathsf{J}\mathsf{P}_n V_{a,n}.
\]

\subsection{Quasi-canonical systems associated with exponential polynomials.} 
In the above settings, a quasi-canonical system associated with an exponential polynomial $E(z)$ of \eqref{0418_2} 
is constructed starting from solutions of the set of linear equations
\begin{equation} \label{LE0}
\aligned
( {\mathsf I} \pm \Theta {\mathsf J}{\mathsf P}_n )\, \phi_{n}^{\pm} 
&= \Theta X(0) \\
\endaligned
\quad (\phi_n^{\pm} \in W_{a,n},\, 0 \leq n \leq 2g),
\end{equation}
where ${\mathsf I}$ is the identity operator. 
The set of $4g+2$ equations \eqref{LE0} is a discrete analogue of 
the (right) Mellin transform of differential equations  \cite[(117a) and (117b), Section 6]{Burnol2011}; 
see also \cite[Section 4]{Su2012_1}.  

Under the assumption that both ${\mathsf I} \pm \Theta{\mathsf J}{\mathsf P}_n$ 
are invertible on $W_{a,n}$ for every $0 \leq n \leq 2g$, that is, 
$({\mathsf I} \pm \Theta{\mathsf J}{\mathsf P}_n)^{-1}$ 
exist as a bounded operator on $W_{a,n}$, 
we define 
\begin{equation} \label{0418_6}
A_n^\ast(a,z):=\frac{1}{2}({\mathsf I} + {\mathsf J}){\mathsf E} \, \phi_{n}^{+}(a,z), \quad 
B_n^\ast(a,z):=\frac{1}{2i}({\mathsf I} - {\mathsf J}){\mathsf E} \, \phi_{n}^{-}(a,z)
\end{equation}
using unique solutions of \eqref{LE0}. 
The functions $A_n^\ast(a,z)$ and $B_n^\ast(a,z)$ are entire functions of $z$ 
and are extended to functions of $a$ on $(0,\infty)$ (by formula \eqref{0330_8}). 
Here we define functions $A^\ast(a,z)$ and $B^\ast(a,z)$ 
of $(a,z) \in [1,q^g) \times \C$ by 
\begin{equation} \label{0330_9}
A^\ast(a,z) := A_n^\ast(a,z), \quad  B^\ast(a,z) := B_n^\ast(a,z) 
\end{equation}
for $q^{(n-1)/2} \leq a < q^{n/2}$. 
In general, the function $A^\ast(a,z)$ is discontinuous at $a \in [1,q^g) \cap q^{\Z/2}$, 
because 
$A_{n}^\ast(q^{n/2},z) =  A_{n+1}^\ast(q^{n/2},z)$
may not be hold. The same is true about $B^\ast(a,z)$. 
However, we will see that 
\begin{equation} \label{0330_12}
\aligned
\alpha_n & =\alpha_n(\underline{C}):=\frac{A_{n-1}^\ast(q^{(n-1)/2},z)}{A_n^\ast(q^{(n-1)/2},z)}, \\
\beta_n & =\beta_n(\underline{C}):=\frac{B_{n-1}^\ast(q^{(n-1)/2},z)}{B_n^\ast(q^{(n-1)/2},z)}
\endaligned
\end{equation}
are independent of $z$ for every $1 \leq n \leq 2g$ (Proposition \ref{prop_0313_1}). 
Therefore, 
we obtain functions $A(a,z)$ and $B(a,z)$ of $(a,z) \in [1,q^g) \times \C$ 
which are continuous for $a$ and entire for $z$ by the modification 
\begin{equation} \label{0330_11}
\aligned
A_n(a,z) :&= \alpha_1 \cdots \alpha_n \cdot A_n^\ast(a,z), \\
B_n(a,z) :&= \beta_1 \cdots \beta_n \cdot B_n^\ast(a,z)
\endaligned
\end{equation}
for $1 \leq n \leq 2g$, and 
\begin{equation} \label{0330_10}
A(a,z) := A_n(a,z), \quad B(a,z) := B_n(a,z) 
\end{equation}
for $q^{(n-1)/2} \leq a < q^{n/2}$. Moreover, $(A(a,z),B(a,z))$ satisfies the system \eqref{0330_1} 
endowed with the boundary conditions \eqref{0418_4} 
for the locally constant function $\gamma(a)$ defined by 
\begin{equation} \label{0425_3}
\gamma_n=\gamma_n(\underline{C}):=\frac{\alpha_1(\underline{C}) \cdots \alpha_n(\underline{C})}{\beta_1(\underline{C}) \cdots \beta_n(\underline{C})}
\end{equation}
and 
\begin{equation} \label{0423_6}
\gamma(a):=\gamma(a;\underline{C}):=\gamma_n
\quad \text{if} \quad q^{(n-1)/2} \leq a < q^{n/2}
\end{equation}
(Proposition \ref{0330_4} and Theorem \ref{thm_03}). 
This $\gamma(a)$ is equal to the function of \eqref{0425_2} (Proposition \ref{prop_0313_1}). 
Therefore, $(A(a,z),B(a,z))/E(0)$ is a solution of a quasi-canonical system on $[1,q^{g})$ 
if $E(0)\not=0$ and $\alpha_n, \beta_n \not=0$ for every $1 \leq n \leq 2g$. 
As a summary of the above argument, we obtain Theorem \ref{thm_01}:  
see Section \ref{section_4} for details.  

To prove Theorems \ref{thm_02}, \ref{thm_02_1}, and \ref{thm_02_2}, 
we use the theory of de Branges spaces, which are a kind of reproducing kernel Hilbert spaces consisting of entire functions. 
Roughly speaking, the positivity of $H(a)$ corresponds to the positivity of the reproducing kernels of de Branges spaces: 
see Section \ref{section_5} for details. 

\section{Preliminaries} \label{section_3}
\subsection{Identities for matrices.}

The following basic facts of linear algebra are used often in later sections. 

\begin{lemma} \label{mat_2} 
Let $A$ and $B$ be square matrices of the same size. Then
\[
\det\begin{bmatrix} A & B \\ B & A \end{bmatrix} = \det(A+B)\det(A-B).
\]
\end{lemma}

\begin{lemma} \label{mat_1}
Let $A$, $B$, $C$, $D$ be square matrices of the same size. 
If $\det A \not=0$, $\det D \not=0$, and $\det(D-CA^{-1}B)\not=0$,   
\[
\begin{bmatrix} A & B \\ C & D \end{bmatrix}^{-1}
=
\begin{bmatrix}
(A-BD^{-1}C)^{-1}& -A^{-1}B(D-CA^{-1}B)^{-1} \\ 
-D^{-1}C(A-BD^{-1}C)^{-1} & (D-CA^{-1}B)^{-1}
\end{bmatrix}.
\]
\end{lemma}

\begin{lemma}[Laplace's expansion formula] \label{mat_3}
Let $A$ be a square matrix of size $n$.
Let $\pmb{i} = (i_1, i_2, \cdots , i_k)$ (respectively, $\pmb{j} = (j_1, j_2, \cdots , j_k)$) be a list of indices of $k$ rows (respectively, columns), 
where $1 \leq k < n$ and $0 \leq i_1 < i_2 < \cdots < i_k < n$ (respectively, $1 \leq j_1 < j_2 < \cdots < j_k < n$). 
Denote by $A(\pmb{i},\pmb{j})$ the submatrix of $A$ obtained by keeping the entries in the intersection of any row and column that are in the lists. 
Denote by $A^c(\pmb{i},\pmb{j})$ the submatrix of $A$ obtained by removing the entries in the rows and columns that are in the lists. 
Laplace's formula for determinants yields
\[
\det A = \sum_{\pmb{j}}(-1)^{|\pmb{i}|+|\pmb{j}|} \det A(\pmb{i},\pmb{j}) \det A^c(\pmb{i},\pmb{j}), 
\]
where $|\pmb{i}| = i_1+i_2+\cdots+i_k$, $|\pmb{j}| = j_1+j_2+\cdots+j_k$, 
and the summation is taken over all $k$-tuples $\pmb{j} = (j_1, j_2, \dots , j_k)$
for which $1 \leq j_1 < j_2 < \cdots < j_k < n$.
\end{lemma}

For a proof of Lemma \ref{mat_3}, see, e.g., \cite[Chapter IV]{GTM23}. 

\subsection{Definition of special matrices, I}
We define several special matrices for the convenience of later arguments. 
We define the matrices $E_0^{\pm}$ to be the square matrices of size $8g$  
\begin{equation}
E_0^{\pm}=E_0^{\pm}(\underline{C}):=
\left[ 
\begin{array}{cc|cc}
 & & &\\
 \multicolumn{2}{c|}{\raisebox{1.5ex}[0pt]{$e_0^{\pm}(\underline{C}) $}}& \multicolumn{2}{c}{\raisebox{1.5ex}[0pt]{\BigZero}}\\ \hline
 &  & & \\
\multicolumn{2}{c|}{\raisebox{1.5ex}[0pt]{\BigZero}} & \multicolumn{2}{c}{\raisebox{1.5ex}[0pt]{$e_0^{\pm}(\underline{C}) $}}  \\
\end{array}
\right],
\end{equation}
where $e_0^{\pm}(\underline{C}) $ are lower triangular matrices of size $4g$ defined by
\begin{equation}
\scalebox{0.9}{$
e_0^{\pm}=e_0^{\pm}(\underline{C}) 
:= 
\left[
\begin{array}{llllllll}
C_{\mp g} & & & & & & & \\
C_{\mp (g-1)} & \ddots & & & & & & \\
\vdots & \ddots &  C_{\mp g} & & & & & \\
C_{\pm (g-1)} & \ddots & C_{\mp (g-1)} & C_{\mp g} & & & & \\
C_{\pm g}   & \ddots & \vdots & C_{\mp (g-1)} & C_{\mp g} & & & \\
0 & \ddots & C_{\pm (g-1)}  & \vdots & \ddots & \ddots & & \\
\vdots & \ddots & C_{\pm g}  & C_{\pm (g-1)} & \ddots  & C_{\mp (g-1)} & C_{\mp g} & \\
0 & \cdots & 0 & C_{\pm g}  & C_{\pm (g-1)} & \cdots  & C_{\mp (g-1)} & C_{\mp g} \\
\end{array}
\right].
$}
\end{equation}
Replacing the column at the left edge of $e_0^{-}$ by the zero column vector,  we have 
\begin{equation}
\scalebox{0.9}{$
e_1^{-}=e_1^{-}(\underline{C}) 
:= \left[
\begin{array}{c|llllllll}
0 &0 &  \cdots & & & & & & \\
0 &C_{g} & & & & & & & \\
\vdots &C_{(g-1)} & \ddots & & & & & & \\
&\vdots & \ddots &  C_{g} & & & & & \\
&C_{-(g-1)} & \ddots & C_{(g-1)} & C_{g} & & & & \\
&C_{-g}   & \ddots & \vdots & C_{(g-1)} & C_{g} & & & \\
\vdots &0 & \ddots & C_{-(g-1)}  & \vdots & \ddots & \ddots & & \\
0 &\vdots & \ddots & C_{-g}  & C_{-(g-1)} & \ddots  & C_{(g-1)} & C_{g} & \\
0 &0 & \cdots & 0 & C_{-g}  & C_{-(g-1)} & \cdots  & C_{(g-1)} & C_{g} \\
\end{array}\right].
$}
\end{equation}
We define 
\[
E_0^\sharp J = E_0^\sharp J(\underline{C}):=0, \quad  
E_{0,1}^\sharp J=E_{0,1}^\sharp J(\underline{C}):=0
\] 
and 
\begin{equation}
\scalebox{0.9}{$
E_n^\sharp J= E_n^\sharp J(\underline{C}):=
\left[ 
\begin{array}{cc|cc}
 & & &\\
 \multicolumn{2}{c|}{\raisebox{1.5ex}[0pt]{\BigZero}} & \multicolumn{2}{c}{\raisebox{1.5ex}[0pt]{$e_{1,n}^{\sharp}(\underline{C})$}}\\ \hline
 &  & & \\
\multicolumn{2}{c|}{\raisebox{1.5ex}[0pt]{$e_{2,n}^{\sharp}(\underline{C})$}} & \multicolumn{2}{c}{\raisebox{1.5ex}[0pt]{\BigZero}}  \\
\end{array}
\right], \quad 
E_{n,1}^\sharp J=E_{n,1}^\sharp J(\underline{C})
:= \left[ 
\begin{array}{cc|cc}
 & & &\\
 \multicolumn{2}{c|}{\raisebox{1.5ex}[0pt]{\BigZero}} & \multicolumn{2}{c}{\raisebox{1.5ex}[0pt]{$e_{3,n}^{\sharp}(\underline{C})$}}\\ \hline
 &  & & \\
\multicolumn{2}{c|}{\raisebox{1.5ex}[0pt]{$e_{2,n}^{\sharp}(\underline{C})$}} & \multicolumn{2}{c}{\raisebox{1.5ex}[0pt]{\BigZero}}  \\
\end{array}
\right]
$}
\end{equation}
by setting  
\begin{equation}
\scalebox{0.9}{$
\aligned 
e_{1,n}^\sharp
& = e_{1,n}^\sharp(\underline{C}) 
:= 
\left[
\begin{array}{cc|ccc|ccc}
 & &  & & C_{g}& & \\
 & &  & \iddots & C_{(g-1)} & & \\
 & & C_{g} & \iddots & \vdots& & \\
 \multicolumn{2}{c|}{\raisebox{1.5ex}[0pt]{\BigZero}}  & C_{(g-1)} & \iddots & C_{-(g-1)} &  \multicolumn{2}{c}{\raisebox{1.5ex}[0pt]{\BigZero}} \\
 & & \vdots & \iddots & C_{-g} & & \\
 & & C_{-(g-1)} & \iddots &  & & \\
 & & C_{-g} & &  &  & \\ \hline
 & & & & &  & \\
 \multicolumn{2}{c|}{\raisebox{1.5ex}[0pt]{\BigZero}} & \multicolumn{3}{c|}{\raisebox{1.5ex}[0pt]{\BigZero}} 
& \multicolumn{2}{c}{\raisebox{1.5ex}[0pt]{\BigZero}} \\
\end{array}
\right]=
\begin{array}{c|c|c|c}
 2g-n & n & 2g & \\ \hline
 & & & 2g+n  \\ \hline
 & & & 2g-n 
\end{array}, \\
e_{2,n}^\sharp 
&= e_{2,n}^\sharp(\underline{C}) 
:= 
\left[
\begin{array}{ccc|cc}
 & & & & \\
 \multicolumn{3}{c|}{\raisebox{1.5ex}[0pt]{\BigZero}}  & \multicolumn{2}{c}{\raisebox{1.5ex}[0pt]{\BigZero}}  \\ \hline
 &  & C_{g}  & &  \\
  & \iddots & C_{(g-1)} & & \\
 C_{g} & \iddots & \vdots & & \\
 C_{(g-1)} & \iddots & C_{-(g-1)}
& \multicolumn{2}{c}{\raisebox{1.5ex}[0pt]{\BigZero}}  \\
 \vdots & \iddots & C_{-g} & & \\
 C_{-(g-1)} & \iddots & & & \\
 C_{-g}& & & & \\
\end{array}
\right]=
\begin{array}{c|c|c}
n & 4g-n & \\ \hline
  & & 2g-n  \\ \hline
  & & 2g+n 
\end{array}
\endaligned 
$}
\end{equation}
and
\begin{equation}
\scalebox{0.9}{$
e_{3,n}^\sharp=e_{3,n}^\sharp(\underline{C}) 
:= 
\left[
\begin{array}{cc|cccc|ccc}
 & & 0 & \cdots & 0 & 0 & \\
 & & \vdots & \iddots & C_{g}& 0 & \\
 & & 0 & \iddots & C_{(g-1)} & 0 & \\
 & & C_{g} & \iddots & \vdots & \vdots  & & \\
 \multicolumn{2}{c|}{\raisebox{1.5ex}[0pt]{\BigZero}}  & C_{(g-1)} & \iddots & C_{-(g-1)} & 0 & \multicolumn{2}{c}{\raisebox{1.5ex}[0pt]{\BigZero}} \\
 & & \vdots & \iddots & C_{-g} & 0 & & \\
 & & C_{-(g-1)} & \iddots & \iddots & \vdots & & \\
 & & C_{-g} & 0 & \cdots & 0 & & \\ \hline
 & & & & & &  & \\
 \multicolumn{2}{c|}{\raisebox{1.5ex}[0pt]{\BigZero}} & & \multicolumn{3}{c|}{\raisebox{1.5ex}[0pt]{\BigZero}} 
& \multicolumn{2}{c}{\raisebox{1.5ex}[0pt]{\BigZero}} \\
\end{array}
\right]=
\begin{array}{c|c|c|c}
 2g-n & n & 2g & \\ \hline
 & & & 2g+n  \\ \hline
 & & & 2g-n 
\end{array}
$}
\end{equation}
for $1 \leq n \leq 2g$, 
where the right-hand sides mean the size of each block of matrices in middle terms. 
In addition, we define square matrices $e_{0}^{\sharp}$ of size $4g$ by 
\[
\scalebox{0.9}{$
e_{0}^{\sharp}=e_{0}^{\sharp}(\underline{C}) 
:= 
\begin{bmatrix}
 & & & C_{g} & C_{g-1} & \cdots & C_{-(g-1)} & C_{-g} \\
 & & \iddots & C_{g-1} & \vdots & \iddots & C_{-g} & \\
 & C_{g} & \iddots & \vdots & C_{-(g-1)} & \iddots &  &  \\
C_{g} & C_{g-1} & \iddots & C_{-(g-1)} &  C_{-g} & & & \\
C_{g-1}   & \vdots & \iddots & C_{-g} & & & & \\
\vdots & C_{-(g-1)} & \iddots  & & & & & \\
C_{-(g-1)} & C_{-g} & & & & & & \\
C_{-g} & & & & & & & \\
\end{bmatrix},
$}
\]
so that
\[
J^{(8g)} E_0^+ =
\left[ 
\begin{array}{cc|cc}
 & & &\\
 \multicolumn{2}{c|}{\raisebox{1.5ex}[0pt]{\BigZero}} & \multicolumn{2}{c}{\raisebox{1.5ex}[0pt]{$e_{0}^{\sharp}$}}\\ \hline
 &  & & \\
\multicolumn{2}{c|}{\raisebox{1.5ex}[0pt]{$e_{0}^{\sharp}$}} & \multicolumn{2}{c}{\raisebox{1.5ex}[0pt]{\BigZero}}  \\
\end{array}
\right].
\]
Replacing $n$ columns from the left edge of $e_0^{\sharp}$ by zero column vectors,  we have 
\[
\scalebox{0.9}{$
\aligned
e_{0,n}^{\sharp}
=e_{0,n}^{\sharp}(\underline{C}) 
:&= 
\left[
\begin{array}{cc|ccccccc}
 & & & & C_{g} & C_{g-1} & \cdots & C_{-(g-1)} & C_{-g} \\
 & & & \iddots & C_{g-1} & \vdots & \iddots & C_{-g} & \\
 & & C_{g} & \iddots & \vdots & C_{-(g-1)} & \iddots &  &  \\
\multicolumn{2}{c|}{\raisebox{1.5ex}[0pt]{\BigZero}} & C_{g-1} & \iddots & C_{-(g-1)} &  C_{-g} & & & \\
& & \vdots & \iddots & C_{-g} & & & & \\
& & C_{-(g-1)} & \iddots  & & & & & \\
& & C_{-g} & & & & & & \\
& & & & & & & & \\
\end{array}
\right] \\
& =
\begin{array}{c|c|c}
n & 4g-n & \\ \hline 
     &      & 4g \\
\end{array}.
\endaligned
$}
\]
We denote by $I^{(m)}$ the identity matrix of size $m$ and by $J_n^{(m)}$ the following square matrix of size $m$:
 \[
J_n^{(m)}=
\left[ 
\begin{array}{cc|cc}
 & & &\\
 \multicolumn{2}{c|}{\raisebox{1.5ex}[0pt]{$J^{(n)}$}}& \multicolumn{2}{c}{\raisebox{1.5ex}[0pt]{\BigZero}}\\ \hline
 &  & & \\
\multicolumn{2}{c|}{\raisebox{1.5ex}[0pt]{\BigZero}} & \multicolumn{2}{c}{\raisebox{1.5ex}[0pt]{\BigZero}}  \\
\end{array}
\right], \quad \left( 
J^{(n)} = 
\begin{bmatrix}
 & & 1 \\
 & \iddots & \\
1 & & 
\end{bmatrix} \right).
\]
We also use the notation 
\[
\chi_{n} ={}^{t} \begin{bmatrix} 1 & 0 & \cdots & 0\end{bmatrix}
= \text{the unit column vector of length $n$}. 
\]

For a square matrix $M$ of size $N$, we denote by $[M]_{\nwarrow n}$ (respectively, $[M]_{\nearrow n}$) the square matrix of size $n \,( \leq N)$
in the top left (respectively, the top right) corner, 
and denote by $[M]_{\searrow n}$ (respectively, $[M]_{\swarrow n}$) the square matrix of size $n \,( \leq N)$ in the lower right (respectively, the lower left) corner: 
\[
M = \begin{bmatrix} [M]_{\nwarrow  n} & \ast \\ \ast & [M]_{\searrow N-n} \end{bmatrix}
=  \begin{bmatrix} \ast & [M]_{\nearrow n} \\ [M]_{\swarrow  N-n} & \ast \end{bmatrix}.
\] 

\subsection{Definition of special matrices, II} 
For every nonnegative integer $k$, 
we define the square matrix $P_k(m_k)$ of size $2k+2$, 
parametrized by $m_k$, 
and the $(2k+2)\times(2k+4)$ matrix $Q_k$  for every nonnegative integer $k$ as follows. 
For $k=0,1$, 
\[
P_0 := 
\left[
\begin{array}{c|c}
1 & 0 \\ \hline 0 & 1 
\end{array}\right],
\quad 
Q_0 := 
\left[
\begin{array}{cc|cc}
1 & 1 & 0 & 0 \\ \hline 0 & 0 & 1 & -1
\end{array}\right],
\]
\[
P_1(m_1) := \left[
\begin{array}{cc|cc}
1 & 0 & 0 & 0 \\ 
0 & 1 & 0 & 0 \\ \hline
0 & 0 & 1 & 0 \\ \hline
0 & 1 & 0 & -m_1 
\end{array}\right], 
\quad 
Q_1 := \left[
\begin{array}{ccc|ccc}
1 & 0 & 1 & 0 & 0 & 0 \\ 
0 & 1 & 0 & 0 & 0 & 0 \\ \hline 
0 & 0 & 0 & 1 & 0 & -1 \\ \hline 
0 & 0 & 0 & 0 & 0 & 0 
\end{array}\right].
\]
For $k \geq 2$,  we define $P_k(m_k)$ and $Q_k$ blockwisely as follows: 
\[
P_k(m_k) := 
\left[\begin{array}{c|c} 
V_k^+ & \pmb{0} \\[2pt] \hline
\pmb{0} & V_k^- \\[2pt] \hline
\pmb{0}I_k & -m_k\cdot \pmb{0}I_k
\end{array}\right], 
\qquad 
Q_k 
:= 
\left[\begin{array}{c|c} 
W_k^+ & \pmb{0} \\[2pt] \hline
\pmb{0} & W_k^- \\[2pt] \hline 
\pmb{0}_{k,k+2} & \pmb{0}_{k,k+2}
\end{array}\right], 
\]
where 
$\pmb{0}I_k 
:=
\begin{bmatrix} 
\pmb{0}_{k,1} & I_k 
\end{bmatrix}$, 
$-m_k\cdot\pmb{0}I_k 
=
\begin{bmatrix} 
\pmb{0}_{k,1} & -m_k\cdot I_k 
\end{bmatrix}$, 
$\pmb{0}_{k,l}$ is the $k \times l$ zero matrix, 
$I_k$ is the identity matrix of size $k$, 
\[
\aligned
V_k^+ 
& := 
\left[
\begin{array}{cccccc}
1 &  &  &  & & 0  \\ 
 & 1 &  &  & & 1  \\
 &  & \ddots &  & \iddots & \\
 & &  & 1 & & 
\end{array}
\right] 
\quad \left(\frac{k+3}{2}\right)\times\left(k+1\right) 
\quad \text{if $k$ is odd,}
\\
& := 
\left[
\begin{array}{ccccccc}
1 &  &  &  & & & 0   \\ 
 & 1 &  &  & & & 1   \\
 &  & \ddots & & & \iddots &    \\
 & &  & 1 & 1 & & 
\end{array}
\right] 
\quad \left(\frac{k+2}{2}\right)\times\left(k+1\right) 
\quad \text{if $k$ is even,}
\endaligned
\]
\[
\aligned
V_k^- 
& := 
\left[
\begin{array}{cccccccc}
1 &  &  &  & & & & 0   \\ 
 & 1 &  &  & & & & -1   \\
 &  & \ddots & & & & \iddots &   \\
 & &  & 1 & 0&  -1 & & 
\end{array}
\right] 
\quad \left(\frac{k+1}{2}\right)\times\left(k+1\right) 
\quad \text{if $k$ is odd}, 
\\
& := 
\left[
\begin{array}{cccccccc}
1 &  &  &  & &  & 0   \\ 
 & 1 &  &  & &  & -1   \\
 &  & \ddots &  & & \iddots &   \\
 & &  & 1 &  -1 & & 
\end{array}
\right] 
\quad \left(\frac{k+2}{2}\right)\times\left(k+1\right) 
\quad \text{if $k$ is even},
\endaligned
\]
and matrices $W_k^{\pm}$ are defined by adding column vectors ${}^{t}(1~0~\cdots~0)$ 
to the right-side end of matrices $V_k^\pm$: 
\[
\aligned
W_k^+ 
& :=\left[
\begin{array}{cccccc|c}
1 & & & & & 0 & 1 \\
  & 1 & & & & 1 & \\
  & & \ddots & & \iddots & & \\
  & & & 1 & & &
\end{array}
\right] 
\quad \left(\frac{k+3}{2}\right)\times\left(k+2\right) 
\quad \text{if $k$ is odd,}
\\
& := \left[
\begin{array}{ccccccc|c}
1 & & & & & & 0 & 1 \\ 
  & 1 & & & & & 1 & \\
  & & \ddots & & & \iddots & & \\
  & & & 1 & 1 & & & 
\end{array}
\right] 
\quad \left(\frac{k+2}{2}\right)\times\left(k+2\right) 
\quad \text{if $k$ is even,}
\endaligned
\]
\[
\aligned
W_k^- 
& := 
\left[
\begin{array}{cccccccc|c}
1 & & & & & & & 0 & -1 \\ 
  & 1 & & & & & & -1 & \\
  & & \ddots & & & & \iddots & & \\
  & & & 1 & 0 & -1 & & &
\end{array}
\right] 
\quad \left(\frac{k+1}{2}\right)\times\left(k+2\right) 
\quad \text{if $k$ is odd,}
\\
& := 
\left[
\begin{array}{ccccccc|c}
1 & & & & &  & 0 & -1 \\ 
 & 1 & & & & & -1 & \\
 & & \ddots & & & \iddots & & \\
 & & & 1 &  -1 & & &
\end{array}
\right] 
\quad \left(\frac{k+2}{2}\right)\times\left(k+2\right) 
\quad \text{if $k$ is even}.
\endaligned
\]
\begin{lemma} \label{lem_201}
For $k \geq 1$, 
\[
\det P_k(m_k)
= 
\begin{cases} 
~\varepsilon_{2j+1}\, 2^{j} m_{2j+1}^{j+1} & \text{if $k=2j+1$}, \\
~\varepsilon_{2j}\, 2^{j} m_{2j}^{j} & \text{if $k=2j$}, 
\end{cases} \\
\]
where 
\[
\varepsilon_{2j+1} 
= 
\begin{cases}
 ~+1 & j \equiv 2, 3 ~{\rm mod}\, 4 \\
 ~-1 & j \equiv 0, 1 ~{\rm mod}\, 4
\end{cases}, \quad 
\varepsilon_{2j} 
= 
\begin{cases}
 ~+1 & j \equiv 0, 1 ~{\rm mod}\, 4 \\
 ~-1 & j \equiv 2, 3 ~{\rm mod}\, 4
\end{cases}.
\]
In particular, $P_k(m_k)$ is invertible if and only if $m_k\not=0$. 
\end{lemma}
\begin{proof}
This is trivial for $k=1$. 
Suppose that $k=2j+1 \geq 3$ and write $P_k(m_k)$ as $(v_1~\cdots~v_{2k+2})$ by its column vectors $v_l$. 
At first, we make the identity matrix $I_{k+2}$ at the left-upper corner by exchanging the columns $v_{(k+5)/2},\cdots,v_{k+1}$ 
and $v_{k+2},\cdots, v_{(3k+3)/2}$, so that
\[
\scalebox{0.9}{$
\det P_k(m_k) = \det(v_1~\cdots~v_{(k+3)/2}~v_{k+2}~\cdots~v_{(3k+3)/2}~v_{(k+5)/2}~\cdots~v_{k+1}~v_{(3k+5)/2}~\cdots~v_{2k+2}). 
$}
\]
Then, by eliminating every $1$ and $-m_k$ under $I_{k+2}$ of the left-upper corner, we have 
\[
\det P_k(m_k) = 
\det 
\begin{bmatrix}
I_{k+2} & \ast \\
~\pmb{0}_{k,k+2} & Z_k 
\end{bmatrix}. 
\]
Here $Z_k$ is the $k \times k$ matrix 
\[
Z_k = \begin{bmatrix}
Z_{k,1} & Z_{k,2} \\
\pmb{0}_{j+1,j} & Z_{k,3}
\end{bmatrix}
\]
for which $Z_{k,1}$ is the $j\times j$ antidiagonal matrix with $-1$ on the antidiagonal line and  
\[
\left[\begin{array}{c}
Z_{k,2} \\ \hline
Z_{k,3}
\end{array}\right]
=
\left[\begin{array}{ccccc} 
 & & & -m_k \\
 & & \iddots &  \\ 
 & -m_k & & \\ \hline
-m_k & & & \\ 
 & -2m_k & &  \\
 & & \ddots & \\
 & & & -2m_k 
\end{array}\right] \quad 
\begin{matrix}
j \times (j+1) \\ \\ \\ \\ \\ 
(j+1) \times (j+1).
\end{matrix}
\]
The above formula for $\det P_k(m_k)$ implies the desired result. 
The case of even $k$ is proved in a way similar to the case of odd $k$.     
\end{proof}

\begin{lemma} \label{lem_203}
\begin{equation} \label{0426_1}
P_k(m_k)^{-1}Q_k=
\begin{bmatrix}
M_{k,1} & M_{k,2} \\ M_{k,3} & M_{k,4}   
\end{bmatrix} \quad (2k+2) \times (2k+4)
\end{equation}
for every $0 \leq k \leq 2g$, 
where $M_{k,1}$, $M_{k,2}$, $M_{k,3}$, $M_{k,4}$ are $(k+1) \times (k+2)$ 
matrices defined by
\[
\begin{bmatrix}
M_{0,1} & M_{0,2} \\ M_{0,3} & M_{0,4}  
\end{bmatrix}
=
\begin{bmatrix}
1 & 1 & 0 & 0 \\  0 & 0 & 1 & -1  
\end{bmatrix}, 
\quad 
\begin{bmatrix}
M_{1,1} & M_{1,2} \\ M_{1,3} & M_{1,4}  
\end{bmatrix}
=
\begin{bmatrix}
1 & 0 & 1 & 0 & 0 & 0 \\  
0 & 1 & 0 & 0 & 0 & 0 \\
0 & 0 & 0 & 1 & 0 & -1 \\ 
0 & 1/m_1 & 0 & 0 & 0 & 0
\end{bmatrix} 
\]
and 
\[
\scalebox{0.9}{$
\aligned
M_{k,1}
&= \frac{1}{2}
\begin{bmatrix}
2 & & & & & & & & 2 \\
 & 1 & & & & & & 1 & \\
 & & \ddots & & & & \iddots & & \\
 & & & 1 & & 1 & & & \\
 & & & & 2 & & & & \\
 & & & 1 & & 1 & & & \\
 & & \iddots & & & & \ddots & & \\
0 & 1 & & & & & & 1 & 0 
\end{bmatrix} \quad \text{if $k \geq 3$ is odd,} \\
&= \frac{1}{2}
\begin{bmatrix}
2 & & & & & & & 2 \\
  & 1 & & & & & 1 & \\
  & & \ddots & & & \iddots & & \\
  & & & 1 & 1 & & &  \\
  & & & 1 & 1 & & & \\
  & & \iddots & & & \ddots & & \\
0 & 1 & & & & & 1 & 0 
\end{bmatrix} \quad \text{if $k \geq 2$ is even,}
\endaligned 
$}
\]
\[
\scalebox{0.9}{$
\aligned
M_{k,4}
&= \frac{1}{2}
\begin{bmatrix}
2 & & & & & & & & -2 \\
 & 1 & & & & & & -1 & \\
 & & \ddots & & & & \iddots & & \\
 & & & 1 & & -1 & & & \\
 & & & & 0 & & & & \\
 & & & -1 & & 1 & & & \\
 & & \iddots & & & & \ddots & & \\
0 & -1 & & & & & & 1 & 0 
\end{bmatrix} \quad \text{if $k \geq 3$ is odd,} \\
&= \frac{1}{2}
\begin{bmatrix}
2 & & & & & & & -2 \\
  & 1 & & & & & -1 & \\
  & & \ddots & & & \iddots & & \\
  & & & 1 & -1 & & &  \\
  & & & -1 & 1 & & & \\
  & & \iddots & & & \ddots & & \\
0 & -1 & & & & & 1 & 0 
\end{bmatrix} \quad \text{if $k \geq 2$ is even,}
\endaligned 
$}
\]
\[
\scalebox{0.9}{$
\aligned
M_{k,2}
&= \frac{m_k}{2}
\begin{bmatrix}
0 & & & & & & & & 0 \\
 & 1 & & & & & & -1 & \\
 & & \ddots & & & & \iddots & & \\
 & & & 1 & & -1 & & & \\
 & & & & 0 & & & & \\
 & & & -1 & & 1 & & & \\
 & & \iddots & & & & \ddots & & \\
0 & -1 & & & & & & 1 & 0 
\end{bmatrix} \quad \text{if $k \geq 3$ is odd,} \\
&= \frac{m_k}{2}
\begin{bmatrix}
0 & & & & & & & 0 \\
  & 1 & & & & & -1 & \\
  & & \ddots & & & \iddots & & \\
  & & & 1 & -1 & & &  \\
  & & & -1 & 1 & & & \\
  & & \iddots & & & \ddots & & \\
0 & -1 & & & & & 1 & 0 
\end{bmatrix} \quad \text{if $k \geq 2$ is even,} \\
M_{k,3}
&= \frac{1}{2m_k}
\begin{bmatrix}
0 & & & & & & & & 0 \\
 & 1 & & & & & & 1 & \\
 & & \ddots & & & & \iddots & & \\
 & & & 1 & & 1 & & & \\
 & & & & 2 & & & & \\
 & & & 1 & & 1 & & & \\
 & & \iddots & & & & \ddots & & \\
0 & 1 & & & & & & 1 & 0 
\end{bmatrix} \quad \text{if $k \geq 3$ is odd,} \\
&= \frac{1}{2m_k}
\begin{bmatrix}
0 & & & & & & & 0 \\
  & 1 & & & & & 1 & \\
  & & \ddots & & & \iddots & & \\
  & & & 1 & 1 & & &  \\
  & & & 1 & 1 & & & \\
  & & \iddots & & & \ddots & & \\
0 & 1 & & & & & 1 & 0 
\end{bmatrix} \quad \text{if $k \geq 2$ is even.}
\endaligned 
$}
\]
\end{lemma}
\begin{proof} Noting the definitions of $P_k(m_k)$, $Q_k$ 
and $M_{k,l}$ ($1 \leq l \leq 4$), the identity 
\[
P_k(m_k)
\begin{bmatrix}
M_{k,1} & M_{k,2} \\ M_{k,3} & M_{k,4}   
\end{bmatrix} =Q_k
\]
is checked in an elementary way. 
\end{proof}

\section{Proof of Theorem \ref{thm_01}.} \label{section_4}

In this section, we complete the proof of Theorem \ref{thm_01} 
by showing each statement in Section \ref{section_2}. 
We fix $g \in \Z_{>0}$,  $q>1$ and $\underline{C} \in \R^\ast \times \R^{2g-1} \times \R^\ast$ throughout this section. 

\begin{lemma} \label{0330_6}
Let ${\mathsf E}$ be the multiplication operator defined by \eqref{0611_1} 
for $E(z)$ of \eqref{0418_2} 
Suppose that $E(z)$ has no zeros on the real line. Then 
${\mathsf E}$ is invertible on $V_a$. 
\end{lemma}
\begin{proof} 
It is sufficient to prove that ${\mathsf E}$ is invertible on $p_1V_a$ and $p_2V_a$, 
since $e^{\pm 2 itz}(1/E(z))$ $\times f(z)=g(z)$ is impossible for any $0 \not=f,g \in L^2(\mathbb{T}_q)$.   
We have $1/E(z) \in L^\infty(\mathbb{T}_q)$, by assumption. 
Therefore, multiplication by $1/E(z)$ defines a bounded operator ${\mathsf E}^{-1}$ on $p_i V_a$ 
with the norm $\Vert {\mathsf E}^{-1} \Vert=\Vert 1/E \Vert_{L^\infty(\mathbb{T}_q)}$ for $i=1,2$. 
\end{proof}

Let us consider the equations
\begin{equation} \label{LE1}
( {\mathsf E} \pm {\mathsf E}^{\sharp} {\mathsf J}{\mathsf P}_n )\, \phi_{n}^{\pm} 
= {\mathsf E}^{\sharp} X(0) 
\quad (\phi_n^{\pm} \in W_{a,n},\, 1 \leq n \leq 2g)
\end{equation}
instead of \eqref{LE0}.  
They are equivalent to \eqref{LE1} 
if ${\mathsf E}$ is invertible. 

\begin{lemma} \label{lem0329_1}
Let $0< a \not \in q^{\Z/2}$. Let $D_n(\underline{C})$ be matrices define in Theorem \ref{thm_01}. 
Suppose that $\det D_n(\underline{C})\not=0$ for every $1 \leq n \leq 2g$.
Then $\Theta{\mathsf J}{\mathsf P_n}$ 
defines a compact operator on $W_{a,n}$ for each $0 \leq n \leq 2g$,  
and the resolvent set of 
$\Theta {\mathsf J}{\mathsf P}_n|_{W_{a,n}}$ 
contains both $\pm 1$. 
In particular, ${\mathsf I} \pm \Theta {\mathsf J}{\mathsf P}_n$ are invertible on $W_{a,n}$, and 
\eqref{LE1} has unique solutions in $W_{a,n}$ for each $0 \leq n \leq 2g$. 
\end{lemma}
\begin{remark}
If $a \in q^{\Z/2} \cap (0,\infty)$, ${\mathsf I} \pm \Theta{\mathsf J}P_n$ may not be invertible. 
\end{remark}
\begin{proof} 
It is trivial for $n=0$, since ${\mathsf P}_0|_{W_{a,0}}=0$. 
Let $n \geq 1$. 
By definition, $\mathsf{P}_n$ is a projection from $V_a$ into $V_{a,n}$, 
so $\Theta\mathsf{J}{\mathsf P}_n$ is an operator on $W_{a,n}$.  
The image of $W_{a,n}$ by $\mathsf{E}^\sharp\mathsf{J}\mathsf{P}_n$ 
is finite dimensional by definition of $\mathsf{E}^\sharp$ and \eqref{0421_2}, 
thus $\Theta\mathsf{J}\mathsf{P}_n
=\mathsf{E}^{-1}(\mathsf{E}^\sharp\mathsf{J}\mathsf{P}_n)$ 
is a finite rank operator on $W_{a,n}$. Hence it is compact. 
Thus, $\lambda \in \C\setminus \{0\}$ is either an eigenvalue of 
$\Theta {\mathsf J}{\mathsf P}_n|_{W_{a,n}}$
or an element of the resolvent set of 
$\Theta {\mathsf J}{\mathsf P}_n|_{W_{a,n}}$. 
Assume that $\Theta{\mathsf J}{\mathsf P_n} \phi_n = \pm \phi_n$.
Then $\Vert \Theta{\mathsf J}{\mathsf P_n} \phi_n \Vert = \Vert \phi_n \Vert$. 
Because $\Theta$ and $p_i$ ($i=1,2$) commute,  
\[
\aligned
\Vert \Theta{\mathsf J}{\mathsf P_n} \phi_n \Vert^2 
& = \frac{\log q}{2\pi}\int_{I_q} \left|\Theta(z)p_1{\mathsf J}{\mathsf P_n}  \phi_n(z)\right|^2 dz 
+ \frac{\log q}{2\pi}\int_{I_q} \left|\Theta(z)p_2{\mathsf J}{\mathsf P_n}  \phi_n(z)\right|^2 dz \\
& = \frac{\log q}{2\pi}\int_{I_q}  |p_1{\mathsf J}{\mathsf P_n}\phi_n(z)|^2 dz +\frac{\log q}{2\pi} \int_{I_q}  |p_2{\mathsf J}{\mathsf P_n}\phi_n(z)|^2 dz \\
& = \sum_{k=0}^{n-1} |u(k)|^2 + \sum_{l=0}^{n-1} |v(l)|^2
\endaligned
\]
by \eqref{0421_1_1}, while, 
\[
\Vert \phi_n \Vert^2 = \frac{\log q}{2\pi}\int_{I_q} |p_1\phi_n(z)|^2 dz + \frac{\log q}{2\pi}\int_{I_q} |p_2\phi_n(z)|^2 dz
= \sum_{k=0}^{\infty} |u(k)|^2 + \sum_{l=-\infty}^{n-1} |v(l)|^2
\]
by \eqref{0421_1_1}. Thus, 
\[
\phi_n = \sum_{k=0}^{n-1} u(k)X(k) + \sum_{l=0}^{n-1} v(l)Y(l).
\]
In this case, 
\[
{\mathsf E}^{\sharp}{\mathsf J}{\mathsf P_n} \phi_n =  \sum_{m=-g}^{g} \sum_{k=0}^{n-1} C_{-m} v(k)X(k+m) +  \sum_{m=-g}^{g} \sum_{l=0}^{n-1} C_{-m}u(l)Y(l-m)
\]
and
\[
\pm {\mathsf E} \phi_n 
= \pm \sum_{m=-g}^{g} \sum_{k=0}^{n-1} C_{m}u(k)X(k+m) \pm \sum_{m=-g}^{g} \sum_{l=0}^{n-1} C_m v(l)Y(l-m).
\]
Comparing the coefficients of $\{X(k)\}_{-g \leq k \leq -g+n-1}$ and $\{X(k)\}_{g \leq k \leq g+n-1}$ in 
the equality 
${\mathsf E}^{\sharp}{\mathsf J}{\mathsf P_n} \phi_n \pm {\mathsf E} \phi_n=0$, 
yields 
\begin{equation} \label{0330_5}
A^{\pm }\cdot
{}^{t}\! \begin{bmatrix}u_n(0) & \cdots & u_n(n-1) & v_n(0) & \cdots & v_n(n-1) \end{bmatrix} =0, 
\end{equation}
where 
\[
A^{\pm}
=
\left[
\begin{array}{cccc|cccc}
\pm C_{-g} & & & & C_{g} & & & \\
\pm C_{-g+1} & \pm C_{-g} & & & C_{g-1} & C_{g} & & \\ 
\vdots & \ddots & \ddots & & \vdots & \ddots & \ddots & \\
\pm C_{-g+n-1} & \pm C_{-g+n-2} & \cdots & \pm C_{-g} & C_{g-n+1} & C_{g-n+2} & \cdots & C_{g} \\ \hline 
\pm C_{g} & \pm C_{g-1} & \cdots & \pm C_{g-n+1} & C_{-g} & C_{-g+1} & \cdots & C_{-g+n-1} \\
 & \pm C_{g} &  \cdots & \pm C_{g-n+2} &  & C_{-g} &  \cdots & C_{-g+n-2}\\ 
 & & \ddots & \vdots &  & & \ddots &  \vdots \\
  & & & \pm C_{g} &   & & & C_{-g} \\
\end{array}\right]. 
\] 
Here $\det A^\pm \not=0$ by $\det D_n(\underline{C}) \not=0$. 
Therefore \eqref{0330_5} has no nontrivial solutions, which means $\phi_n=0$. 
Consequently, neither $1$ nor $-1$ is an eigenvalue, 
and hence both belong to the resolvent set 
of $\Theta{\mathsf J}{\mathsf P_n}|_{W_{a,n}}$.  
This means that $({\mathsf I} \pm \Theta{\mathsf J}{\mathsf P_n}|_{W_{a,n}})^{-1}$ 
exist and are bounded on $W_{a,n}$. 
\end{proof}

In the remaining part of the section, 
we assume that 
\begin{equation*} 
\det D_n(\underline{C}) \not=0 \quad \text{for $1 \leq n \leq 2g$},  
\end{equation*}
so that both $\mathsf{I} \pm \Theta{\mathsf J}{\mathsf P}_n$ 
are invertible on $W_{a,n}$ for every $0 \leq n \leq 2g$; 
cf. Lemmas \ref{0330_6} and \ref{lem0329_1}. 
As we see in Section \ref{section_5_1}, 
$E$ has no real zeros if $\det D_{2g}\not=0$.  
Note that 
each $\phi \in W_{a,n}$ has the absolutely convergent expansion 
\[
\phi = \sum_{k=0}^{\infty} u(k)X(k) 
+ \sum_{l=-\infty}^{n-1} v(l)Y(l)
\] 
if $\Im(z)>0$ is large enough. This is trivial for $\phi \in V_{a,n}$ 
and follows for $\phi \in \Theta\mathsf{J}{\mathsf P}_nV_{a,n}$ from \eqref{0421_2} 
and the expansion 
\[
\frac{E^\sharp (z)}{E(z)}
= \frac{C_{g} + \sum_{k=1}^{2g} C_{g-k} e^{ikz}}
{C_{-g}+\sum_{k=1}^{2g} C_{-(g-k)} e^{ikz}} 
= \sum_{m=0}^{\infty} \widetilde{C}_m e^{irmz} 
\]
that holds if $\Im(z)>0$ is large enough. 
\smallskip

Let $\phi_n^{\pm}=\sum_{k=0}^{\infty} u_n^{\pm}(k)X(k) + \sum_{l=-\infty}^{n-1} v_n^{\pm}(l)Y(l)$ 
be solutions of \eqref{LE1} for $0 \leq n \leq 2g$. 
Using coefficient of $\phi_n^{\pm}$ and putting  
\[
v_n^{\pm}(n)=v_n^{\pm}(n+1)\cdots=v_n^{\pm}(2g-1)=0
\] 
if $0 \leq n \leq 2g-1$, 
we define the column vectors $\Phi_{n}^{\pm}$ of length $8g$ by
\[
\scalebox{0.9}{$
\aligned
\Phi_{n}^{\pm} 
&= {}^{t}\begin{bmatrix} u_n^{\pm}(0) & u_n^{\pm}(1) & \cdots & u_n^\pm(4g-1) & v_n^{\pm}(2g-1) & v_n^{\pm}(2g-2) &\cdots & v_n^{\pm}(-2g) \end{bmatrix}. 
\endaligned
$}
\]
Substituting $\phi_n^{\pm}$ in \eqref{LE1} and then comparing coefficients of $X(k)$ and $Y(l)$, 
we obtain 
\begin{equation} \label{LE2}
(E_0^{+} \pm E_{n}^{\sharp}J) \Phi_{n}^{\pm} = E_0^{-} \chi_{8g} 
\quad (\Phi_n^{\pm} \in W_{a,n},\, 0 \leq n \leq 2g) 
\end{equation}
and
\begin{equation} \label{LE3}
\sum_{j=0}^{2g} C_{g-j} u_n^{\pm}(K+j) = 0, \quad  
\sum_{j=0}^{2g} C_{g-j} v_n^{\pm}(L-j) = 0 
\end{equation}
for every $K \geq 4g$ and $L \leq -2g-1$. 
\begin{lemma} \label{0501_1}
Let $0 \leq n \leq 2g$. Then $\det(E_{0}^{+} \pm E_{n}^{\sharp}J)\not=0$ 
if $\mathsf{I} \pm \Theta{\mathsf J}{\mathsf P}_n$ are invertible. 
\end{lemma} 
\begin{proof} 
For fixed $n$, equations \eqref{LE1} have unique solutions $\phi_n^\pm$, 
if $\mathsf{I} \pm \Theta{\mathsf J}{\mathsf P}_n$ are invertible. 
On the other hand, the solutions $\phi_n^\pm$ correspond to solutions of the pair of linear equations \eqref{LE2} and \eqref{LE3}. 
Therefore, $\det(E_{0}^{+} \pm E_{n}^{\sharp}J)\not=0$.  
\end{proof}

On the other hand, 
\[
\aligned
{\mathsf E} \, \phi_{n}^{\pm} 
&= {\mathsf E}^{\sharp} X(0) \mp {\mathsf E}^{\sharp} {\mathsf J}{\mathsf P}_n  \phi_{n}^{\pm} \\
&= \sum_{m=-g}^{g}C_{-m} \left( X(m) \mp \sum_{k=0}^{n-1}  \left(v_n^{\pm}(k)X(k+m) + u_n^{\pm}(k)Y(k-m)\right) \right)
\endaligned
\]
by \eqref{LE1}. 
Therefore, we can write 
\begin{equation} \label{0330_8}
{\mathsf E} \, \phi_{n}^{\pm} 
= \sum_{k=-g}^{g+n-1}\Bigl( p_n^{\pm}(k)X(k) + q_n^{\pm}(k)Y(k) \Bigr)
\end{equation}
for some real numbers $p_n^{\pm}(k)$ and $q_n^{\pm}(k)$. 
Hence 
${\mathsf E} \, \phi_{n}^{\pm}(a,z)$ are extended to smooth functions of $a$ on $(0,\infty)$ 
by the right-hand side of \eqref{0330_8}. 
We use the same notation for such extended functions. 

We put $p_n^{\pm}(k)=q_n^{\pm}(k)=0$ for every $g+n \leq k \leq 3g-1$ if $0 \leq n \leq 2g-1$ and 
define the column vectors $\Psi_n^{\pm}$ of length $8g$ by
\[
\scalebox{0.9}{$
\Psi_n^{\pm}
={}^{t}
\begin{bmatrix} 
p_n^{\pm}(-g) & p_n^{\pm}(-g+1) & \cdots & p_n^{\pm}(3g-1) & q_n^{\pm}(3g-1) & q_n^{\pm}(3g-2) & \cdots & q_n^{\pm}(-g)
\end{bmatrix}.
$}
\]
Then we have
\begin{equation} \label{0421_4}
\Psi_n^{\pm} = E_{0}^{+} \Phi_n^{\pm} 
\end{equation}
by  
\[
\aligned
{\mathsf E} \, \phi_{n}^{\pm} 
&= \sum_{k=-g}^{g+n-1}\Bigl( p_n^{\pm}(k)X(k) + q_n^{\pm}(k)Y(k) \Bigr) \\
&= \sum_{m=-g}^{g} \sum_{k=0}^{\infty} C_{m} u_n(k)X(k+m) + \sum_{m=-g}^{g}  \sum_{l=-\infty}^{n-1} C_{m} v_n(l)Y(l-m). 
\endaligned
\]

\begin{lemma} \label{0320_4} Let $0 \leq n \leq 2g$. Then 
$p_n^{\pm}(k) \pm q_n^{\pm}(k)=0$
if $g+1 \leq k \leq g+n-1$ or $-g \leq k \leq -g+n-1$, and $p_n^{\pm}(g) \pm q_n^{\pm}(g)=C_{-g}$. 
Therefore, 
\begin{equation} 
({\mathsf I} \pm {\mathsf J}){\mathsf E} \, \phi_{n}^{\pm}
= \sum_{k=-g+n}^{g} (p_n^{\pm}(k) \pm q_n^{\pm}(k))(X(k) \pm Y(k)). 
\end{equation}
\end{lemma}
\begin{proof}
The equality 
$(E_0^{+} \pm E_n^{\sharp}J)\Phi_n^{\pm} = E_0^{-} \chi_{8g}$ is written in matrix form as 
\begin{equation} \label{0312_1}
\left[ 
\begin{array}{cc|cc}
 & & &\\
 \multicolumn{2}{c|}{\raisebox{1.5ex}[0pt]{$e_0^{+}$}}& \multicolumn{2}{c}{\raisebox{1.5ex}[0pt]{$\pm e_{1,n}^\sharp$}}\\ \hline
 &  & & \\
\multicolumn{2}{c|}{\raisebox{1.5ex}[0pt]{$\pm e_{2,n}^{\sharp}$}} & \multicolumn{2}{c}{\raisebox{1.5ex}[0pt]{$e_0^{+}$}}  \\
\end{array}
\right]\Phi_n^{\pm} = E_0^{-} \chi_{8g}. 
\end{equation}
On the other hand, 
\[
\left[ 
\begin{array}{cc|cc}
 & & &\\
 \multicolumn{2}{c|}{\raisebox{1.5ex}[0pt]{$e_{0}^{+}$}} & \multicolumn{2}{c}{\raisebox{1.5ex}[0pt]{$\pm e_{0}^{\sharp}$}}\\ \hline
 &  & & \\
\multicolumn{2}{c|}{\raisebox{1.5ex}[0pt]{$\pm e_{0}^{\sharp}$}} & \multicolumn{2}{c}{\raisebox{1.5ex}[0pt]{$e_{0}^{+}$}}  \\
\end{array}
\right] \Phi_n^{\pm} = 
\begin{bmatrix} 
p_n^{\pm}(-g) \pm q_n^{\pm}(-g) \\ p_n^{\pm}(-g+1) \pm q_n^{\pm}(-g+1) \\ \vdots \\ p_n^{\pm}(3g-1) \pm q_n^{\pm}(3g-1) \\ 
\pm (p_n^{\pm}(3g-1) \pm q_n^{\pm}(3g-1)) \\ \pm(p_n^{\pm}(3g-2) \pm q_n^{\pm}(3g-2)) \\ \vdots \\ \pm (p_n^{\pm}(-g) \pm q_n^{\pm}(-g))
\end{bmatrix},
\]
since 
\begin{equation}
E_0^+ \pm J^{(8g)} E_0^+  =
\left[ 
\begin{array}{cc|cc}
 & & &\\
 \multicolumn{2}{c|}{\raisebox{1.5ex}[0pt]{$e_{0}^{+}$}} & \multicolumn{2}{c}{\raisebox{1.5ex}[0pt]{$\pm e_{0}^{\sharp}$}}\\ \hline
 &  & & \\
\multicolumn{2}{c|}{\raisebox{1.5ex}[0pt]{$\pm e_{0}^{\sharp}$}} & \multicolumn{2}{c}{\raisebox{1.5ex}[0pt]{$e_{0}^{+}$}}  \\
\end{array}
\right], 
\end{equation}
by definition of $e_{0}^{\sharp}$, and 
\[
(I^{(8g)} \pm J^{(8g)}) E_0^{+} \Phi_n^{\pm}
=
(I^{(8g)} \pm J^{(8g)}) \Psi_n^\pm
=
\begin{bmatrix} 
p_n^{\pm}(-g) \pm q_n^{\pm}(-g) \\ p_n^{\pm}(-g+1) \pm q_n^{\pm}(-g+1) \\ \vdots \\ p_n^{\pm}(3g-1) \pm q_n^{\pm}(3g-1) \\ 
\pm (p_n^{\pm}(3g-1) \pm q_n^{\pm}(3g-1)) \\ \pm(p_n^{\pm}(3g-2) \pm q_n^{\pm}(3g-2)) \\ \vdots \\ \pm (p_n^{\pm}(-g) \pm q_n^{\pm}(-g))
\end{bmatrix},
\] 
by \eqref{0421_4}. In addition,  
\begin{equation} 
\left[ 
\begin{array}{cc|cc}
 & & &\\
 \multicolumn{2}{c|}{\raisebox{1.5ex}[0pt]{$e_{0}^{+}$}} & \multicolumn{2}{c}{\raisebox{1.5ex}[0pt]{$\pm e_{0}^{\sharp}$}}\\ \hline
 &  & & \\
\multicolumn{2}{c|}{\raisebox{1.5ex}[0pt]{$\pm e_{0}^{\sharp}$}} & \multicolumn{2}{c}{\raisebox{1.5ex}[0pt]{$e_{0}^{+}$}}  \\
\end{array}
\right]\Phi_n^{\pm} = 
\left[ 
\begin{array}{cc|cc}
 & & &\\
 \multicolumn{2}{c|}{\raisebox{1.5ex}[0pt]{$e_{0}^{+}$}} & \multicolumn{2}{c}{\raisebox{1.5ex}[0pt]{$\pm e_{0,n}^{\sharp}$}}\\ \hline
 &  & & \\
\multicolumn{2}{c|}{\raisebox{1.5ex}[0pt]{$\pm e_{0}^{\sharp}$}} & \multicolumn{2}{c}{\raisebox{1.5ex}[0pt]{$e_{0}^{+}$}}  \\
\end{array}
\right]\Phi_n^{\pm}, 
\end{equation}
since $v_n^{\pm}(n)=v_n^{\pm}(n+1)\cdots=v_n^{\pm}(2g-1)=0$ for $1 \leq n \leq 2g-1$, by definition of $\Phi_n^\pm$. 
Therefore, 
\begin{equation} \label{0312_2}
\left[ 
\begin{array}{cc|cc}
 & & &\\
 \multicolumn{2}{c|}{\raisebox{1.5ex}[0pt]{$e_{0}^{+}$}} & \multicolumn{2}{c}{\raisebox{1.5ex}[0pt]{$\pm e_{0,n}^{\sharp}$}}\\ \hline
 &  & & \\
\multicolumn{2}{c|}{\raisebox{1.5ex}[0pt]{$\pm e_{0}^{\sharp}$}} & \multicolumn{2}{c}{\raisebox{1.5ex}[0pt]{$e_{0}^{+}$}}  \\
\end{array}
\right]\Phi_n^{\pm} = 
\begin{bmatrix} 
p_n^{\pm}(-g) \pm q_n^{\pm}(-g) \\ p_n^{\pm}(-g+1) \pm q_n^{\pm}(-g+1) \\ \vdots \\ p_n^{\pm}(3g-1) \pm q_n^{\pm}(3g-1) \\ 
\pm (p_n^{\pm}(3g-1) \pm q_n^{\pm}(3g-1)) \\ \pm(p_n^{\pm}(3g-2) \pm q_n^{\pm}(3g-2)) \\ \vdots \\ \pm (p_n^{\pm}(-g) \pm q_n^{\pm}(-g))
\end{bmatrix}.
\end{equation}
Here we find that $2g$ rows of both $E_0^{+} \pm E_n^{\sharp}J$ 
and 
\[
\begin{bmatrix} e_0^+ & e_{0,n}^\sharp \\ e_{0}^\sharp & e_0^+ \end{bmatrix}
\] 
with indices $(2g+1, 2g+2, \cdots, 4g)$ 
and $n$ rows of both $E_0^{+} \pm E_n^{\sharp}J$ and 
\[
\begin{bmatrix} e_0^+ & e_{0,n}^{\sharp} \\ e_{0}^{\sharp} & e_0^{+}\end{bmatrix}
\]
with indices $(8g-n+1, 8g-n+2,\cdots, 8g)$ have the same entries. 
Therefore, by comparing $2g$ rows of \eqref{0312_1} and \eqref{0312_2} 
with indices $(2g+1, 2g+2,\cdots, 4g)$, we obtain 
$p_n^{\pm}(g) \pm q_n^{\pm}(g) = C_{-g}$ 
and $p_n^{\pm}(k) \pm q_n^{\pm}(k) = 0$ $(g+1 \leq k \leq 3g-1)$.
Similarly, by comparing $n$ rows of \eqref{0312_1} and \eqref{0312_2} 
with indices $(8g-n+1, 8g-n+2, \cdots, 8g)$, we obtain
$p_n^{\pm}(k) \pm q_n^{\pm}(k) = 0 $ $(-g \leq k \leq -g+n-1)$.
Hence we complete the proof. 
\end{proof}

We define the column vectors $A_n^\ast$ and $B_n^\ast$ of length $8g$ by 
\begin{equation} \label{0501_6}
\aligned
A_n^\ast=A_n^\ast(\underline{C}) &:= (I+J)\Psi_n^+ = (I+J)E_0^{+}\Phi_{n}^{+}, \\
B_n^\ast=B_n^\ast(\underline{C}) &:= (I-J)\Psi_n^- = (I-J)E_0^{+}\Phi_{n}^{-},
\endaligned
\end{equation}
where $I=I^{(8g)}$ and $J=J^{(8g)}$. We define the row vector $F(a,z)$ of length $8g$ by
\[
F(a,z)
:= \begin{bmatrix} X(-g)~X(-g+1)~\cdots~X(g)\quad 0~\cdots~0 \quad Y(g)~Y(g-1)~\cdots~Y(-g) \end{bmatrix}.
\]
Then,  we obtain
\begin{equation} \label{0330_7}
\aligned
A_n^\ast(a,z) = \frac{1}{2}F(a,z) \cdot A_n^\ast, \quad 
B_n^\ast(a,z) = \frac{1}{2i}F(a,z) \cdot B_n^\ast
\endaligned
\end{equation}
by \eqref{0330_9} and Lemma \ref{0320_4}. 

\begin{proposition} \label{0330_4} 
\[
- a\frac{d}{da} A_n^\ast(a,z) = -z B_n^\ast(a,z), \quad -a\frac{d}{da} B_n^\ast(a,z) = z A_n^\ast(a,z)
\]
for every $0 \leq n \leq 2g$. 
\end{proposition}

\noindent
We provide two lemmas used to prove Proposition \ref{0330_4}. 
\begin{lemma} \label{co1} 
\begin{equation} 
\aligned 
u_n^{+}(k) & = u_n^{-}(k) \quad (0 \leq k \leq 4g-1), \\
v_n^{+}(k) & = -v_n^{-}(k) \quad (-2g \leq k \leq 2g-1)
\endaligned
\end{equation}
for every $0 \leq n \leq 2g$. 
\end{lemma}
\begin{proof} We have 
$\Phi_n^{\pm} = (E_{0}^{+} \pm E_{n}^{\sharp}J)^{-1}E_{0}^{-}\chi_{8g}$ with 
\[
E_0^{+} \pm E_n^{\sharp}J 
=
\left[ 
\begin{array}{cc|cc}
 & & &\\
 \multicolumn{2}{c|}{\raisebox{1.5ex}[0pt]{$e_0^{+}$}}& \multicolumn{2}{c}{\raisebox{1.5ex}[0pt]{$ \pm e_{1,n}^\sharp$}}\\ \hline
 &  & & \\
\multicolumn{2}{c|}{\raisebox{1.5ex}[0pt]{$\pm e_{2,n}^\sharp$}} & \multicolumn{2}{c}{\raisebox{1.5ex}[0pt]{$e_0^{+}$}}  \\
\end{array}
\right] 
\]
by \eqref{LE2} and definition of $E_n^{\sharp} J$. 
Applying Lemma \ref{mat_1} to 
$A=D=P:=e_0^{+}$, 
$B=\pm Q:= \pm e_{1,n}^\sharp$ and 
$C=\pm R:= \pm e_{2,n}^\sharp$, 
we obtain 
\[
\Phi_n^{\pm} = 
\begin{bmatrix}
(P-QP^{-1}R)^{-1}& \mp P^{-1}Q(P-RP^{-1}Q)^{-1} \\ 
\mp P^{-1}R(P-QP^{-1}R)^{-1} & (P-RP^{-1}Q)^{-1}
\end{bmatrix}
E_0^{-} \chi_{8g}.
\]
This establishes Lemma \ref{co1}, since 
all $4g$ entries of $E_0^{-}\chi_{8g}$ with indices $(4g+1,4g+2,\cdots,8g)$ are zero. 
\end{proof}
\begin{lemma} \label{co2} 
\begin{equation} 
p_n^{+}(k) + q_n^{+}(k) = p_n^{-}(k) - q_n^{-}(k) \quad (-g \leq k \leq g+n-1). 
\end{equation}
for every $0 \leq n \leq 2g$, where $p_n^{\pm}(k) \pm q_n^{\pm}(k)=0$ if $-g \leq k \leq -g+n-1$ or $g+1 \leq k \leq g+n-1$ 
by Lemma \ref{0320_4}.
\end{lemma}
\begin{proof} By \eqref{LE1}, 
\[
\scalebox{0.9}{$
\aligned
{\mathsf E}\, \phi_{n}^{\pm} 
&= {\mathsf E}^{\sharp} X(0) \mp {\mathsf E}^{\sharp} {\mathsf J}{\mathsf P}_n \phi_{n}^{\pm} \\
&= \sum_{m=-g}^{g} C_{-m}X(m) \mp \sum_{m=-g}^{g} C_{-m}\sum_{k=0}^{n-1} v_n^{\pm}(k)X(k+m) \mp \sum_{m=-g}^{g} C_{-m} \sum_{l=0}^{n-1} u_n^{\pm}(l)Y(l-m).
\endaligned
$}
\]
Therefore, 
\[
\scalebox{0.85}{$
\aligned
({\mathsf I} + {\mathsf J}){\mathsf E}\, \phi_{n}^{+} 
&=  \sum_{m=-g}^{g} C_{-m}X(m)
- \sum_{m=-g}^{g} \left( C_{m} \sum_{k=0}^{n-1} u_n^{+}(k)X(k+m) + C_{-m}\sum_{k=0}^{n-1} v_n^{+}(k)X(k+m) \right)\\ 
& + \sum_{m=-g}^{g} C_{-m}Y(m) - \sum_{m=-g}^{g} \left( C_{m} \sum_{l=0}^{n-1} u_n^{+}(l)Y(l+m) + C_{-m}\sum_{l=0}^{n-1} v_n^{+}(l)Y(l+m) \right).
\endaligned
$}
\]
On the other hand, 
\[
\scalebox{0.85}{$
\aligned
({\mathsf I} - {\mathsf J}){\mathsf E}\, \phi_{n}^{-} 
&=  \sum_{m=-g}^{g} C_{-m}X(m)
- \sum_{m=-g}^{g} \left( C_{m} \sum_{k=0}^{n-1} u_n^{-}(k)X(k+m) - C_{-m}\sum_{k=0}^{n-1} v_n^{-}(k)X(k+m) \right)\\ 
& \quad - \sum_{m=-g}^{g} C_{-m}Y(m) + \sum_{m=-g}^{g} \left( C_{m} \sum_{l=0}^{n-1} u_n^{-}(l)Y(l+m) - C_{-m}\sum_{l=0}^{n-1} v_n^{-}(l)Y(l+m) \right) \\
&=  \sum_{m=-g}^{g} C_{-m}X(m)
- \sum_{m=-g}^{g} \left( C_{m} \sum_{k=0}^{n-1} u_n^{+}(k)X(k+m) + C_{-m}\sum_{k=0}^{n-1} v_n^{+}(k)X(k+m) \right)\\ 
& \quad - \sum_{m=-g}^{g} C_{-m}Y(m) + \sum_{m=-g}^{g} \left( C_{m} \sum_{l=0}^{n-1} u_n^{+}(l)Y(l+m) + C_{-m}\sum_{l=0}^{n-1} v_n^{+}(l)Y(l+m) \right)
\endaligned
$}
\]
by Lemma \ref{co1}. 
Comparing the right-hand sides of the above formulas of 
$({\mathsf I} \pm {\mathsf J}){\mathsf E}\, \phi_{n}^{\pm}$ 
with formulas of $({\mathsf I} \pm {\mathsf J}){\mathsf E}\, \phi_{n}^{\pm}$ in Lemma \ref{0320_4}, we obtain Lemma \ref{co2}. 
\end{proof}

\begin{proof}[Proof of Proposition \ref{0330_4}]
By definition of $X(k)$ and $Y(l)$, \eqref{0330_8}, and Lemma \ref{0320_4},
\begin{equation} 
({\mathsf I}\pm {\mathsf J}){\mathsf E} \, \phi_{n}^{\pm}(a,z)
= \sum_{k=-g+n}^{g} (p_n^{\pm}(k) \pm q_n^{\pm}(k))((q^k/a)^{iz} \pm (q^k/a)^{-iz}). 
\end{equation}
Therefore, the differentiability of $A_n^\ast(a,z)$ and $B_n^\ast(a,z)$ with respect to $a$ is trivial, and
\begin{equation} 
\aligned
-a\frac{d}{da}
({\mathsf I} + {\mathsf J}){\mathsf E} \, \phi_{n}^{+}(a,z)
&= iz \sum_{k=-g+n}^{g} (p_n^{+}(k) + q_n^{+}(k))((q^k/a)^{iz} - (q^k/a)^{-iz}), \\
-a\frac{d}{da}\frac{1}{i}
({\mathsf I} - {\mathsf J}){\mathsf E} \, \phi_{n}^{-}(a,z)
&= z \sum_{k=-g+n}^{g} (p_n^{-}(k) - q_n^{-}(k))((q^k/a)^{iz} + (q^k/a)^{-iz}).
\endaligned
\end{equation}
Applying Lemma \ref{co2} to the right-hand sides, we obtain Proposition \ref{0330_4} by definition \eqref{0418_6}. 
\end{proof}

As mentioned above, $A_n^\ast(a,z)$ and $B_n^\ast(a,z)$ of \eqref{0418_6} are smooth functions of $a$ on $(0,\infty)$ for every $0 \leq n \leq 2g$. 
However $A^\ast(a,z)$ and $B^\ast(a,z)$ of \eqref{0330_9} may be discontinuous at $a \in (0,\infty) \cap q^{\Z/2}$. 
Therefore, as mentioned in Section \ref{section_2}, 
the next aim is to make modifications so that  $A^\ast(a,z)$ and $B^\ast(a,z)$ become continuous functions of $a$ on $[1,q^g)$. 
The essential part is the following proposition. 

\begin{proposition} \label{prop_0313_1} 
Let $1 \leq n \leq 2g$. 
Suppose that $\mathsf{I} \pm \Theta{\mathsf J}{\mathsf P}_{n-1}$ are invertible 
on $W_{a,n-1}$ for every $q^{(n-2)/2} < a < q^{n/2}$ and 
$\mathsf{I} \pm \Theta{\mathsf J}{\mathsf P}_n$ are  invertible 
on $W_{a,n}$ for every $q^{(n-1)/2} < a < q^{n/2}$.
Define $\alpha_n$ and $\beta_n$ by \eqref{0330_12}. Then 
\begin{equation} \label{0331_5}
\alpha_n = 
\begin{cases}
~\displaystyle{\frac{\det(E^{+} + E^{-}J_1)}{\det(E^+)}}
& \text{if $n=1$}, \\[4pt]
~\displaystyle{
\frac{\det(E^{+} + E^{-}J_2)}{\det(E^+)}
\frac{\det(E_0^{+})}{\det(E_0^{+} + E_{1}^{\sharp}J)}}
& \text{if $n=2$}, \\[10pt]
~\displaystyle{
\frac{\det(E^{+} + E^{-}J_n)}{\det(E^{+} + E^{-}J_{n-2})}
\frac{\det(E_0^{+} + E_{n-2}^{\sharp}J)}{\det(E_0^{+} + E_{n-1}^{\sharp}J)}
}
& \text{if $3 \leq n \leq 2g$},
\end{cases}
\end{equation}
\begin{equation} \label{0331_6}
\beta_n = 
\begin{cases}
~\displaystyle{\frac{\det(E^{+} - E^{-}J_1)}{\det(E^+)}}
& \text{if $n=1$}, \\[4pt]
~\displaystyle{
\frac{\det(E^{+} - E^{-}J_2)}{\det(E^+)}
\frac{\det(E_0^{+})}{\det(E_0^{+} + E_{1}^{\sharp}J)}}
& \text{if $n=2$}, \\[10pt]
~\displaystyle{
\frac{\det(E^{+} - E^{-}J_n)}{\det(E^{+} - E^{-}J_{n-2})}
\frac{\det(E_{0}^{+} - E_{n-2}^{\sharp}J)}{\det(E_0^{+} - E_{n-1}^{\sharp}J)}
}
& \text{if $3 \leq n \leq 2g$}.
\end{cases}
\end{equation}
In particular, $\alpha_n$ and $\beta_n$ depend only on $\underline{C}$. 
\end{proposition}
We prove Proposition \ref{prop_0313_1} 
after preparing several lemmas. 

\begin{lemma} \label{lem_0314_2} For every $0 \leq n \leq 2g$, 
\begin{equation} \label{0320_3}
\det(E_{0}^{+} + E_{n}^{\sharp}J)=\det(E_{0}^{+} - E_{n}^{\sharp}J).
\end{equation}
\end{lemma}
\begin{proof} 
Since $E_{n}^{\sharp}J=0$ by definition, we suppose that $n \geq 1$.
We have
\begin{equation}
\aligned
E_0^{+} \pm E_n^{\sharp}J
&=
\left[ 
\begin{array}{cc|cc}
 & & &\\
 \multicolumn{2}{c|}{\raisebox{1.5ex}[0pt]{$e_0^{+}$}}& \multicolumn{2}{c}{\raisebox{1.5ex}[0pt]{$ \pm e_{1,n}^{\sharp}$}}\\ \hline
 &  & & \\
\multicolumn{2}{c|}{\raisebox{1.5ex}[0pt]{$\pm e_{2,n}^{\sharp}$}} & \multicolumn{2}{c}{\raisebox{1.5ex}[0pt]{$e_0^{+}$}}  \\
\end{array}
\right] \\
&=
\left[ 
\begin{array}{cc|cc|cc}
 & & & & & \\
  \multicolumn{2}{c|}{\raisebox{1.5ex}[0pt]{$A$}} 
& \multicolumn{2}{c|}{\raisebox{1.5ex}[0pt]{\BigZero}}
& \multicolumn{2}{c}{\raisebox{1.5ex}[0pt]{$\pm B$}} \\ \hline
 & & & & & \\
  \multicolumn{2}{c|}{\raisebox{1.5ex}[0pt]{$C$}} 
& \multicolumn{2}{c|}{\raisebox{1.5ex}[0pt]{$D$}} 
& \multicolumn{2}{c}{\raisebox{1.5ex}[0pt]{\BigZero}} \\ \hline
 & & & & & \\
  \multicolumn{2}{c|}{\raisebox{1.5ex}[0pt]{$\pm B$}} 
& \multicolumn{2}{c|}{\raisebox{1.5ex}[0pt]{$E$}} 
& \multicolumn{2}{c}{\raisebox{1.5ex}[0pt]{$A$}}
\end{array}
\right]
=
\begin{array}{c|c|c|c}
 2g+n & 4g-2n & 2g+n & \\ \hline
 & & & 2g+n  \\ \hline
 & & & 4g-2n \\ \hline
 & & & 2g+n 
\end{array}
\endaligned
\end{equation}
by the definition of $e_{1,n}^{\sharp}$ and $e_{2,n}^{\sharp}$, where 
\[
A = [e_0^{+}]_{\nwarrow 2g+n}, \quad B = [e_{2,n}^{\sharp}]_{\swarrow  2g+n}, \quad 
D = \begin{bmatrix} [e_0^{+}]_{\searrow 2g-n} & \\ & [e_0^{+}]_{\nwarrow 2g-n} \end{bmatrix}
\]
and the right-hand side means the size of each block of the matrix of the second line. 
Applying Lemma \ref{mat_3} to the $(4g-2g-n)$ columns of $E_0^{+} \pm E_n^{\sharp}J$ 
with indices $(2g+n+1,\cdots,4g)$, 
and then applying Lemma \ref{mat_3} to the $(4g-2g-n)$ rows of the resulting matrix with indices $(2g+n+1,\cdots ,4g)$, 
we obtain
\[
\det( E_0^{+} \pm E_n^{\sharp}J) 
= \det(D)\det \begin{bmatrix} A & \pm B \\ \pm B & A \end{bmatrix}.
\]
By Lemma \ref{mat_2}, 
\[
\det \begin{bmatrix} A & \pm B \\ \pm B & A \end{bmatrix}
= \det(D)\det(A+B)\det(A-B).
\]
This equality shows \eqref{0320_3}.
\end{proof}

\begin{lemma}\label{0319_1} 
\[
\aligned
{\rm adj}(e_0^{+} \pm e_0^{-}J_{n}^{(4g)})
e_0^{-}\chi_{4g} 
=
{\rm adj}(e_0^{+} \pm e_1^{-}J_{n}^{(4g)})
e_0^{-}\chi_{4g}
\endaligned
\]
for every $1 \leq n \leq 2g$, where ${\rm adj}(A)=(\det A)A^{-1}$ is the adjugate matrix of $A$. 
\end{lemma}
\begin{proof} 
Let $w={}^{t}\begin{bmatrix}C_g & \cdots & C_{-g} & 0 & \cdots & 0 \end{bmatrix}$ 
be the column vector of length $4g$. 
Denote by $u_{n,j}$ and $v_{n,j}$ the column vectors of $e_0^{+} \pm e_0^{-}J_{n}^{(4g)}$ 
and $e_0^{+} \pm e_1^{-}J_{n}^{(4g)}$, respectively:
\[
e_0^{+} \pm e_0^{-}J_{n}^{(4g)} = 
\begin{bmatrix}
u_{n,1} & \cdots & u_{n,4g}
\end{bmatrix}, \quad 
e_0^{+} \pm e_1^{-}J_{n}^{(4g)} = 
\begin{bmatrix}
v_{n,1} & \cdots & v_{n,4g}
\end{bmatrix}.
\]
By Cramer's formula, it is sufficient to prove that 
\begin{equation} \label{0319_2}
\aligned
\det[&u_{n,1}~\cdots~u_{n,j-1}~\stackrel{j}{\stackrel{\smile}{w}}~u_{n,j+1}~\cdots~u_{n,4g}]\\
&=
\det[v_{n,1}~\cdots~v_{n,j-1}~\stackrel{j}{\stackrel{\smile}{w}}~v_{n,j+1}~\cdots~v_{n,4g}]
\endaligned
\end{equation}
for every $1 \leq j \leq 2g$. 
By definition of $e_0^{-}$ and $e_1^{-}$, $u_{n,j} = v_{n,j}$ for every  $j \not= n$, $1 \leq j \leq 4g$,  
and $u_{n,n} = v_{n,n} + w$.
Therefore, \eqref{0319_2} is trivial if $j=n$. If $1 \leq j <n$, 
\begin{equation}
\aligned
\det[&u_{n,1}~\cdots~u_{n,j-1}~\stackrel{j}{\stackrel{\smile}{w}}~u_{n,j+1}~\cdots~u_{n,4g}] \\
& = \det[v_{n,1}~\cdots~v_{n,j-1}~\stackrel{j}{\stackrel{\smile}{w}}~v_{n,j+1}~\cdots~v_{n,4g}] +\det[v_{n,1}~\cdots~~\stackrel{j}{\stackrel{\smile}{w}}~\cdots~\stackrel{n}{\stackrel{\smile}{w}}~\cdots~v_{n,4g}] \\
& = \det[v_{n,1}~\cdots~v_{n,j-1}~\stackrel{j}{\stackrel{\smile}{w}}~v_{n,j+1}~\cdots~v_{n,4g}] +0.
\endaligned
\end{equation}
Hence we obtain \eqref{0319_2}. 
\end{proof}

\begin{lemma} \label{0319_3} 
\[
\det(e_0^{+} \pm e_1^{-}J_{1}^{(4g)})
=
\det(e_0^{+} \pm e_1^{-}J_{2}^{(4g)})
= \det(e_0^{+})
\]
and
\[
\det(e_0^{+} \pm e_1^{-}J_{n}^{(4g)})
= \det(e_0^{+} \pm e_0^{-}J_{n-2}^{(4g)})
\]
for $3 \leq n \leq 2g$. 
\end{lemma}
\begin{proof}
We have 
$\det(e_0^{+} \pm e_1^{-}J_{1}^{(4g)}) = \det(e_0^{+})$, 
since $e_1^{-}J_{1}^{(4g)}=0$. 
On the other hand, $\det(e_0^{+} \pm e_1^{-} J_{2}^{(4g)}) = C_{-g}^{4g}= \det(e_0^{+})$, 
since $e_0^{+}$ is a lower triangular matrix of size $4g$ and 
\[
e_1^{-}J_{2}^{(4g)}=
\left[
\begin{array}{c|ccc}
0 & 0 & \cdots & 0  \\ \hline
\ast  & & & \\
\vdots & \multicolumn{2}{c}{\raisebox{1.5ex}[0pt]{\BigZero}} \\
\ast & & & 
\end{array}
\right].
\]
For $3 \leq n \leq 2g$, we have 
\[
e_0^{+} \pm e_1^{-}J_{n}^{(4g)}
= 
\left[
\begin{array}{c|ccc}
C_{-g} & 0 & \cdots & 0 \\ \hline
\ast & & & \\
\vdots & \multicolumn{3}{c}{\raisebox{1.0ex}[0pt]{$[e_0^{+} \pm e_0^{-}J_{n-2}^{(4g)}]_{\nwarrow 4g-1}$}} \\
\ast & & & 
\end{array}
\right].
\]
The $4g$-th row and the $4g$-th column of $e_0^{+} \pm e_0^{-}J_{n-2}^{(4g)}$ 
are $[0~\cdots~0~C_{g}~\cdots~C_{-g}]$ and ${}^{t}[0~\cdots~0~C_{-g}]$, respectively. 
Therefore, 
\[
\det([e_0^{+} \pm e_0^{-}J_{n-2}^{(4g)}]_{\nwarrow 4g-1})
= C_{-g}^{-1}\det(e_0^{+} \pm e_0^{-}J_{n-2}^{(4g)}).
\]
Hence, we obtain the equality of Lemma \ref{0319_3}. 
\end{proof}

\begin{lemma} \label{0418_10}
For $1 \leq n \leq 2g$, 
\begin{equation} \label{0331_4}
\det (E_0^{+} \pm E_{n}^{\sharp}J) 
= C_{-g}^{8g-2n}  \det([E_0^{+} + E_0^{-} J_n^{(8g)}]_{\nwarrow n})
\det([E_0^{+} - E_0^{-} J_n^{(8g)}]_{\nwarrow n}).
\end{equation}
\end{lemma}
\begin{proof} By the proof of Lemma \ref{lem_0314_2},
\[
\det (E_0^{+} \pm E_{n}^{\sharp}J) 
= C_{-g}^{4g-2n}
\det([e_0^{+}]_{\nwarrow 2g+n} + [e_{2,n}^\sharp]_{\swarrow  2g+n})
\det([e_0^{+}]_{\nwarrow 2g+n} - [e_{2,n}^\sharp]_{\swarrow  2g+n})
\]
Here  
\[
[e_0^{+}]_{\nwarrow 2g+n} \pm [e_{2,n}^\sharp]_{\swarrow  2g+n}
= 
[E_0^{+} \pm E_0^{-}J_n^{(8g)}]_{\nwarrow  2g+n},
\]
since 
$[e_0^{+}]_{\nwarrow 2g+n} = [E_0^+]_{\nwarrow  2g+n}$ 
and  
$[e_{2,n}^\sharp]_{\swarrow  2g+n} = [E_0^{-} J_n^{(8g)}]_{\nwarrow  2g+n}$ 
by definition of matrices. 
Applying Lemma \ref{mat_3} to the $2g$ columns of $[E_0^{+} \pm E_0^{-} J_n^{(8g)}]_{\nwarrow  2g+n}$ 
with indices $(n+1,n+2,\cdots,2g+n)$, we have 
\[
\det[E_0^{+} \pm E_0^{-} J_n^{(8g)}]_{\nwarrow  2g+n}
= C_{-g}^{2g} \det([E_0^{+} \pm E_0^{-} J_n^{(8g)}]_{\nwarrow n}), 
\]
since all entries above $[E_0^{+} \pm E_0^{-} J_n^{(8g)}]_{\searrow 2g}$ 
in $[E_0^{+} \pm E_0^{-} J_n^{(8g)}]_{\nwarrow  2g+n}$
are zero and  $[E_0^{+} \pm E_0^{-} J_n^{(8g)}]_{\searrow 2g}$ 
is lower triangular with diagonal $(C_{-g},\cdots,C_{-g})$.
Therefore, we obtain \eqref{0331_4}. 
\end{proof}

\begin{lemma} \label{0330_21} 
\[ 
\det(E_0^{+} \pm E_0^{-}J_{n}^{(8g)})\bigl(1 \mp ( u_{n}^{\pm}(n-1) + v_{n}^{\pm}(0) )\bigr)
=
\det(E_0^{+} \pm E_0^{-}J_{n-2}^{(8g)}).
\]
\end{lemma}
\begin{proof}
Let $w={}^{t}\begin{bmatrix}C_g & \cdots & C_{-g} & 0 & \cdots & 0 \end{bmatrix}$ 
be the column vector of length $8g$. 
Denote by $u_{n,j}^{\pm}$ the columns of $E_0^+ \pm E_{n}^{\sharp}J$. 
By Cramer's formula, 
\[
\scalebox{0.9}{$
\aligned
\det&(E_0^{+} \pm E_{n}^{\sharp}J)
(1 \mp ( u_{n}^{\pm}(n-1) + v_{n}^{\pm}(0) )) \\
& = \det(E_0^{+} \pm E_{n}^{\sharp}J) 
  \mp  \det[u_{n,1}^{\pm}~\cdots~\stackrel{n}{\stackrel{\smile}{w}}~\cdots~u_{n,8g}^{\pm}] 
  \mp \det[u_{n,1}^{\pm}~ \cdots~\stackrel{6g}{\stackrel{\smile}{w}}~\cdots~u_{n,8g}^{\pm}] \\
& = \det(E_0^{+} \pm E_{n}^{\sharp}J) 
  \mp  \det[u_{n,1}^{\pm}~\cdots~\stackrel{n}{\stackrel{\smile}{w}}~\cdots~u_{n,8g}^{\pm}] 
  - \det[u_{n,1}^{\pm}~ \cdots~\stackrel{6g}{\stackrel{\smile}{ \pm w}}~\cdots~u_{n,8g}^{\pm}] \\
& = \det(E_0^{+} \pm E_{n}^{\sharp}J) 
  \mp  \det[u_{n,1}^{\pm}~\cdots~\stackrel{n}{\stackrel{\smile}{w}}~\cdots~u_{n,8g}^{\pm}] \\
& \quad - \det[u_{n,1}^{\pm}~ \cdots~\stackrel{6g}{\stackrel{\smile}{u_{n,6g}^{\pm}}}~\cdots~u_{n,8g}^{\pm}] 
  + \det[u_{n,1}^{\pm}~ \cdots~\stackrel{6g}{\stackrel{\smile}{u_{n,6g}^{\pm} \mp w}}~\cdots~u_{n,8g}^{\pm}] \\
& = 
  \mp  \det[u_{n,1}^{\pm}~\cdots~\stackrel{n}{\stackrel{\smile}{w}}~\cdots~u_{n,8g}^{\pm}] 
  + \det[u_{n,1}^{\pm}~ \cdots~\stackrel{6g}{\stackrel{\smile}{u_{n,6g}^{\pm} \mp w}}~\cdots~u_{n,8g}^{\pm}] \\
& = \mp  \det[u_{n,1}^{\pm}~\cdots~\stackrel{n}{\stackrel{\smile}{w}}~\cdots~\stackrel{6g}{\stackrel{\smile}{\pm w}}~\cdots~u_{n,8g}^{\pm}] 
    \mp  \det[u_{n,1}^{\pm}~\cdots~\stackrel{n}{\stackrel{\smile}{w}}~\cdots~\stackrel{6g}{\stackrel{\smile}{u_{n,6g}^{\pm} \mp w}}~\cdots~u_{n,8g}^{\pm}] \\
& \quad  + \det[u_{n,1}^{\pm}~ \cdots~\stackrel{6g}{\stackrel{\smile}{u_{n,6g}^{\pm} \mp w}}~\cdots~u_{n,8g}^{\pm}] \\
& = \det[u_{n,1}^{\pm}~\cdots~\stackrel{n}{\stackrel{\smile}{u_{n,6g}^{\pm} \mp w}}~\cdots
~\stackrel{6g}{\stackrel{\smile}{u_{n,6g}^{\pm} \mp w}}~\cdots~u_{n,8g}^{\pm}].
\endaligned
$}
\]
Therefore, it is sufficient to prove 
\begin{equation} \label{0331_1}
\aligned
\,& \det(E_0^{+} \pm E_0^{-}J_{n-2}^{(8g)})\det(E_0^{+} \pm E_{n}^{\sharp}J) \\
& =
\det(E_0^{+} \pm E_0^{-}J_{n}^{(8g)}) 
\det[u_{n,1}^{\pm}~\cdots~\stackrel{n}{\stackrel{\smile}{u_{n,6g}^{\pm} \mp w}}~\cdots
~\stackrel{6g}{\stackrel{\smile}{u_{n,6g}^{\pm} \mp w}}~\cdots~u_{n,8g}^{\pm}].
\endaligned
\end{equation}
Applying Lemma \ref{mat_3} to the columns $(n+1,n+2,\cdots,8g)$, we obtain 
\[
\aligned
\det(E_0^+ \pm E_0^-J_{n-2}^{(8g)})
& = \det([E_0^+ \pm E_0^-J_{n-2}^{(8g)}]_{\nwarrow  n-2})
\det([E_0^+ \pm E_0^-J_{n-2}^{(8g)}]_{\searrow 8g-n+2}) \\
& = C_{-g}^{8g-n+2} \det([E_0^+ \pm E_0^-J_{n-2}^{(8g)}]_{\nwarrow n-2}),
\endaligned 
\]
since all entries above $[E_0^+ \pm E_0^- J_{n-2}^{(8g)}]_{\searrow 8g-n+2}$ 
are zero and $[E_0^+ \pm E_0^- J_{n-2}^{(8g)}]_{\searrow 8g-n+2}$ 
is lower triangular with diagonal $(C_{-g},\cdots,C_{-g})$. 

By the proof of Lemma \ref{lem_0314_2}, 
\[
\det(E_0^+ \pm E_{n}^{\sharp}J)
= C_{-g}^{4g-n} 
\det([E_0^+]_{\nwarrow  2g+n} + [e_{1,n}^\sharp]_{\nearrow 2g+n})
\det([E_0^+]_{\nwarrow  2g+n} - [e_{1,n}^\sharp]_{\nearrow 2g+n}), 
\]
since $[e_{1,n}^\sharp]_{\nearrow 2g+n} = [e_{2,n}^\sharp]_{\swarrow  2g+n}$. 
We observe that 
\[
[E_0^+]_{\nwarrow  2g+n} \pm [e_{1,n}^\sharp]_{\nearrow 2g+n} 
= [E_0^+ \pm E_0^- J_n^{(8g)}]_{\nwarrow  2g+n}. 
\]
By Lemma \ref{0418_10}, 
\[
\aligned
\text{(LHS of \eqref{0331_1})}
& = C_{-g}^{16g-2n+2} \det([E_0^+ \pm E_0^-J_{n-2}^{(8g)}]_{\nwarrow n-2}) \\
& \quad \times \det([E_0^+ + E_0^- J_n^{(8g)}]_{\nwarrow n})\det([E_0^+ - E_0^- J_n^{(8g)}]_{\nwarrow n}).
\endaligned
\]
On the right-hand side of \eqref{0331_1}, we have 
\[
\det(E_0^+ \pm E_0^-J_{n}^{(8g)})
= C_{-g}^{8g-n} \det([E_0^+ \pm E_0^-J_{n}^{(8g)}]_{\nwarrow n})
\]
by applying Lemma \ref{mat_3} to the columns $(n+1,n+2,\cdots,8g)$ of $E_0^+ \pm E_0^-J_{n}^{(8g)}$. 
Therefore, it is sufficient to prove
\begin{equation} \label{0331_3}
\aligned
D_n^\pm & := \det[u_{n,1}^{\pm}~\cdots~\stackrel{n}{\stackrel{\smile}{u_{n,6g}^{\pm} \mp w}}~\cdots
~\stackrel{6g}{\stackrel{\smile}{u_{n,6g}^{\pm} \mp w}}~\cdots~u_{n,8g}^{\pm}] \\
& = C_{-g}^{8g-n+2} \det([E_0^+ \mp E_0^- J_n^{(8g)}]_{\nwarrow n})\det([E_0^+ \pm E_0^-J_{n-2}^{(8g)}]_{\nwarrow n-2}).
\endaligned
\end{equation}
Subtracting the $(6g-j)$-th column from the $(n-j)$-th column for $1 \leq j \leq n-1$, we obtain 
\[
\aligned
D_n^\pm = 
\det [v_{n,1}^{\pm}~\cdots~v_{n,n-1}^\pm~\stackrel{n}{\stackrel{\smile}{u_{n,6g}^{\pm} \mp w}}~\cdots
~\stackrel{6g}{\stackrel{\smile}{u_{n,6g}^{\pm} \mp w}}~\cdots~u_{n,8g}^{\pm}],
\endaligned
\]
where $v_{n,n-j}^\pm=u_{n,n-j}^{\pm} - u_{n,6g-j}$ for $1 \leq j \leq n-1$. 
Successively adding the $(n-j+1)$-th column to the $(6g-j-1)$-th column for $1 \leq j \leq n-2$, 
we obtain 
\[
\scalebox{0.9}{$
\aligned
D_n^\pm = 
\det [v_{n,1}^{\pm}~\cdots~v_{n,n-1}^\pm~\stackrel{n}{\stackrel{\smile}{u_{n,6g}^{\pm} \mp w}}~\cdots
~v_{n,6g-n+1}~\cdots~v_{n,6g-2}~\cdots~
~\stackrel{6g}{\stackrel{\smile}{u_{n,6g}^{\pm} \mp w}}~\cdots~u_{n,8g}^{\pm}],
\endaligned
$}
\]
where $v_{n,6g-j-1}^\pm=u_{n,6g-j-1}^{\pm} - v_{n,n-j+1}^\pm$ for $1 \leq j \leq n-2$. 
Note that the columns $(n+1,\cdots,6g-n)$ are not changed.
Put
\[
\scalebox{0.9}{$
M_{n}^{\pm} := \bigl[ v_{n,1}^{\pm}~\cdots~v_{n,n-1}^\pm~\stackrel{n}{\stackrel{\smile}{u_{n,6g}^{\pm} \mp w}}~\cdots
~v_{n,6g-n+1}~\cdots~v_{n,6g-2}~\cdots~
~\stackrel{6g}{\stackrel{\smile}{u_{n,6g}^{\pm} \mp w}}~\cdots~u_{n,8g}^{\pm} \bigr].
$}
\]
By the deformation of the matrix, we observe that 
\[
[M_{n}^{\pm}]_{\nwarrow 4g} = [E_0^+ \mp E_0^-J_{n}^{(8g)}]_{\nwarrow 4g}.
\]
Therefore, applying Lemma \ref{mat_3} to the columns $(n+1,n+2,\cdots,4g)$, we obtain 
\[
D_n^\pm 
= C_{-g}^{4g-n}\det([E_0^+ \mp E_0^-J_{n}^{(8g)}]_{\nwarrow n}) \det([M_{n}^{\pm}]_{\searrow 4g}).
\]
Here we note that 
\[
[M_{n}^{\pm}]_{\searrow 2g+n}  = [E_0^+ \mp E_0^-J_{n-2}^{(8g)}]_{\nwarrow 2g+n}
\]
and that the columns $(4g+1,\cdots,6g-n)$ 
are not changed by the deformation of matrix. 
Then, applying Lemma \ref{mat_3} to the rows $(1,\cdots, 2g-n)$ of $[M_{n}^{\pm}]_{\searrow 2g+n}$, we obtain 
\[
\aligned
D_n^\pm 
&= C_{-g}^{6g-2n}\det([E_0^+ \mp E_0^-J_{n}^{(8g)}]_{\nwarrow n}) \det([E_0^+ \mp E_0^-J_{n-2}^{(8g)}]_{\nwarrow 2g+n}) \\
&= C_{-g}^{8g-2n+2}\det([E_0^+ \mp E_0^-J_{n}^{(8g)}]_{\nwarrow n}) \det([E_0^+ \mp E_0^-J_{n-2}^{(8g)}]_{\nwarrow n-2}), 
\endaligned
\]
and \eqref{0331_3} is proved. 
\end{proof}

\begin{lemma} \label{0330_22} 
\[
\aligned
\det(E_0^{+} \pm E_0^{-} J_{n}^{(8g)})
& = C_{-g}^{4g}\cdot \det(e_0^{+} \pm e_0^{-} J_{n}^{(4g)}) \\
& = C_{-g}^{6g-1}\cdot \det(E^{+} \pm E^{-} J_{n}^{(2g+1)})
\endaligned
\]
for every $1 \leq n \leq 2g$. 
\end{lemma}
\begin{proof}
If $1 \leq n \leq 2g$, we have 
\[
E_0^{+} \pm E_0^{-} J_{n}^{(8g)}
=
\left[ 
\begin{array}{cc|cc}
 & & &\\
 \multicolumn{2}{c|}{\raisebox{1.5ex}[0pt]{$e_0^{+} \pm e_0^{-}J_{n}^{(4g)}$}} 
& \multicolumn{2}{c}{\raisebox{1.5ex}[0pt]{\BigZero}}\\ \hline
 &  & & \\
\multicolumn{2}{c|}{\raisebox{1.5ex}[0pt]{\BigZero}} & \multicolumn{2}{c}{\raisebox{1.5ex}[0pt]{$e_0^{+}$}}  \\
\end{array}
\right].
\]
Therefore, 
\[
\det(E_0^{+} \pm E_0^{-}J_{n}^{(8g)})
= C_{-g}^{4g} \det(e_0^{+} \pm e_0^{-}J_{n}^{(4g)}),
\]
since $\det(e_0^{+})=C_{-g}^{4g}$.
Further, if $1 \leq n \leq 2g$, we have
\[
e_0^{+} \pm e_0^{-}J_{n}^{(4g)}
=
\left[ 
\begin{array}{cc|cc}
 & & &\\
 \multicolumn{2}{c|}{\raisebox{1.5ex}[0pt]{$E^+ \pm E^-J_{n}^{(2g+1)}$}} 
& \multicolumn{2}{c}{\raisebox{1.5ex}[0pt]{\BigZero}}\\ \hline
 &  & & \\
\multicolumn{2}{c|}{\raisebox{1.5ex}[0pt]{$\ast$}} & \multicolumn{2}{c}{\raisebox{1.5ex}[0pt]{$[e_0^+]_{\searrow (2g-1)}$}}  \\
\end{array}
\right].
\]
Hence 
\[
\det(E_0^{+} \pm E_0^{-} J_{n}^{(8g)})
= C_{-g}^{6g-1} \det(E^{+} \pm E^{-} J_{n}^{(2g+1)}),
\]
and the proof is complete.
\end{proof}

\begin{lemma} \label{0418_7} 
\[
\det(E_0^{+} \pm E_{1,1}^{\sharp} J)=\det(E_0^+)
\] 
and 
\[
\det(E_0^+ \pm E_{n,1}^{\sharp}J)
=\det(E_0^+ \pm E_{n-1}^{\sharp}J) 
\]
for $2 \leq n \leq 2g$.
\end{lemma}
\begin{proof}
By the proof of Lemma \ref{lem_0314_2}, 
\begin{equation}\label{0418_8}
\det(E_0^+ \pm E_{n-1}^{\sharp} J) = C_{-g}^{4g-2n+2}\det(A_{n-1}+B_{n-1})\det(A_{n-1}-B_{n-1})
\end{equation}
for 
$A_{n-1} = [e_0^{+}]_{\nwarrow 2g+n-1}$ and $B_{n-1} = [e_{2,n-1}^{\sharp}]_{\swarrow  2g+n-1}$.
On the other hand, 
applying Lemma \ref{mat_3} to the columns $(2g+n+1,\cdots,4g)$  of $E_0^{+} \pm E_{n,1}^{\sharp}J$, 
and then applying Lemma \ref{mat_3} to the rows $(2g+n+1,4g)$ of the resulting matrix, 
we obtain
\[
\det( E_0^{+} \pm E_{n,1}^{\sharp}J) 
= C_{-g}^{4g-2n}\det \begin{bmatrix} A_n & \pm B_{n-1}^\prime \\ \pm B_n & A_n \end{bmatrix},
\]
where $B_{n-1}^\prime = [e_{2,n-1}^\sharp]_{\swarrow  2g+n}$. 
Applying Lemma \ref{mat_3} to the first row of 
\[
\begin{bmatrix} A_n & \pm B_{n-1}^\prime \\ \pm B_n & A_n \end{bmatrix}
\] 
and then applying Lemma \ref{mat_3} to the $(4g+2n-1)$-th column of the resulting matrix, 
we obtain 
\[
\aligned
\det \begin{bmatrix} A_n & \pm B_{n-1}^\prime \\ \pm B_n & A_n \end{bmatrix}
&= C_{-g}^{2}\det \begin{bmatrix} A_{n-1} & \pm B_{n-1} \\ \pm B_{n-1} & A_{n-1} \end{bmatrix} \\
&= C_{-g}^{2}\det(A_{n-1}+B_{n-1})\det(A_{n-1}-B_{n-1}).
\endaligned
\]
Hence 
\begin{equation}\label{0418_9}
\det( E_0^{+} \pm E_{n,1}^{\sharp}J) 
= C_{-g}^{4g-2n+2}\det(A_{n-1}+B_{n-1})\det(A_{n-1}-B_{n-1}). 
\end{equation}
We complete the proof by comparing \eqref{0418_8} and \eqref{0418_9}. 
\end{proof}

\begin{proof}[Proof of Proposition \ref{prop_0313_1}]
Fix one of $0 \leq n \leq 2g-1$. 
Taking the limit $a \to q^{n/2}$ in \eqref{0313_1}, we have 
$X(k)=q^{(k-n/2)iz}$ and $Y(l)=q^{-(l-n/2)iz}$.
Therefore, $X(k)=Y(l)$ as a function of $z$ if and only if $n = k+l$. 
Taking the limit $a \to q^{n/2}$ 
on the left-hand side of \eqref{LE1}, we obtain 
\begin{equation} \label{0330_13}
\scalebox{0.84}{$
\aligned
({\mathsf E} \pm {\mathsf E}^{\sharp}{\mathsf J}{\mathsf P}_n) \phi_n^{\pm} (q^{n/2},z) 
& = \pm v_n^{\pm}(0)\sum_{m=-g}^g C_{-m} X(m) + u_n^{\pm}(0) \sum_{m=-g}^g \left( C_m X(m) \pm C_{-m} X(n+m) \right) \\
& \quad + \sum_{m=-g}^g \sum_{k=1}^{\infty} C_m (  u_n^{\pm}(k) + v_n^{\pm}(n-k) ) X(k+m) \\
& \quad  \pm \sum_{m=-g}^g \sum_{k=1}^{n-1} C_{-m} (u_n^{\pm}(k) + v_n^{\pm}(n-k)) X(n-k+m)
\endaligned
$}
\end{equation}
and  
\begin{equation} \label{0330_14}
\scalebox{0.9}{$
\aligned
({\mathsf E} \pm {\mathsf E}^{\sharp}{\mathsf J}{\mathsf P}_{n+1}) \phi_{n+1}^{\pm} (q^{n/2},z)  
& = \sum_{m=-g}^g \sum_{k=0}^{\infty} C_m ( u_{n+1}^{\pm}(k) + v_{n+1}^{\pm}(n-k) ) X(k+m) \\
& \quad  \pm \sum_{m=-g}^g \sum_{k=0}^{n} C_{-m} ( u_{n+1}^{\pm}(k) + v_{n+1}^{\pm}(n-k) )X(n-k+m). 
\endaligned
$}
\end{equation}
By \eqref{0330_13}, $({\mathsf E} \pm {\mathsf E}^{\sharp}{\mathsf J}{\mathsf P}_n) \phi_n^{\pm} = {\mathsf E}^{\sharp} X(0)$ 
implies 
\begin{equation} \label{0330_15}
\aligned
(1 \mp v_n^{\pm}(0)) \sum_{m=-g}^g C_{-m} & X(m) 
 = u_n^{\pm}(0) \sum_{m=-g}^g \left( C_m X(m) \pm C_{-m} X(n+m) \right) \\
& + \sum_{m=-g}^g \sum_{k=1}^{\infty} C_m (  u_n^{\pm}(k) + v_n^{\pm}(n-k) ) X(k+m) \\
& \pm \sum_{m=-g}^g \sum_{k=1}^{n-1} C_{-m} (u_n^{\pm}(k) + v_n^{\pm}(n-k)) X(n-k+m). 
\endaligned
\end{equation}
Let we write 
\[
\scalebox{0.9}{$
L\Phi_n^{\pm}=\begin{bmatrix} 
u_n^{\pm}(0) \\[4pt] 
u_n^{\pm}(1) + v_n^{\pm}(n-1) \\[4pt] 
u_n^{\pm}(2) + v_n^{\pm}(n-2) \\
\vdots \\ 
u_n^{\pm}(4g-1) + v_n^{\pm}(n-4g+1) 
\end{bmatrix}, \quad
R\Phi_{n+1}^{\pm}=\begin{bmatrix} 
u_{n+1}^{\pm}(0) + v_{n+1}^{\pm}(n) \\[4pt] 
u_{n+1}^{\pm}(1) + v_{n+1}^{\pm}(n-1) \\[4pt] 
u_{n+1}^{\pm}(2) + v_{n+1}^{\pm}(n-2) \\
\vdots \\
u_{n+1}^{\pm}(4g-1) + v_{n+1}^{\pm}(n-4g+1)
\end{bmatrix}.
$}
\]
Then, comparing coefficients of $X(k)$ ($-g \leq k \leq 3g-1$) in  \eqref{0330_15}, we obtain
\[
(e_0^{+} \pm e_1^{-}J_{n+1}^{(4g)})L\Phi_{n}^{\pm}
= (1 \mp v_n^{\pm}(0))
e_0^{-}\chi_{4g}.
\]
Note that $\det(e_0^{+} \pm e_1^{-} J_{n+1}^{(4g)})\not=0$ 
by Lemmas \ref{0319_3}, \ref{0418_10}, and \ref{0330_22},  
since $\mathsf{I} \pm \Theta{\mathsf J}{\mathsf P}_n$ are 
invertible on $W_{a,n}$. 
Therefore, 
\begin{equation} \label{0330_16}
\aligned
L\Phi_{n}^{\pm}
= \frac{1 \mp v_n^{\pm}(0)}{\det(e_0^{+} \pm e_1^{-} J_{n+1}^{(4g)})} 
{\rm adj}(e_0^{+} \pm e_1^{-} J_{n+1}^{(4g)})
e_0^{-} \chi_{4g}.
\endaligned
\end{equation}
On the other hand, comparing coefficients of $X(k)$ ($-g \leq k \leq 3g-1$) in 
$({\mathsf E} \pm {\mathsf E}^{\sharp}{\mathsf J}{\mathsf P}_{n+1}) \phi_{n+1}^{\pm} = {\mathsf E}^{\sharp} X(0)$ 
together with \eqref{0330_14} , we obtain
\[
(e_0^{+} \pm e_0^{-} J_{n+1}^{(4g)})R\Phi_{n+1}^{\pm}
=
e_0^{-}\chi_{4g}.
\]
Note that $\det(e_0^{+} \pm e_0^{-}J_{n+1}^{(4g)})\not=0$ 
by Lemma \ref{0418_10} and \ref{0330_22}, 
since $\mathsf{I} \pm \Theta{\mathsf J}{\mathsf P}_{n+1}$ are 
invertible on $W_{a,n+1}$. 
This implies 
\[
\det(e_0^{+} \pm e_0^{-} J_{n+1}^{(4g)})R\Phi_{n+1}^{\pm}
= {\rm adj}(e_0^{+} \pm e_0^{-} J_{n+1}^{(4g)})
e_0^{-} \chi_{4g}.
\]  
Then, by Lemma \ref{0319_1}, we have
\begin{equation} \label{0330_17}
\det(e_0^{+} \pm e_0^{-} J_{n+1}^{(4g)}) R\Phi_{n+1}^{\pm}
= {\rm adj}(e_0^{+} \pm e_1^{-} J_{n+1}^{(4g)})
e_0^{-} \chi_{4g}.
\end{equation}
By \eqref{0330_16}, \eqref{0330_17}, and Lemma \ref{0319_3}, we obtain 
\begin{equation} \label{0319_4}
L\Phi_{n}^{\pm}
= 
\frac{\det(e_0^{+} \pm e_0^{-} J_{n+1}^{(4g)})}{ \det(e_0^{+} \pm e_0^{-} J_{n-1}^{(4g)})}
(1 \mp v_n^{\pm}(0))
R\Phi_{n+1}^{\pm},
\end{equation}
where we understand that 
$\det(e_0^{+} \pm e_0^{-} J_{n-1}^{(4g)})=\det(e_0^{+})$ if $n=0,1$.
Passing to the limit $a \to q^{n/2}$ in the equalities in the proof of Lemma \ref{co2}, we have
\begin{equation} \label{0330_19}
\aligned
({\mathsf I} \pm {\mathsf J}){\mathsf E}\, \phi_{n}^{\pm} (q^{n/2},z) 
& = (1 \mp v_n^{\pm}(0))\sum_{m=-g}^{g} (C_{-m}X(m) \pm C_{m}X(n-m)) \\ 
& \quad - u_n^{\pm}(0) \sum_{m=-g}^{g} (C_{m} X(m) \pm C_{-m}X(n+m)) \\
& \quad - \sum_{m=-g}^{g} C_{m}\sum_{k=1}^{n-1} (u_n^{\pm}(k) + v_n^{\pm}(n-k))X(k+m) \\
& \quad  \mp \sum_{m=-g}^{g} C_{-m}\sum_{k=1}^{n-1} (u_n^{\pm}(k) + v_n^{\pm}(n-k))X(n-k+m)
\endaligned
\end{equation}
and 
\begin{equation} \label{0330_20}
\aligned
({\mathsf I} \pm {\mathsf J}){\mathsf E}\, \phi_{n+1}^{\pm}(q^{n/2},z) 
&= \sum_{m=-g}^{g} (C_{-m}X(m) \pm C_{m}X(n+m)) \\ 
& - \sum_{m=-g}^{g} C_{m}\sum_{k=0}^{n} (u_{n+1}^{\pm}(k) + v_{n+1}^{\pm}(n-k))X(k+m) \\
& \mp \sum_{m=-g}^{g} C_{-m}\sum_{k=0}^{n} (u_{n+1}^{\pm}(k) + v_{n+1}^{\pm}(n-k)) X(n-k+m). 
\endaligned
\end{equation}
Applying \eqref{0319_4} to \eqref{0330_19}, we obtain
\[
\aligned
({\mathsf I} & \pm {\mathsf J}){\mathsf E}\, \phi_{n}^{\pm} (q^{n/2},z)/(1 \mp v_n^{\pm}(0)) \\ 
& = \left(
1 \pm (u_{n+1}^{\pm}(n) + v_{n+1}^{\pm}(0)) 
\frac{\det(e_0^{+} \pm e_0^{-} J_{n+1}^{(4g)})}{ \det(e_0^{+} \pm e_0^{-} J_{n-1}^{(4g)})}
\right) \sum_{m=-g}^{g} ( C_{-m}X(m) \pm C_{m}X(n+m))  \\
& \quad -  
\frac{\det(e_0^{+} \pm e_0^{-}J_{n+1}^{(4g)})}{\det(e_0^{+} \pm e_0^{-} J_{n-1}^{(4g)})}
\sum_{m=-g}^{g} C_{m}\sum_{k=0}^{n} (u_{n+1}^{\pm}(k) + v_{n+1}^{\pm}(n-k))X(k+m) \\
& \quad \mp 
\frac{\det(e_0^{+} \pm e_0^{-} J_{n+1}^{(4g)})}{\det(e_0^{+} \pm e_0^{-} J_{n-1}^{(4g)})}
\sum_{m=-g}^{g} C_{-m}\sum_{k=0}^{n} (u_{n+1}^{\pm}(k) + v_{n+1}^{\pm}(n-k))X(n-k+m), 
\endaligned
\]
and then, therefore, by \eqref{0330_20} and Lemma \ref{0330_21}, 
\[
({\mathsf I} \pm {\mathsf J}){\mathsf E}\, \phi_{n}^{\pm} (q^{n/2},z)
=\frac{\det(e_0^{+} \pm e_0^{-}J_{n+1}^{(4g)})}{ \det(e_0^{+} \pm e_0^{-} J_{n-1}^{(4g)})}
 (1 \mp v_n^{\pm}(0)) ({\mathsf I} \pm {\mathsf J}){\mathsf E}\, \phi_{n+1}^{\pm} (q^{n/2},z).
\]
By \eqref{LE2} and Cramer's formula, we have
\[
1 \mp v_n^{\pm}(0)=\frac{\det(E_{0}^{+} \pm E_{n,1}^{\sharp}J)}{\det(E_0^{+} \pm E_n^{\sharp}J)}.
\]
Hence we obtain \eqref{0331_5} and \eqref{0331_6} by Lemma \ref{0330_22} and Lemma \ref{0418_7}. 
\end{proof}

\begin{proposition} \label{0419_1}
For $1 \leq n \leq 2g$, 
define the column vectors $A_n^{\ast\ast}$ and $B_n^{\ast\ast}$ of length $2g-n+1$ by
\begin{equation*}
\aligned
A_n^{\ast\ast}=A_n^{\ast\ast}(\underline{C}) 
 &= {}^{t}\begin{bmatrix} A_n^{\ast}(2g+1) & \cdots &  A_n^{\ast}(n+1) \end{bmatrix}, \\
B_n^{\ast\ast}=B_n^{\ast\ast}(\underline{C}) 
 &= {}^{t}\begin{bmatrix} B_n^{\ast}(2g+1) & \cdots &  B_n^{\ast}(n+1) \end{bmatrix}, 
\endaligned
\end{equation*}
where $v(j)$ means the $j$th component of a column vector $v$. 
Define the column vectors $\Omega_n$ of length $4g-2n+2$ by 
\begin{equation} \label{0501_2}
\Omega_n = {}^t 
\begin{bmatrix} 
(\prod_{j=1}^{n} \alpha_j) \, {}^t\! A_n^{\ast\ast} & 
(\prod_{j=1}^{n} \beta_j) \, {}^t B_n^{\ast\ast} \end{bmatrix}.
\end{equation}
Then 
\begin{equation} \label{0427_1}
P_{2g-(n+1)}(\gamma_{n+1}) \Omega_{n+1}=Q_{2g-(n+1)} \Omega_{n}
\end{equation}
for every $0 \leq n \leq 2g-1$, where $\gamma_n$ is of \eqref{0425_3}.
\end{proposition}
\begin{proof}
We have 
\begin{equation} \label{0423_5}
\Omega_0= 
{}^{t} 
\begin{bmatrix} C_{-g} & C_{-g+1}& \cdots & C_{g-1} & C_{g} & C_{-g} & C_{-g+1}& \cdots & C_{g-1} & C_{g}  \end{bmatrix},
\end{equation}
since $A_0^\ast=(I+J)E_0^- \chi_{8g}$ 
and $B_0^\ast=(I-J)E_0^- \chi_{8g}$ by $E_0^\sharp J =0$, 
where $I=I^{(8g)}$ and  $J=J^{(8g)}$. 
By Proposition \ref{prop_0313_1} , we have 
\begin{equation} \label{0320_2}
\aligned
({\mathsf I} + {\mathsf J}){\mathsf E}  \phi_n^{+}(q^{n/2},z) 
& = \alpha_{n+1}({\mathsf I} + {\mathsf J}){\mathsf E}  \phi_{n+1}^{+}(q^{n/2},z), \\
({\mathsf I} - {\mathsf J}){\mathsf E}  \phi_n^{-}(q^{n/2},z) 
& = \beta_{n+1}({\mathsf I} - {\mathsf J}){\mathsf E}  \phi_{n+1}^{-}(q^{n/2},z).
\endaligned 
\end{equation}
On the other hand, writing 
\[
({\mathsf I} \pm {\mathsf J}){\mathsf E}  \phi_n^{\pm}(a,z)
= \sum_{k=-g+n}^{g} r_n^{\pm}(k)(X(k) \pm Y(k)),
\]
we have 
\begin{equation} \label{0320_1}
\aligned
({\mathsf I} \pm {\mathsf J}){\mathsf E}  \phi_n^{\pm}(q^{n/2},z)
& = \sum_{k=-g+n}^{g} (r_n^{\pm}(k) \pm  r_n^{\pm}(n-k)) X(k) \\
& = \frac{1}{2}\sum_{k=-g+n}^{g} (r_n^{\pm}(k) \pm  r_n^{\pm}(n-k)) (X(k) \pm X(n-k)), \\
({\mathsf I} \pm {\mathsf J}){\mathsf E}  \phi_{n+1}^{\pm}(q^{n/2},z)
& = \sum_{k=-g+n+1}^{g-1} (r_{n+1}^{\pm}(k) \pm r_{n+1}^{\pm}(n-k)) X(k) \\
& \quad + r_{n+1}^{\pm}(g)(X(g) \pm X(n-g)) \\
& = \frac{1}{2}\sum_{k=-g+n+1}^{g-1} (r_{n+1}^{\pm}(k) \pm r_{n+1}^{\pm}(n-k)) (X(k) \pm X(n-k)) \\
& \quad + r_{n+1}^{\pm}(g)(X(g) \pm X(n-g)), 
\endaligned
\end{equation}
since 
\[
\scalebox{0.9}{$
\aligned
({\mathsf I} \pm {\mathsf J}){\mathsf E}  \phi_n^{\pm}(q^{n/2},z) 
& = u_n^{\pm}(0)  \sum_{m=-g}^g C_m (X(m) \pm  X(n-m)) \\ 
& \quad + \sum_{m=-g}^g\sum_{k=1}^{\infty}C_m (u_n^{\pm}(k)+v_n^{\pm}(n-k))(X(k+m)\pm X(n-k-m))
\endaligned
$}
\]
and 
\[
\scalebox{0.9}{$
\aligned
({\mathsf I} \pm {\mathsf J}){\mathsf E}  \phi_{n+1}^{\pm}(q^{n/2},z)  
& = \sum_{m=-g}^g \sum_{k=0}^{\infty}C_m (u_{n+1}^{\pm}(k)+v_{n+1}^{\pm}(n-k))(X(k+m)\pm X(n-k-m))
\endaligned
$}
\]
by a rearrangement of \eqref{0330_19} and \eqref{0330_20}. 
Noting
\[
r_n^+(k)=A_n^\ast(k+g+1), \quad r_n^-(k)=B_n^\ast(k+g+1), 
\]
we obtain \eqref{0427_1} for $0 \leq n \leq 2g-1$ by \eqref{0320_2} and \eqref{0320_1}. 
\end{proof}

\begin{proposition} \label{0501_3}
Let $\underline{C}=(C_g,C_{g-1},\cdots,C_{-g}) \in \R^\ast \times \R^{2g-1} \times \R^\ast$ 
and define $E(z)$ by \eqref{0418_2}. 
Suppose that $\det D_n(\underline{C})\not=0$ for every $1 \leq n \leq 2g$.  
Then the pair of functions $(A(a,z),B(a,z))$ of \eqref{0330_10} satisfies the boundary conditions \eqref{0418_4}.
\end{proposition}
\begin{proof}
First we prove the first half of \eqref{0418_4}. 
By \eqref{0427_1}, 
$\Omega_{1}=P_{2g-1}(\gamma)^{-1}Q_{2g-1}\cdot \Omega_0$ 
for $\gamma=\gamma_1(\underline{C})$.  
Therefore, 
$
\Omega_{1} 
= {}^t 
\begin{bmatrix}
{}^t {\mathbf a}_{1} & {}^t {\mathbf b}_{1}
\end{bmatrix}
$
with 
\begin{equation} \label{0503_10}
{\mathbf a}_{1}
= \frac{1}{2}\begin{bmatrix}
2(C_{-g}+C_{g}) \\
C_{-(g-1)}+C_{g-1} + \gamma (C_{-(g-1)}-C_{g-1}) \\
C_{-(g-2)}+C_{g-2} + \gamma (C_{-(g-2)}-C_{g-2}) \\
\vdots \\
C_{-1}+C_{1} + \gamma (C_{-1}-C_{1}) \\
2 C_0 \\
C_{-1}+C_{1} - \gamma (C_{-1}-C_1) \\
\vdots \\
C_{-(g-2)}+C_{g-2} - \gamma (C_{-(g-2)}-C_{g-2}) \\
C_{-(g-1)}+C_{g-1} - \gamma (C_{-(g-1)}-C_{g-1}) \\
\end{bmatrix} 
\end{equation}
and 
\begin{equation} \label{0503_11}
{\mathbf b}_{1}
= \frac{1}{2}
\begin{bmatrix}
2(C_{-g}-C_{g}) \\
C_{-(g-1)}-C_{g-1} + \gamma^{-1}(C_{-(g-1)}+C_{g-1}) \\
C_{-(g-2)}-C_{g-2} +  \gamma^{-1}(C_{-(g-2)}+C_{g-2}) \\
\vdots \\
C_{-1}-C_{1} + \gamma^{-1}(C_{-1}+C_{1}) \\
2\gamma^{-1} C_0 \\
C_{-1}-C_{1} + \gamma^{-1}(C_{-1}+C_1) \\
\vdots \\
-(C_{-(g-2)}-C_{g-2}) + \gamma^{-1}(C_{-(g-2)}+C_{g-2})  \\
-(C_{-(g-1)}-C_{g-1}) + \gamma^{-1}(C_{-(g-1)}+C_{g-1}) \\
\end{bmatrix}.
\end{equation}
Therefore,  
\[
\aligned
A(1,z) 
& = \sum_{k=1}^{g} (C_{-m}+C_{m}) \cos(z\log q^m) + C_0=A(z),   \\
B(1,z) 
& = \sum_{k=1}^{g} (C_{-m}-C_{m}) \sin(z \log q^m)=B(z)
\endaligned
\]
by \eqref{0330_7} and definition \eqref{0501_2}; the second equalities of right-hand sides 
are immediate consequence of definitions \eqref{0418_2}, \eqref{def_A} and \eqref{def_B}. 

Finally, we prove the latter half of \eqref{0418_4}. 
By definition \eqref{0330_10}, it is sufficient to prove
\begin{equation} \label{0418_5}
\lim_{a \nearrow q^{g}}\begin{bmatrix}A_{2g}(a,z) \\ B_{2g}(a,z) \end{bmatrix}=\begin{bmatrix} E(0) \\ 0 \end{bmatrix}.
\end{equation}
By definition \eqref{0418_6} and Lemma \ref{0320_4}, 
\[
A_{2g}^\ast(a,z) = C_{-g} \cos(z \log(q^{2g}/a)), \quad 
B_{2g}^\ast(a,z) = C_{-g} \sin(z \log(q^{2g}/a)). 
\]
Therefore, $\lim_{a \nearrow q^{g}} B_{2g}(a,z) =0$ and 
\[
\lim_{a \nearrow q^{g}} A_{2g}(a,z) = \alpha_1 \cdots \alpha_{2g} \cdot C_{-g}
\]
for fixed $z \in \C$. 
By \eqref{0330_11} and \eqref{0330_7}, 
\[
A_{2g}(a,z) = \alpha_1 \cdots \alpha_{2g} \cdot C_{-g} \cos(z \log(q^{2g}/a)). 
\]
On the other hand, 
\begin{equation} \label{0503_1}
\alpha_1 \cdots \alpha_{2g} \cdot C_{-g} = \Omega_{2g}(1)
\end{equation}
by Proposition \ref{0419_1}.
Therefore, it is sufficient to prove $\Omega_{2g}(1)=E(0)$. 
To prove this, we put 
\[
S_n=P_0^{-1}Q_0 P_1(\gamma_{2g-1})^{-1}Q_1 \cdots P_n(\gamma_{2g-n})^{-1}Q_n \quad (n=0,1,2,\cdots). 
\]
The the size of $S_n$ is $2 \times (2n+4)$ by definitions of $P_k(m_k)$ and $Q_k$. 
Applying \eqref{0427_1} repeatedly, we obtain  
\[
\Omega_{2g}(1)=(\text{the first row of $S_{2g-1}$})\cdot \Omega_0.
\] 
On the other hand, 
\[
S_0=P_0^{-1}Q_0=\begin{bmatrix} 1 & 1 & 0 & 0 \\ 0 & 0 & 1 & -1\end{bmatrix}, 
\]
and we find that the first row of $S_n$ has the form  
\[
\underbrace{1~1~\cdots~1}_{n+2}~\underbrace{0~0~\cdots~0}_{n+2}
\]
by induction using \eqref{0426_1}. Hence 
\begin{equation} \label{0503_2}
\Omega_{2g}(1)=(\underbrace{1~1~\cdots~1}_{2g+1}~\underbrace{0~0~\cdots~0}_{2g+1})\cdot \Omega_0= \sum_{k=-g}^{g} C_{m} = E(0)
\end{equation}
by \eqref{0418_2} and \eqref{0423_5}.
\end{proof}

\noindent
{\bf Proof of Theorem \ref{thm_01}.} 
As a summary of the above results, we obtain the following theorem 
which implies Theorem \ref{thm_01}. 

\begin{theorem} \label{thm_03}
Let $\underline{C}=(C_g,C_{g-1},\cdots,C_{-g}) \in \R^\ast \times \R^{2g-1} \times \R^\ast$ 
and define $E(z)$ by \eqref{0418_2}. 
Suppose that 
$\det D_n(\underline{C})\not=0$ for every $1 \leq n \leq 2g$.  
Then, for arbitrary fixed $q>1$, 
\begin{enumerate}
\item $\det(E^{+} \pm E^{-}J_{n}) \not=0$ for every $1 \leq n \leq 2g$, 
\item $A(a,z)$ and $B(a,z)$ are well-defined and continuous on $[1,q^{g})$  with respect to $a$, 
\item $A(a,z)$ and $B(a,z)$ are differentiable on $(q^{(n-1)/2},q^{n/2})$ with respect to $a$ for every $1 \leq n \leq 2g$, 
\item the left-sided limit $\lim_{a \nearrow q^{n/2}}(A(a,z), B(a,z))$ 
defines entire functions of $z$ for every $1 \leq n \leq 2g$, 
\item the pair of functions $(A(a,z),B(a,z))$ defined in \eqref{0330_10} satisfies the system \eqref{0330_1}, 
\item the pair of functions $(A(a,z),B(a,z))$ satisfies the boundary condition \eqref{0418_4}.
\end{enumerate}
\end{theorem}
\begin{proof}
(1) is a consequence of Lemmas \ref{lem0329_1}, \ref{0501_1}, and \ref{0330_22}. 
(2), (3) and (4) are consequences of definitions \eqref{0330_11}, \eqref{0330_10}, formula \eqref{0330_7} and Proposition \ref{prop_0313_1}. 
(5) is a consequence of Proposition \ref{0330_4}, definitions \eqref{0330_10} and \eqref{0330_11}. 
In fact, 
\[
-a \frac{d}{da}
\begin{bmatrix} A_n(a,z) \\ B_n(a,z) \end{bmatrix}
= z \begin{bmatrix} 0 & -1 \\ 1 & 0 \end{bmatrix} 
\begin{bmatrix} \gamma_n^{-1} & 0 \\ 0 &  \gamma_n \end{bmatrix}
\begin{bmatrix} A_n(a,z) \\ B_n(a,z) \end{bmatrix}
\]
for every $q^{(n-1)/2} \leq a < q^{n/2}$ and $1 \leq n \leq 2g$ by Proposition \ref{0330_4}, 
where $\gamma_n$ is of \eqref{0425_3}. 
This implies \eqref{0330_1}. 
(6) is a consequence of Proposition \ref{0501_3}. 
\end{proof}

\section{Proofs of Theorems \ref{thm_02}, \ref{thm_02_1} and \ref{thm_02_2}} \label{section_5}

We use the theory of de Branges spaces 
together with the theory of canonical systems to prove Theorems \ref{thm_02}, \ref{thm_02_1}, and \ref{thm_02_2}. 
De Branges spaces are a kind of reproducing kernel Hilbert spaces 
consisting of entire functions; see \cite{deBranges68, Dym70, Lagarias06, Remling02} for details.   
Firstly, we review two propositions from these theories 
as a preparation to the proofs of Theorems \ref{thm_02}, \ref{thm_02_1} and \ref{thm_02_2}. 
Note that their proofs presented below are the almost same as the argument 
in the literature on canonical systems and de Branges spaces; 
see, for example, the proof of equation (2.4) and Lemma 2.1, and Step 1 of the proof of Theorem 5.1 in \cite{Dym70}. 
However, we purposely give their detailed proofs 
to confirm that the positive semidefiniteness of the Hamiltonian, 
which is usually assumed in the theory of canonical systems, 
is not necessary for their proofs. 
\begin{proposition} \label{prop_601}
Let $\gamma(a)$ be as in \eqref{0423_6}, and let $1 \leq a_1< a_0 \leq q^g$. 
\begin{enumerate}
\item Assume that $\gamma(a)\not=0$ and $|\gamma(a)|< \infty$ for every $ 1 \leq  a \leq a_0$. 
Then there exists a $2 \times 2$ matrix-valued function $M(a_1,a_0;z)$ 
such that all entries are entire functions of $z$ and that satisfies 
\begin{equation} \label{prop_601_1}
\begin{bmatrix} 
A(a_1,z) \\ B(a_1,z)
\end{bmatrix} 
= 
M(a_1,a_0;z)
\begin{bmatrix}
A(a_0,z) \\ B(a_0,z)
\end{bmatrix},
\end{equation}
and $\det M(a_1,a_0;z)=1$. 
\item Assume that $\gamma(a)\not=0$ and $|\gamma(a)|<\infty$ for every $ 1 \leq  a < a_0$. 
Then the matrix-valued function $M(a_1,a;z)$ of {\rm (1)} is left-continuous as a function of $a$ and 
\begin{equation} \label{prop_601_2}
\begin{bmatrix} 
A(a_1,z) \\ B(a_1,z)
\end{bmatrix} 
=
\lim_{a \nearrow a_0} M(a_1,a;z)
\lim_{a \nearrow a_0}
\begin{bmatrix}
A(a,z) \\ B(a,z)
\end{bmatrix}
\end{equation}
holds as a vector-valued function of $z \in \C$.
\end{enumerate}
\end{proposition}
\begin{proof} 
(1) Put 
\[
J(a) = 
\begin{bmatrix}
0 & - \gamma(a) \\ \gamma(a)^{-1} & 0 
\end{bmatrix}.
\]
Then the system \eqref{0330_1} for $1 \leq a <a_0$
is written as 
\begin{equation} \label{system_2}
-a\frac{\partial}{\partial a}
\begin{bmatrix}
A(a,z) \\ B(a,z)
\end{bmatrix}
= z J(a)
\begin{bmatrix}
A(a,z) \\ B(a,z)
\end{bmatrix} \quad (1 \leq a < a_0 ,\,z \in \C).
\end{equation}
By assumption, both $\gamma(a)$ and $\gamma(a)^{-1}$ are integrable on $[a_1, a_0]$. 
Hence, 
\begin{equation} \label{601}
\scalebox{0.9}{$
\aligned
\, 
&
\begin{bmatrix}
A(a_1,z) \\ B(a_1,z)
\end{bmatrix} 
= 
\begin{bmatrix}
A(a_0,z) \\ B(a_0,z)
\end{bmatrix} 
+
z \int_{a_1}^{a_0}
J(t_1)
\begin{bmatrix}
A(t_1,z) \\ B(t_1,z)
\end{bmatrix} \frac{dt_1}{t_1} \\
& =
\begin{bmatrix}
A(a_0,z) \\ B(a_0,z)
\end{bmatrix} 
+
z \int_{a_1}^{a_0}
J(t_1) \frac{dt_1}{t_1}
\begin{bmatrix}
A(a_0,z) \\ B(a_0,z)
\end{bmatrix} 
+
z^2  \int_{a_1}^{a_0}\int_{t_1}^{a_0}
J(t_1)J(t_2)
\begin{bmatrix}
A(t_2,z) \\ B(t_2,z)
\end{bmatrix} \frac{dt_2}{t_2}\frac{dt_1}{t_1} \\
& =
\left(
I
+
z \int_{a_1}^{a_0}
J(t_1) \frac{dt_1}{t_1}
+
z^2  \int_{a_1}^{a_0}\int_{t_1}^{a_0}
J(t_1)J(t_2)
\frac{dt_2}{t_2}\frac{dt_1}{t_1} \right. \\
&\qquad \qquad \left.
+
z^3  \int_{a_1}^{a_0}\int_{t_1}^{a_0}\int_{t_2}^{a_0}
J(t_1)J(t_2)J(t_3)
\frac{dt_3}{t_3}\frac{dt_2}{t_2}\frac{dt_1}{t_1}
+
\cdots 
\right)
\begin{bmatrix}
A(a_0,z) \\ B(a_0,z)
\end{bmatrix},
\endaligned
$}
\end{equation}
where $I=I^{(2)}$. 
On the other hand, 
\[
J(t_1)\cdots J(t_k)
 = (-1)^{k'}
\begin{cases}
\begin{bmatrix}
0 & - \frac{\gamma(t_1)\gamma(t_3)\cdots \gamma(t_{k})}{\gamma(t_2)\gamma(t_4)\cdots \gamma(t_{k-1})} \\ 
\frac{\gamma(t_2)\gamma(t_4)\cdots \gamma(t_{k-1})}{\gamma(t_1)\gamma(t_3)\cdots \gamma(t_{k})} & 0 
\end{bmatrix} & \text{if $k=2k'+1$}, \\[12pt] 
\begin{bmatrix}
\frac{\gamma(t_1)\gamma(t_3)\cdots \gamma(t_{k-1})}{\gamma(t_2)\gamma(t_4)\cdots \gamma(t_{k})} & 0 \\ 
0 & \frac{\gamma(t_2)\gamma(t_4)\cdots \gamma(t_{k})}{\gamma(t_1)\gamma(t_3)\cdots \gamma(t_{k-1})}
\end{bmatrix} & \text{if $k=2k'$}.
\end{cases}
\]
Therefore, taking 
$C(a_0,a_1):=\sup \{\gamma(a),\gamma(a)^{-1};~ a \in [a_1,a_0]\}$ 
and by using the formula 
\[
\int_{a_1}^{a_0}\int_{t_1}^{a_0}\int_{t_2}^{a_0} \cdots \int_{t_{k-1}}^{a_0} 
1 \,\frac{dt_k}{t_k} \cdots \frac{dt_2}{t_2} \frac{dt_1}{t_1} = \frac{1}{k!}\left(\log \frac{a_0}{a_1}\right)^k, 
\]
we obtain 
\[
\left|
\left[\int_{a_1}^{a_0}\int_{t_1}^{a_0}\int_{t_2}^{a_0} \cdots \int_{t_{k-1}}^{a_0} 
J(t_1)\cdots J(t_k) \, \frac{dt_k}{t_k} \cdots \frac{dt_2}{t_2} \frac{dt_1}{t_1} \right]_{ij} \right| \leq  
\frac{1}{k!}C(a_0,a_1)^k\left(\log \frac{a_0}{a_1}\right)^k,
\]
for every $1 \leq i,j \leq 2$, where $[M]_{ij}$ means the $(i,j)$-entry of a matrix $M$. 
This estimate implies that 
the right-hand side of \eqref{601} 
converges absolutely and uniformly 
if $z$ lies in a bounded region. 

Suppose that $\gamma(a)=\gamma\not=0$ for $a_1 \leq a \leq a_0$. 
Then 
\[
\aligned 
I
+
z \int_{a_1}^{a_0}
J(t_1) 
\frac{dt_1}{t_1}
+
z^2  \int_{a_1}^{a_0}\int_{t_1}^{a_0} &
J(t_1)J(t_2)
\frac{dt_2}{t_2} \frac{dt_1}{t_1} \\
& +
z^3  \int_{a_1}^{a_0}\int_{t_1}^{a_0}\int_{t_2}^{a_0}
J(t_1)J(t_2)J(t_3)
\frac{dt_3}{t_3} \frac{dt_2}{t_2} \frac{dt_1}{t_1} 
+
\cdots 
\endaligned 
\]
is equal to 
\[
\begin{bmatrix}
\cos(z\log(a_0/a_1)) & - \gamma \sin(z\log(a_0/a_1)) \\ \frac{1}{\gamma} \sin(z\log(a_0/a_1)) & \cos(z\log(a_0/a_1))
\end{bmatrix}
\]
and hence \eqref{prop_601_1} holds, which is seen by taking this matrix as $M(a_1,a_0;z)$. 
Therefore, 
if we suppose that $\gamma(a)=\gamma_j\not=0$ on $[t_{j+1},t_j)$ 
for a partition $[a_1,a_0]=[a_1,t_{k-1}) \cup \cdots \cup [t_1,a_0]$ 
with $t_0=a_0$ and $t_k=a_1$,   
then we have \eqref{prop_601_1} by taking 
\[
\scalebox{0.9}{$
\aligned
M&(a_0,a_1;z) \\
:=&
\begin{bmatrix}
\cos(z\log(t_{k-1}/a_1)) & - \gamma_k \sin(z\log(t_{k-1}/a_1)) \\ \frac{1}{\gamma_k} \sin(z\log(t_{k-1}/a_1)) & \cos(z\log(t_{k-1}/a_1))
\end{bmatrix} \\
& \times 
\begin{bmatrix}
\cos(z\log(t_{k-2}/t_{k-1})) & - \gamma_{k-1} \sin(z\log(t_{k-2}/t_{k-1})) \\ \frac{1}{\gamma_{k-1}} \sin(z\log(t_{k-2}/t_{k-1})) & \cos(z\log(t_{k-2}/t_{k-1}))
\end{bmatrix}
\times \cdots \\
& \times \begin{bmatrix}
\cos(z\log(t_1/t_2)) & - \gamma_2 \sin(z\log(t_1/t_2)) \\ \frac{1}{\gamma_2} \sin(z\log(t_1/t_2)) & \cos(z\log(t_1/t_2))
\end{bmatrix}
\begin{bmatrix}
\cos(z\log(a_0/t_1)) & - \gamma_1 \sin(z\log(a_0/t_1)) \\ \frac{1}{\gamma_1} \sin(z\log(a_0/t_1)) & \cos(z\log(a_0/t_1))
\end{bmatrix}.
\endaligned
$}
\]
Moreover, $\det M(a_1,a_0;z)=1$ is obvious by this definition. 
Now the proof is complete, 
since $\gamma(a)=\gamma_n$ on $[q^{(n-1)/2},q^{n/2})$ for every $1 \leq n \leq 2g$, by definition \eqref{0423_6}. 
\smallskip

\noindent
(2) The matrix-valued function $M(a_1,a;z)$ is left-continuous with respect to $a$ by the above definition, 
since $\gamma(a)$ is left-continuous by definition \eqref{0423_6}. 
Because $A(a,z)$ and $B(a,z)$ are continuous 
with respect to $a$ by definition \eqref{0330_11}, \eqref{0330_10}, and Proposition \ref{0330_4}, 
we obtain \eqref{prop_601_2} from \eqref{prop_601_1}. 
\end{proof}

\begin{corollary} \label{cor_602}
Let $\gamma(a)$ be of \eqref{0423_6}. 
Assume that $\gamma(a)\not=0,\infty$ on $[1,q^g)$. 
Then 
\[
\aligned
\begin{bmatrix}
A(a,z) \\ B(a,z)
\end{bmatrix}
& = E(0)
\begin{bmatrix}
\cos(z\log(q^{n/2}/a)) & -\gamma_{n} \sin(z\log(q^{n/2}/a))\\
\gamma_{n}^{-1}\sin(z\log(q^{n/2}/a)) & \cos(z\log(q^{n/2}/a))
\end{bmatrix} \\
& \quad \times
\prod_{k=1}^{2g-n}
\begin{bmatrix}
\cos(\frac{z}{2}\log q) & -\gamma_{n+k} \sin(\frac{z}{2}\log q)\\
\gamma_{n+k}^{-1}\sin(\frac{z}{2}\log q) & \cos(\frac{z}{2}\log q)
\end{bmatrix}
\begin{bmatrix}
1 \\ 0
\end{bmatrix}
\endaligned
\]
for $q^{(n-1)/2} \leq a < q^{n/2}$ and $1 \leq n \leq 2g$. 
\end{corollary}
\begin{proof}
We obtain the formula by applying the proof of Proposition \ref{prop_601} (1) to 
$a_1=a$ ($q^{(n-1)/2} \leq a < q^{n/2}$), $a_0=q^g$ and $t_k=q^{(g-k)/2}$ ($1 \leq k \leq g-n$).
\end{proof}

\begin{proposition} \label{lem_602}
Define 
\begin{equation} \label{0427_2}
E(a,z) := A(a,z) - i B(a,z)
\end{equation}
and
\begin{equation} \label{lem_602_1}
K(a;z,w)
:=\frac{\overline{E(a,w)}E(a,z)-\overline{E^\sharp(a,w)}E^\sharp(a,z)}{2\pi i(\bar{w}-z)}. 
\end{equation}
Then 
\begin{equation} \label{lem_602_2}
K(a;z,w)
=\frac{\overline{A(a,w)}B(a,z)-\overline{B(a,w)}A(a,z)}{\pi(z-\bar{w})}.
\end{equation}
Moreover, if $\gamma(a)$ and $\gamma(a)^{-1}$ are integrable on $[a_1,a_0]$, then 
\begin{equation} \label{lem_602_3}
\aligned
K&(a_1;z,w) - K(a_0;z,w) \\
& \qquad = \frac{1}{\pi}\int_{a_1}^{a_0}\overline{A(a,w)}A(a,z) \, \frac{1}{\gamma(a)} \, \frac{da}{a} 
+ \frac{1}{\pi}\int_{a_1}^{a_0}\overline{B(a,w)}B(a,z) \, \gamma(a) \, \frac{da}{a} 
\endaligned
\end{equation}
for every $z,w \in \C$. 
\end{proposition}
\begin{proof}
We obtain \eqref{lem_602_2} easily  
by substituting \eqref{0427_2} into \eqref{lem_602_1}. 
By integration by parts, together with \eqref{0330_1}, we obtain 
\[
\aligned
z \int_{a_1}^{a_0}\overline{A(a,w)}A(a,z) \, \frac{1}{\gamma(a)} \, \frac{da}{a}
& = - \left.\overline{A(a,w)}B(a,z)\right|_{a_1}^{a_0} 
+ \bar{w} \int_{a_1}^{a_0}\overline{B(a,w)} B(a,z) \, \gamma(a) \, \frac{da}{a},
\endaligned
\]
\[
\aligned
z \int_{a_1}^{a_0}\overline{B(a,w)}B(a,z) \, \gamma(a) \, \frac{da}{a} 
& = \left.\overline{B(a,w)}A(a,z)\right|_{a_1}^{a_0} 
+ \bar{w} \int_{a_1}^{a_0}\overline{A(a,w)} A(a,z) \, \frac{1}{\gamma(a)} \, \frac{da}{a} .
\endaligned
\]
Moving the second terms of the right-hand sides of the two equations to the left-hand sides, 
then adding both sides of the resulting two equations, 
and finally dividing both sides by $(z-\bar{w})$, we obtain
\[
\aligned
\int_{a_1}^{a_0} & \overline{A(a,w)}A(a,z) \, \frac{1}{\gamma(a)} \, \frac{da}{a}  
+ \int_{a_1}^{a_0}\overline{B(a,w)}B(a,z) \, \gamma(a) \, \frac{da}{a}  \\
&= \frac{\left.(-\overline{A(a,w)}B(a,z)+\overline{B(a,w)}A(a,z))\right|_{a_1}^{a_0}}{z-\bar{w}} 
 = \pi\Bigl( K(a_1;z,w) - K(a_0;z,w) \Bigr).
\endaligned
\]
This implies \eqref{lem_602_3}. 
\end{proof}

\subsection{Proof of Theorem \ref{thm_02}(1).}  \label{section_5_1}
For $\underline{C}= 
(C_g,C_{g-1},\cdots,C_{-g}) \in \R^\ast \times \R^{2g-1} \times \R^\ast
$, 
the polynomial 
\[
f(T):= T^{g}\sum_{m=-g}^{g} C_{-m} \, T^{m} ~\in \R[T]
\]
and $f^\sharp(T):=T^{2g}f(1/T)$ have a common root 
if and only if  $\det D_{2g}(\underline{C})$ is zero 
(\cite[Lemmas 11.5.11 and 11.5.12]{RaSch02}). 
The former is equivalent that 
$E$ and $E^\sharp$ have a common zero, 
since $E(z) = q^{igz} f(q^{-iz})$ 
and $E^\sharp(z)=E(-z) = q^{igz} f^\sharp(q^{-iz})$. 
If $E$ belongs to the class HB, it has no real zeros and 
$|E(\bar{z})|<|E(z)|$ in $\C_+$ by definition of the class HB. 
Therefore $E$ and $E^\sharp$ have no common zeros. 
Hence $\det D_{2g}(\underline{C})\not=0$, 
which implies that $\det D_n(\underline{C})\not=0$ 
for every $1 \leq n \leq 2g$. 
Hence $\gamma_n \not=0,\infty$ for every $1 \leq n \leq 2g$ 
by Theorem \ref{thm_03} (1). 
\medskip

Therefore, it is sufficient  to prove that 
$E(z)$ is not a function of the class HB 
if $\gamma_{n} < 0$ for some $1 \leq n \leq 2g$. 
We proceed in three steps as follows. 
\medskip

{\bf Step 1.} 
We show that there is no loss of generality 
in assuming that  
there exists  $1 \leq n_0 \leq 2g-1$ 
such that $\gamma_{n}>0$ 
for every $1 \leq n \leq n_0$ 
and $\gamma_{n_0+1} < 0$ holds.  
We have 
\begin{equation} \label{0503_3}
\gamma_{1}
=\frac{\alpha_1}{\beta_1}
=\frac{\det(E^+ + E^- J_1)}{\det(E^+ - E^- J_1)}
=\frac{C_{-g}+C_{g}}{C_{-g}-C_{g}}
\end{equation}
by definition \eqref{0425_3} and Proposition \ref{prop_0313_1}. 
In addition, $|C_g/C_{-g}|<1$ if $E(z)$ belongs to the class HB, 
since 
\[
\left| \frac{E^\sharp(iy)}{E(iy)} \right| = \left| \frac{C_{g}}{C_{-g}} + Cq^{k(ix-y)} + O(q^{-(k+1)y}) \right|
\quad (y \to + \infty)
\] 
for some $k \geq 1$ and $C \in \R$. 
Therefore, 
\begin{equation} \label{m_2g-1}
1 + \frac{C_{g}}{C_{-g}} >0, \quad  1 - \frac{C_{g}}{C_{-g}} >0. 
\end{equation}
This implies that $\gamma_1>0$ if $E(z)$ belongs to the class HB. 

{\bf Step 2.} Let $n_0$ be the number of Step 1. 
In this part, we show that 
$E(z)$ is not a function of the class HB 
if $E(a,z)$ of \eqref{0427_2} is not a function of the class HB for some $1< a \leq q^{(n_0+1)/2}$. 
We have 
\begin{equation} \label{603}
\begin{bmatrix}
A(z) \\ B(z)
\end{bmatrix} 
=
\begin{bmatrix}
A(1,z) \\ B(1,z)
\end{bmatrix} 
= 
M(1,a;z)
\begin{bmatrix}
A(a,z) \\ B(a,z)
\end{bmatrix} 
\end{equation}
for $1 \leq a < q^{(n_0+1)/2}$ by applying \eqref{prop_601_1} to $(a_1,a_0)=(1,a)$, 
when $\gamma(a)\not=0$ and $\gamma(a)^{-1}\not=0$ for $1 \leq a < q^{(n_0+1)/2}$.   

Suppose that $E(a_0,z)$ is not a function of the class HB for some $1 < a_0 \leq  q^{(n_0+1)/2}$, 
that is, $E(a_0,z)$ has a real zero for some $1<a_0 \leq q^{(n_0+1)/2}$ 
or $|E^\sharp(a_0,z)| \geq |E(a_0,z)|$ for some $z \in \C_+$ and $1<a_0 \leq q^{(n_0+1)/2}$. 
If $E(a_0,z)$ has a real zero for some $1<a_0 \leq q^{(n_0+1)/2}$, 
then $A(a_0,z)$ and $B(a_0,z)$ have a common real zero, 
since they are real-valued on the real line.  
Therefore, \eqref{603} and $\det M(1,a_0;z)=1$ imply that $A(z)$ and $B(z)$ have a common real zero. 
Hence $E(z)$ has a real zero by $E(z)=A(z)-iB(z)$. Thus $E(z)$ is not a function of the class HB. 

On the other hand, we assume that $E(a,z)$ has no real zeros for every $1 < a \leq  q^{(n_0+1)/2}$ 
but it has a zero in the upper half plane for some $1 < a_0 \leq q^{(n_0+1)/2}$. 
By \eqref{0330_10} and \eqref{0427_2}, $E(a,z)$ is a continuous function of $(a,z)\in [1, q^{(n_0+1)/2}]\times \C$. 
Therefore, any zero locus of $E(a,z)$ is a continuous curve in $\C$ parametrized by $a\in[1,q^{(n_0+1)/2}]$. 
Denote by $z_a \subset \C$ a zero locus through a zero of $E(a_0,z)$ in the upper-half plane, 
that is, $E(a,z_a)=0$ for every $1 \leq a \leq q^{(n_0+1)/2}$. 
If ${\rm Im}(z_{a_1})<0$ for some $1\leq a_1 < a_0$, 
then ${\rm Im}(z_{a_2})=0$ for some $a_1<a_2<a_0$. 
This implies that $E(a_2,z)$ has a real zero at $z=z_{a_2}$. This is a contradiction. 
Therefore, ${\rm Im}(z_a) \geq 0$ for every $1 \leq a < a_0$, 
in particular ${\rm Im}(z_1) \geq 0$. 
This implies $E(z)=E(1,z)$ is not a function of the class HB. 

Assume that $E(a,z)\not=0$ for every ${\rm Im}\, z \geq 0$ and $1 < a \leq q^{(n_0+1)/2}$
but $|E^\sharp(a_0,z_0)| \geq |E(a_0,z_0)|$ for some $1 < a_0 \leq q^{(n_0+1)/2}$ and ${\rm Im}(z_0)>0$. 
Then it derives a contradiction. 
Because $A(a,z)$ and $B(a,z)$ are bounded on the real line as a function of $z$ by definition \eqref{0330_10}, 
$E(a,z)$ is a function of the Cartwright class~\cite[the first page of Chapter II]{Levin96}. 
Therefore, we have the factorization
\[
E(a_0,z) = C \lim_{R \to \infty} \prod_{\substack{ |\rho|<R \\ E(a_0,\rho)=0 }}\left(1-\frac{z}{\rho} \right) ; 
\]
see \cite[Remark 2 of Lecture 17.2]{Levin96}. Here ${\rm Im}(\rho)<0$ for every zero of $E(a_0,z)$ by the assumption. 
Hence, we have
\[
\left|\frac{E^\sharp(a_0,z)}{E(a_0,z)}\right| 
= \lim_{R \to \infty} \prod_{\substack{ |\rho|<R \\ E(a_0,\rho)=0 }}\left| \frac{z-\bar{\rho}}{z-\rho} \right|<1 \quad \text{for} \quad {\rm Im}\, z>0.
\]
This contradicts the assumption $|E^\sharp(a_0,z_0)| \geq |E(a_0,z_0)|$. 

{\bf Step 3.} 
For the number $n_0$ of Step 1, 
we prove that $E(z)$ is not a function of the class HB 
if $\gamma_{n_0+1}<0$. 
Considering the argument in Step 2, 
we assume that $E(a,z)$ is a function of the class HB for every $1< a \leq q^{(n_0+1)/2}$ 
in both cases and find a contradiction. 

Suppose that $E(a,z)$ is a function of the class HB for every $1<a\leq q^{(n_0+1)/2}$. 
Put $a_1=q^{n_0/2}$, 
$a_0=(q^{(n_0+1)/2}+q^{n_0/2})/2$ and $\gamma_{n_0+1}=-\gamma<0$. 
Then, 
for every $a_1 \leq a \leq a_0$, 
$\gamma(a)=-\gamma$ by \eqref{0423_6}, 
and we find that  
$E(a,z)$ generates the de Branges space $B(E(a,z))$ 
which is the Hilbert space of all entire functions $F(z)$ such that 
$\int_{\R}|F(x)/E(a,x)|^2 dx<\infty$,  
and $F(z)/E(a,z)$, $F^\sharp(z)/E(a,z)$ are 
functions of the Hardy space $H^2$ in the upper half-plane; 
see \cite[\S19]{deBranges68} and \cite[Proposition 2.1]{Remling02}. 
By applying \eqref{lem_602_3} to $z=w$ with $\gamma(a)=-\gamma<0$, we have $K(a_0;z,z) > K(a_1;z,z)$. 
Therefore, for every $f \in B(E(a_1,z))$, 
\[
|f(z)|^2 \leq \Vert f \Vert_{a_1}^2K(a_1;z,z) < \Vert f \Vert_{a_1}^2K(a_0;z,z)
\]
by \cite[Theorem 20]{deBranges68}, 
where $\Vert \cdot \Vert_{a_1}$ is the norm of $B(E(a_1,z))$. 
Applying this to the function 
$g(z):=(E(a_1,z)-E(a_1,iy_0))/(z-iy_0)$ $(y_0 \in \R)$, 
which is a function of $B(E(a_1,z))$ by Lemmas 3.3 and 3.4 of \cite{Dym70}, 
we obtain  
\[
\aligned
|g(iy)|^2
\leq \Vert g \Vert_{a_1}^2 K(a_0,iy,iy) =\Vert g \Vert_{a_1}^2
\frac{|E(a_0,iy)|^2-|E^\sharp(a_0,z)|^2}{4\pi y}
\leq \Vert g \Vert_{a_1}^2
\frac{|E(a_0,iy)|^2}{4\pi y}. 
\endaligned
\]
By $E(a,z)=A(a,z)-iB(a,z)$ with \eqref{0330_10}, we see that 
\[
y^{-1}q^{(g-\frac{n_0-1}{2})y} \ll |g(iy)| \ll 
y^{-1/2}|E(a_0,iy)| \ll y^{-1/2} q^{(g-\frac{n_0}{2})y} 
\quad \text{as $y \to +\infty$}. 
\]
This is a contradiction. Hence $E(a,z)$ is not a function of the class HB for some $1<a \leq a_0$. 

The conclusions of Steps 2 and 3 show that $E(z)$ is not a function of the class HB 
if $\gamma_{n}$ is neither positive nor finite for some $1 \leq n \leq 2g$. 
\hfill $\Box$

\subsection{Proof of Theorem \ref{thm_02}(2).}  \label{section_5_2}

By assumption, $\det(E^{+} \pm E^{-}J_n) \not=0$ for every $1 \leq n \leq 2g$. 
Therefore, $0<\gamma_{n}<\infty$ for every $1 \leq n \leq 2g$, 
since definition \eqref{0425_2} and \eqref{0423_6} are equivalent to each other by Proposition \ref{0330_4}. 
Hence, $\alpha_n, \beta_n \not=0,\infty$ for every $1 \leq n \leq 2g$, 
$A(a,z)$ and $B(a,z)$ are well-defined 
and satisfy the system \eqref{0330_1}, 
and both $\gamma(a)$ and $\gamma(a)^{-1}$ are integrable on $[1,q^g)$ and positive real-valued. 
Hence, applying \eqref{lem_602_3} to $(a_1,a_0)=(a,b)$ 
and then letting $b$ tend to $q^g$ and using Theorem \ref{thm_03}\,(3) and \eqref{lem_602_1}, 
we obtain 
\[
\aligned
0 & 
< \frac{1}{\pi}\int_{a}^{q^g} |A(t,z)|^2 \, \frac{1}{\gamma(t)} \, \frac{dt}{t} 
+ \frac{1}{\pi}\int_{a}^{q^g} |B(t,z)|^2 \, \gamma(t) \, \frac{dt}{t} \\
& \quad = K(a;z,z) - \lim_{b \nearrow q^g} K(b;z,z) = \frac{|E(a,z)|^2-|E^\sharp(a,z)|^2}{4\pi {\rm Im}\, z}
\endaligned
\]
for every $1 \leq a < q^g$ if ${\rm Im}\, z>0$. 
Thus $E(a,z)$ satisfy \eqref{HB} for every $1 \leq a < q^g$. 
The proof is then completed once we prove that $E(a,z)$ has no real zeros. 

We have 
\begin{equation} \label{0501_4}
\begin{bmatrix}
A(a,z) \\ B(a,z)
\end{bmatrix} 
= \left( \lim_{b \nearrow q^g}
M(a,b;z) \right)
\begin{bmatrix}
E(0) \\ 0
\end{bmatrix} \quad \text{with} \quad \lim_{b \nearrow q^g}\det M(a,b;z)=1
\end{equation}
for every $1 \leq a < q^g$. 
Here $E(0)=A(0)$, since $B(z)$ is odd.    
By Proposition \ref{prop_0313_1}, 
\begin{equation} 
\aligned
\alpha_1 \cdots \alpha_{2g}
& = \frac{\det(E_0^{+})}{(\det(E^{+}))^2} \frac{\det(E^{+} + E^{-}J_{2g-1})\det(E^{+} + E^{-}J_{2g})}
{\det(E_0^{+} + E_{2g-1}^{\sharp}J)} \\
& = C_{-g}^{4g-2} \frac{\det(E^{+} + E^{-}J_{2g-1})\det(E^{+} + E^{-}J_{2g})}
{\det(E_0^{+} + E_{2g-1}^{\sharp}J)}.
\endaligned
\end{equation}
Applying Lemma \ref{mat_3} to the $4g$-th column and then to the $(4g+1)$-th row, we obtain 
\[
\det(E_0^+ + E_{2g-1}^{\sharp}J)
= C_{-g}^2 \det(A+B)\det(A-B)
\]
with 
\[
A = [E_0^+ + E_{2g-1}^{\sharp}J]_{\nwarrow 4g-1}, \quad B = [E_0^+ + E_{2g-1}^{\sharp}J]_{\nearrow 4g-1}, 
\]
by Lemma \ref{mat_2}. Applying Lemma \ref{mat_3} to the $(4g-2g-1)$ columns of $A\pm B$ 
with indices $(2g+2,2g+3,\cdots,4g-1)$, we obtain 
\[ \det(A\pm B)=C_{-g}^{2g-2}\det(E^{+} \pm E^{-}J_{2g-1}), \]
because
\[ [E_0^{+} + E_{2g-1}^{\sharp}J]_{\nwarrow  2g+1} = E^{+} \]
and
\[ [E_0^{+} + E_{2g-1}^{\sharp}J]_{\nearrow 2g+1} = E^{-}J_{2g-1}. \]
Therefore, 
\begin{equation} 
\alpha_1 \cdots \alpha_{2g}
= \frac{\det(E^{+} + E^{-}J_{2g})}{\det(E^{+} - E^{-}J_{2g-1})}.
\end{equation}
Hence $A(0)=E(0)\not=0$ by \eqref{0503_1} and  \eqref{0503_2}, 
which implies that $A(a,z)$ and $B(a,z)$ have no common zeros 
for every $1 \leq a <q^g$ by \eqref{0501_4}. 
Thus $E(a,z)$ has no real zeros, by \eqref{0427_2}. 
As a consequence, $E(a,z)$ is a function of the class HB for every $1 \leq a < q^g$. 
\hfill $\Box$

\subsection{Proof of Theorem \ref{thm_02_1}.}  \label{section_5_3} 
First, we note that the proofs of Proposition \ref{prop_601} and Corollary \ref{cor_602} are valid 
for $H(a)$ of Theorem \ref{thm_02_1}, since they use only the property that $\gamma(a)$ is constant on every $[q^{(n-1)/2},q^{n/2})$ 
($1 \leq n \leq 2g$). 
Hence, the quasi-canonical system \eqref{can_0} for $H(a)$ of Theorem \ref{thm_02_1} 
has the unique solution $(A(a,z),B(a,z))$ as in Corollary \ref{cor_602} (with $E(0)=1$). 
To prove \eqref{0829_1}, we put 
\begin{equation} \label{0829_4}
\begin{bmatrix} 
M_{11}^{n,K}(z) & M_{12}^{n,K}(z) \\ 
M_{21}^{n,K}(z) & M_{22}^{n,K}(z)
\end{bmatrix} 
=
\prod_{k=2g-n+1-K}^{2g-n}
\begin{bmatrix}
\cos(\frac{z}{2}\log q) & -\gamma_{n+k} \sin(\frac{z}{2}\log q)\\
\gamma_{n+k}^{-1}\sin(\frac{z}{2}\log q) & \cos(\frac{z}{2}\log q)
\end{bmatrix}
\end{equation}
for $1 \leq K \leq 2g-n$. 
Then Corollary \ref{cor_602} implies 
\begin{equation} \label{0829_2}
\aligned 
A(a,z) & = \cos(z\log(q^{n/2}/a))M_{11}^{n,2g-n}(z) - \gamma_n \sin(z\log(q^{n/2}/a))M_{21}^{n,2g-n}(z), \\
B(a,z) & = \gamma_n^{-1} \sin(z\log(q^{n/2}/a))M_{11}^{n,2g-n}(z) + \cos(z\log(q^{n/2}/a))M_{21}^{n,2g-n}(z)
\endaligned 
\end{equation}
for $q^{(n-1)/2} \leq a < q^{n/2}$ and $1 \leq n \leq 2g$. 
On the other hand, we obtain the formula 
\begin{equation} \label{0829_3}
M_{rs}^{n,K}(z) = \mu_{rs} \sum_{\nu=0}^{\lfloor K/2 \rfloor} m_{rs}^{n,K}(\nu)(q^{(K-2\nu)iz/2}+(-1)^{r+s}q^{-(K-2\nu)iz/2}) 
\end{equation}
for $r,s \in \{1,2\}$ by induction for $K \geq 1$, where $\mu_{11}=\mu_{22}=1$, $\mu_{12}=i$, $\mu_{21}=-i$, 
and $m_{rs}^{n,K}(\nu)$ are real numbers depending only on the set $\{\gamma_n\}_{1 \leq n \leq 2g}$ 
of values of $\gamma(a)$. 
Substituting \eqref{0829_3} into \eqref{0829_2} and then carrying out a simple calculation,  
we obtain \eqref{0829_1}. \hfill $\Box$

\subsection{Proof of Theorem \ref{thm_02_2}.}  \label{section_5_4} 
We have 
\[
\aligned
E(a,z) & = A(a,z) - i B(a,z) \\
& = 
 \frac{\gamma_n^{-1}}{2}
 \sum_{\nu=0}^{\lfloor 2g-n/2 \rfloor} (\gamma_n - 1)(m_{11}^{n,2g-n}(\nu)+\gamma_n m_{21}^{n,2g-n}(\nu)) (q^{(g-\nu)}/a)^{iz} \\
& \quad + 
\frac{\gamma_n^{-1}}{2}
 \sum_{\nu=0}^{\lfloor 2g-n/2 \rfloor} (\gamma_n + 1)(m_{11}^{n,2g-n}(\nu)+\gamma_n m_{21}^{n,2g-n}(\nu)) (q^{(g-\nu)}/a)^{-iz} \\
& \quad + \frac{\gamma_n^{-1}}{2}
 \sum_{\nu=0}^{\lfloor 2g-n/2 \rfloor} (\gamma_n - 1)(m_{11}^{n,2g-n}(\nu)-\gamma_n m_{21}^{n,2g-n}(\nu)) (q^{-(g-n-\nu)}/a)^{iz} \\
& \quad + \frac{\gamma_n^{-1}}{2}
 \sum_{\nu=0}^{\lfloor 2g-n/2 \rfloor} (\gamma_n + 1)(m_{11}^{n,2g-n}(\nu)-\gamma_n m_{21}^{n,2g-n}(\nu))(q^{-(g-n-\nu)}/a)^{-iz} 
\endaligned
\]
for $q^{(n-1)/2} \leq a < q^{n/2}$ and $1 \leq n \leq 2g$ by \eqref{0829_2} and \eqref{0829_3}. 
Applying this to $a=1$, we obtain 
\[
\aligned
E(1,z) 
& = 
\frac{\gamma_1^{-1}}{2} (\gamma_1 - 1)(m_{11}^{1,2g-1}(0)+\gamma_1 m_{21}^{1,2g-1}(0)) \, q^{giz} + \cdots  \\
& \quad  \cdots + \frac{\gamma_1^{-1}}{2}
 (\gamma_1 + 1)(m_{11}^{1,2g-1}(0)+\gamma_1 m_{21}^{1,2g-1}(0)) \, q^{-giz}.
\endaligned
\]
Therefore, to complete the proof of Theorem \ref{thm_02_2}, 
it suffices to to show that $m_{11}^{1,2g-1}(0)$ and $m_{21}^{1,2g-1}(0)$ are positive 
if $\gamma_n>0$ for every $1 \leq n \leq 2g$. 
As easily derived from definition \eqref{0829_4}, 
 $m_{11}^{n,K}(0)$ and $m_{21}^{n,K}(0)$ 
satisfy the inductive relations 
\[
\aligned 
m_{11}^{n,K+1}(0) & = \frac{1}{2}(m_{11}^{n,K}(0)+\gamma_{2g-K}m_{21}^{n,K}(0)), \\ 
m_{21}^{n,K+1}(0) & = \frac{1}{2\gamma_{2g-K}}(m_{11}^{n,K}(0)+\gamma_{2g-K}m_{21}^{n,K}(0))
\endaligned 
\]
for $1 \leq K \leq 2g-n-1$, and $m_{11}^{n,1}(0)=1$, $m_{21}^{n,1}(0)=\gamma_{2g}^{-1}$. 
Hence, $m_{11}^{n,K}(0)$ and $m_{21}^{n,K}(0)$ are positive for every $1 \leq K \leq 2g-n$ and $1 \leq n \leq 2g$ 
if $\gamma_n>0$ for every $1 \leq n \leq 2g$.  
\hfill $\Box$

\section{Inductive construction} \label{section_7}

The pair of functions $(A(a,z),B(a,z))$ of \eqref{0330_10} is written as 
\[
\aligned
A_n(a,z) &=  \frac{1}{2}\prod_{j=1}^n \alpha_j \cdot F(a,z) \cdot (I+J)E_0^{+}
(E_0^{+} + E_n^\sharp J)^{-1}E_0^{-} \chi_{8g}, \\
B_n(a,z) &= \frac{1}{2i}\prod_{j=1}^n \beta_j \cdot F(a,z) \cdot (I-J)E_0^{+}
(E_0^{+} - E_n^\sharp J)^{-1}E_0^{-} \chi_{8g}
\endaligned
\]
for $q^{(n-1)/2} \leq a < q^{n/2}$ by \eqref{0330_11}, \eqref{LE2}, \eqref{0501_6}, and \eqref{0330_7}, where $I=I^{(8g)}$ and $J=J^{(8g)}$. 
The above formula is explicit but is rather complicated from a computational point of view. 
In contrast, the following method, based on Proposition \ref{0419_1}, 
is often useful for computing the triple $(\gamma(a), A(a,z),B(a,z))$. 
\begin{theorem} \label{0503_6} 
Let $\tilde{\Omega}_{0}$ be a column vector of length $(4g+2)$. 
Define the column vectors $\tilde{\Omega}_{n}$ $(1 \leq n \leq 2g)$ of length $(4g-2n+2)$ inductively as follows: 
\begin{equation}\label{def_m1}
\tilde{\gamma}_{n+1}
:=\frac{\tilde{\Omega}_{n}(1)+\tilde{\Omega}_{n}(2g-n+1)}{\tilde{\Omega}_{n}(2g-n+2) - \tilde{\Omega}_{n}(4g-2n+2)},
\end{equation}
\begin{equation}\label{def_v1}
\tilde{\Omega}_{n+1} := P_{2g-(n+1)}(\tilde{\gamma}_{n+1})^{-1}Q_{2g-(n+1)} \, \tilde{\Omega}_{n},
\end{equation}
where $P_0(m_0):=P_0$ and $v(j)$ means the $j$-th component of a column vector $v$. 

Suppose that $\tilde{\Omega}_{0}$ is the vector defined by \eqref{0423_5} 
for a numerical vector $\underline{C} \in \R^\ast \times \R^{2g-1} \times \R^\ast$ such that 
the exponential polynomial \eqref{0418_2} has no zeros on the real line. 
Then $\tilde{\gamma}_n$ and $\tilde{\Omega}_{n}$ are well-defined as functions of $\underline{C}$ 
for every $1 \leq n \leq 2g$, and  
\[
\gamma_n=\tilde{\gamma}_n ,\quad \Omega_n=\tilde{\Omega}_n,
\]
where $\gamma_n$ and $\Omega_n$ are defined in \eqref{0425_3} and \eqref{0501_2}, respectively. 
\end{theorem}
\begin{proof}
Let $\underline{C}$ be a numerical vector such that 
the exponential polynomial \eqref{0418_2} has no zeros on the real line. 
Then $\gamma_n$ and $\Omega_{n}$ of \eqref{0425_3} and \eqref{0501_2} 
satisfy \eqref{def_m1} and \eqref{def_v1}, by the definitions of $P_k(m_k)$, $Q_k$, and \eqref{0427_1}. 
Therefore, $\gamma_n \not=0$ as a function of $\underline{C}$ for every $1 \leq n \leq 2g$, by Theorem \ref{thm_01}, 
since the cyclotomic polynomial of degree $2g$ is a self-reciprocal polynomials of degree $2g$, all of whose roots are simpl and on $T$. 
Hence, Lemma \ref{lem_201} implies that $\tilde{\Omega}_{1},\tilde{\Omega}_{2}, \cdots, \tilde{\Omega}_{2g}$ 
are uniquely determined from the initial vector $\tilde{\Omega}_{0}$. 
Therefore, $\Omega_{n}=\tilde{\Omega}_{n}$ for every $1 \leq n \leq 2g$, by definition $\tilde{\Omega}_{0}=\Omega_{0}$. 
\end{proof}

\begin{proposition} \label{0503_4}
Let $\underline{C}=(C_g,C_{g-1},\cdots,C_{-g})$ 
be a vector consisting of $2g+1$ indeterminate elements. 
Define $\tilde{\gamma}_n(\underline{C})$ by  \eqref{def_m1} and \eqref{def_v1}, 
starting with the initial vector \eqref{0423_5}. 
Define 
\begin{equation} \label{0503_7}
\tilde{\Delta}_n
=
\tilde{\Delta}_n(\underline{C}) 
:= 
\begin{cases}
\displaystyle{ \tilde{\gamma}_{1}(\underline{C})
\prod_{j=1}^{J} \frac{\tilde{\gamma}_{2j+1}(\underline{C})}{\tilde{\gamma}_{2j}(\underline{C})}} & \text{if $n=2J+1\geq 1$}, \\
\displaystyle{\prod_{j=0}^{J} \frac{\tilde{\gamma}_{2j+2}(\underline{C})}{\tilde{\gamma}_{2j+1}(\underline{C})}} & \text{if $n=2J+2 \geq 2$}, 
\end{cases}
\end{equation}
where $\tilde{\gamma}_n$ are functions of $\underline{C}$ in Theorem \ref{0503_6}. 
Then $\tilde{\Delta}_n=\Delta_n$ for every $1 \leq n \leq 2g$ 
if $\underline{C}$ is a numerical vector such that 
the exponential polynomial \eqref{0418_2} has no zeros on the real line. 
\end{proposition}
\begin{proof}
By Theorem \ref{0503_6}, $\gamma_n=\tilde{\gamma}_n$. 
On the other hand, $\Delta_n$ and $\gamma_n$ satisfy \eqref{0503_7}, 
by \eqref{0425_2}. Hence $\tilde{\Delta}_n=\Delta_n$.  
\end{proof}
Here we mention that the vector $\Omega_{n}$ of \eqref{def_v1} can be defined from $\Omega_{n-1}$ by a slightly different way 
according to the following lemma. 

\begin{lemma} \label{0501_5}
For every $1 \leq n \leq 2g$, 
\[
\scalebox{0.9}{$
\aligned
\,& 
\frac{\Omega_{n-1}(1)+\Omega_{n-1}(2g-n+2)}{\Omega_{n-1}(2g-n+3)-\Omega_{n-1}(4g-2n+4)} 
=
\frac{(P_{2g-n}(m_{2g-n})^{-1}Q_{2g-n} \, \Omega_{n-1})(1)}{(P_{2g-n}(m_{2g-n})^{-1}Q_{2g-n} \, \Omega_{n-1})(2g-n+2)},
\endaligned
$}
\]
that is, the right-hand side is independent of the indeterminate element $m_{2g-n}$.
\end{lemma}
\begin{proof} 
Formula \eqref{0426_1} shows that 
\[
\aligned
(P_{2g-n}(m_{2g-n})^{-1}Q_{2g-n}\Omega_{n-1})(1) &= \Omega_{n-1}(1)+\Omega_{n-1}(2g-n+2), \\
(P_{2g-n}(m_{2g-n})^{-1}Q_{2g-n}\Omega_{n-1})(2g-n+2)  &= \Omega_{n-1}(2g-n+3)-\Omega_{n-1}(4g-2n+4)
\endaligned
\]
for every $1 \leq n \leq 2g$. These equalities imply the lemma. 
\end{proof}

Using Lemma \ref{0501_5}, we can define $\Omega_{n}$ by 
taking 
\[
\Omega_{n}^\prime=P_{2g-n}(\gamma_{n})^{-1}Q_{2g-n} \, \Omega_{n-1}
\] 
for $\Omega_{n-1}$ 
and then substituting $\Omega_{n}^{\prime}(1)/\Omega_{n}^{\prime}(2g-n+2)$ into $\gamma_n$ of $\Omega_{n}^{\prime}$.

\section{Applications to self-reciprocal polynomials} \label{section_8}

\subsection{Proof of Theorem \ref{thm_1}.} \label{section_8_1}

For a self-reciprocal polynomial $P_g(x)$ of \eqref{def_Pg} and a real number $q>1$,  
we define
\begin{equation} \label{def_AB_g}
A_q(z) := q^{-giz}P_g(q^{iz}), \qquad B_q(z):=-\frac{d}{dz}A_q(z)
\end{equation}
and 
\begin{equation} \label{def_E}
E_q(z) := A_q(z) - i B_q(z).
\end{equation}
Then the reality of the coefficients of $P_g(x)$ and the self-reciprocal condition $P_g(x)=x^{2g}P_g(1/x)$ imply that 
that $A_q(z)$ (respectively, $B_q(z)$) is an even (respectively, odd) real entire function of exponential type, namely, 
$A_q(-z)=A_q(z)$ and $A_q^\sharp(z)=A_q(z)$ 
(respectively, $B_q(-z)=-B_q(z)$ and $B_q^\sharp(z)=B_q(z)$). 
In particular, $E_q^\sharp(z) = A_q(z) + i B_q(z)$. 
\smallskip

By \eqref{def_AB_g}, all roots of $P_g(x)$ are simple and on $T$ if and only if $A_q(z)$ has only simple real zeros. 
The following lemma enables us to obtain Theorem \ref{thm_1} 
as a corollary of Theorems \ref{thm_01} and \ref{thm_02}. 
\begin{lemma} \label{lem_401}
Let $E_q(z)$, $A_q(z)$, $B_q(z)$ be as above. Then 
\begin{enumerate}
\item $E_q(z)$ satisfies condition \eqref{HB}
if and only if $A_q(z)$ has only real zeros, 
\item $E_q(z)$ is a function of the class {\rm HB}
if and only if $A_q(z)$ has only simple real zeros. 
\end{enumerate}
\end{lemma}
\begin{proof} 
(1) Assume that $E_q(z)$ satisfies \eqref{HB}. 
Then $A_q(z) \not=0$ for ${\rm Im}\, z>0$.  
Furthermore, $A_q(z) \not=0$ for ${\rm Im}\, z<0$, 
by the functional equation $A_q(z)=A_q(-z)$. 
Hence all zeros of $A_q(z)$ lie on the real line. 

Conversely, assume that all zeros of $A_q(z)$ are real. 
Then $A_q(z)$ has the factorization 
\begin{equation} \label{H-factorization}
A_q(z) = C \lim_{R \to \infty}\prod_{\substack{|\rho| \leq R \\ A_q(\rho)=0}}\left(1-\frac{z}{\rho} \right) \quad (C,\,\rho \in \R), 
\end{equation}
because $A_q(z)$ is real, even and of exponential type. 
Therefore, 
\[
{\rm Re}\left( i\frac{A_q^\prime(z)}{A_q(z)} \right) 
= {\rm Re}\left(  \sum_{\substack{\rho \in \R \\ A_q(\rho)=0}} \frac{i(x-\rho) + y}{|z-\rho|^2} \right) 
= \sum_{\substack{\rho \in \R \\ A_q(\rho)=0}} \frac{y}{|z-\rho|^2} \quad (z=x+iy).
\]
Hence, for ${\rm Im}\, z>0$, 
\[
|E_q(z)|=|A_q(z)|\left|1+i\frac{A_q^\prime(z)}{A_q(z)} \right|>|A_q(z)|\left|1-i\frac{A_q^\prime(z)}{A_q(z)} \right|=|E_q^\sharp(z)|.
\]
\noindent
(2) Suppose that $E_q(z)$ is a function of the class HB, that is, 
$E_q(z)$ satisfies \eqref{HB} and has no real zeros. 
Then, by (1), $A_q(z)$ has only real zeros. 
If $A_q(z)$ has a multiple real zero, 
then $A_q(z)$ and $B_q(z)=-A_q^\prime(z)$ have a common real zero.  
Thus $E_q(z)=A_q(z)-iB_q(z)$ has a real zero, 
which is a contradiction. 
Hence $A_q(z)$ has only simple real zeros. 

The converse assertion follows from (1) and definition \eqref{def_B}. 
\end{proof}

Let $\underline{C}$ be the vector defined by \eqref{0418_1} for $\underline{c} \in \R^\ast \times \R^g$ and $q>1$. 
Then the exponential polynomial $E(z)$ of \eqref{0418_2} 
is equal to $E_q(z)$ of \eqref{def_E}. By Lemma \ref{lem_401}, all roots of $P_g(x)$ are simple and on $T$ 
if and only if $E_q(z)$ is a function of the class HB. 
Therefore, we obtain Theorem \ref{thm_1} by applying Theorems \ref{thm_01} and \ref{thm_02} to \eqref{0423_8}.  
However it remains to prove that $\delta_n(\underline{c})$ is independent of the choice of $q>1$. 
This is proved in the following proposition. 

\begin{proposition} \label{0503_5} 
$\delta_n(\underline{c}) \in \Q(c_0,\cdots,c_g)$ for every $0 \leq n \leq 2g$. 
\end{proposition}
\begin{proof} 
Let $\underline{C}$ be of \eqref{0418_1}, and put $F=\Q(c_0,\cdots,c_g)$. 
Firstly, we show that $\Omega_{n}(k) \in F$ for $1 \leq k \leq 2g-n+1$ 
and $(\log q)^{-1}\Omega_{n}(k) \in F$ for $2g-n+2 \leq k \leq 4g-2n+2$, 
because it implies $(\log q)\gamma_{n+1}(\underline{C}) \in F$ by  Theorem \ref{0503_6} and \eqref{def_m1}. 
We have $\Omega_{1}(k)\in F$ for $1 \leq k \leq 2g$ 
and  $(\log q)^{-1}\Omega_{n}(k) \in F$ for $2g+1 \leq k \leq 4g$ 
by the formula for $\Omega_{1}$ in the proof of Proposition \ref{0501_3}, 
together with \eqref{0418_1} and \eqref{0503_3}. 

Assume the above assertion for $\Omega_{n-1}$. 
Then $\gamma_{n}(\underline{C})=(\log q)^{-1}\mu_{n}$ 
for some $\mu_{n} \in F$, by Theorem \ref{0503_6} and \eqref{def_m1}. 
Applying the formula of $P_{k}(m_{k})^{-1}Q_{k}$ in Lemma \ref{lem_203} to $m_k=\gamma_{2g-k}$ and $k=2g-n$, 
we obtain 
$\Omega_{n}(k) \in F$ for $1 \leq k \leq 2g-n+1$ 
and $(\log q)^{-1}\Omega_{n}(k) \in F$ for $2g-n+2 \leq k \leq 4g-2n+2$,  
because of Theorem \ref{0503_6} and \eqref{def_v1}. 
Hence $(\log q)\gamma_{n}(\underline{C}) \in F$ 
for every $1 \leq n \leq 2g$, by induction. 

On the other hand, 
\begin{equation} \label{def_R}
\delta_{n}(\underline{c})
=
\begin{cases}
\displaystyle{ (g \log q)\cdot \gamma_1(\underline{C})
\prod_{j=1}^{J} \frac{\gamma_{2j+1}(\underline{C})}{\gamma_{2j}(\underline{C})}} & \text{if $n=2J+1\geq 1$}, \\
\displaystyle{\prod_{j=0}^{J} \frac{\gamma_{2j+2}(\underline{C})}{\gamma_{2j+1}(\underline{C})}} & \text{if $n=2J+2 \geq 2$}, 
\end{cases}
\end{equation}
by definition \eqref{0423_8} and Proposition \ref{0503_4}. 
This formula for $\delta_n(\underline{c})$ implies that $\delta_n(\underline{c}) \in F$ 
for every $1 \leq n \leq 2n$.
\end{proof}

\subsection{Proof of Theorem \ref{thm_2}.}  \label{section_8_2}

Using the function of \eqref{def_AB_g}, we define 
\begin{equation} \label{def_E2}
E_{q,\omega}(z) := A_q(z+i\omega), 
\end{equation}
\begin{equation}
A_{q,\omega}(z) := \frac{1}{2}(E_{q,\omega}(z)+E_{q,\omega}^\sharp(z)), \quad 
B_{q,\omega}(z) := \frac{i}{2}(E_{q,\omega}(z)-E_{q,\omega}^\sharp(z)).
\end{equation}
Then $A_{q,\omega}(z)$ and $B_{q,\omega}(z)$ are real entire functions satisfying 
$E_{q,\omega}(z) = A_{q,\omega}(z) - i B_{q,\omega}(z)$. 
We use the following lemma instead of Lemma \ref{lem_401}. 

\begin{lemma} \label{lem_2_1}
Let $E_{q,\omega}(z)$ be in \eqref{def_E2}. 
Then, all roots of $P_g(x)$ lie on $T$ 
if and only if 
$E_{q,\omega}(z)$ is a function of the class {\rm HB} 
for every $\omega>0$. 
\end{lemma}
\begin{proof}
By definition \eqref{def_AB_g}, 
all roots of $P_g(x)$ lie on $T$ 
if and only if $A_q(z)$ has only real zeros. 
Suppose that $A_q(z)$ has only real zeros (allowing multiple zeros). 
Then $E_{q,\omega}(z)$ satisfies inequality \eqref{HB} for every $\omega>0$, 
because
\[
\aligned
\left| \frac{\overline{E_{q,\omega}(\bar{z})}}{E_{q,\omega}(z)} \right|^2 
& = 
\left| \frac{A_q(z-i\omega)}{A_q(z+i\omega)} \right|^2 = 
\prod_{\substack{\rho \in \R \\ A_q(\rho)=0}} \left| \frac{(x-\rho)+i(y-\omega)}{(x-\rho)+i(y+\omega)} \right|^2 \\
& = 
\prod_{\substack{\rho \in \R \\ A_q(\rho)=0}} \left( 1 - \frac{4 \omega y}{(x - \rho)^2 + (y+\omega)^2} \right) <1
\endaligned
\]
for $z=x+iy$ with $y>0$, by the factorization \eqref{H-factorization}. 
Moreover, $E_{q,\omega}(z)$ has no real zeros for every $\omega>0$, 
by definition \eqref{def_E2} and by assumption. 
Hence $E_{q,\omega}(z)$ is a function of the class ${\rm HB}$ for every $\omega>0$. 

Conversely, suppose that $E_{q,\omega}(z)$ is a function of the class ${\rm HB}$ for every $\omega>0$. 
Then all zeros of $A_{q,\omega}(z)$ and $B_{q,\omega}(z)$ are real, simple, 
and they interlace; see \cite[Chapter VII, Theorems 3, 5, p. 313]{Levin80}, 
but note the footnote of the first page).  
In particular, $A_{q,\omega}(z)$ has only real zeros for every $\omega>0$. 
Hence $A_q(z)=\lim_{\omega \searrow 0}A_{q,\omega}(z)$ has only real zeros 
by Hurwitz's Theorem (\cite[Theorem (1,5)]{Marden66}). 
\end{proof}

\noindent
{\bf Proof of necessity.} 
By Lemma \ref{lem_2_1}, $P_g(x)$ has a zero outside $T$
if and only if 
$E_{q,\omega}(z)$ is not a function of the class HB 
for some $\omega>0$.  
Hence it is sufficient to prove that 
$E_{q,\omega_0}(z)$ is not a function of the class HB 
if there exists $\omega_0>0$ 
such that  $\delta_{n}(\underline{c}\,; q^{\omega_0}) \leq 0$ 
or $\delta_{n}(\underline{c}\,; q^{\omega_0})^{-1} \leq 0$ 
for some $1 \leq n \leq 2g$. 
This is proved similarly to Theorem \ref{thm_01}.  
\hfill $\Box$
\medskip

\noindent
{\bf Proof of sufficiency.} 
Let $\omega>0$. Suppose that $\delta_{n}(\underline{c}\,;q^\omega) > 0$ 
and $\delta_{n}(\underline{c}\,;q^\omega)^{-1} > 0$
for every $1 \leq n \leq 2g$.  
Then we can prove that $E_{q,\omega}(z)$ is a function of the class HB 
in a way similar to the proof of Theorem \ref{thm_02}. 
Therefore, all roots of $P_g(x)$ lie on $T$, by Lemma \ref{lem_2_1}, 
if $\delta_{n}(\underline{c}\,;q^\omega) > 0$ 
and $\delta_{n}(\underline{c}\,;q^\omega)^{-1} > 0$
for every $1 \leq n \leq 2g$ and $\omega>0$.     
\hfill $\Box$

\subsection{Proof of Theorem \ref{thm_5}.}  
Using \eqref{def_m1} and \eqref{def_v1}, starting from the initial vectors 

\begin{equation} \label{0704_1}
\Omega_{0,\omega} =
\begin{bmatrix}
{\mathbf a}_{0,\omega} \\
{\mathbf a}_{0,\omega}
\end{bmatrix}
\quad \text{and} \quad 
\tilde{\Omega}_{0,\omega} =
\begin{bmatrix}
{\mathbf a}_{0,\omega} \\
\omega^{-1}{\mathbf a}_{0,\omega}
\end{bmatrix}, 
\end{equation}
respectively, where
\[ 
{\mathbf a}_{0,\omega}
= {}^{t}
\begin{bmatrix}
c_0 \, q^{g \omega} &
c_1 \, q^{(g-1)\omega} &
\cdots &
c_{g-1} \, q^{\omega} &
c_g &  
c_{g-1} \, q^{-\omega} & 
\cdots &
c_0 \, q^{-g\omega}
\end{bmatrix}, 
\]
we define $(\gamma_{n}(\underline{c}\,;q^\omega), \Omega_g(n))$ and 
$(\tilde{\gamma}_{n}(\underline{c}\,;q^\omega), \tilde{\Omega}_g(n))$ for $1 \leq n \leq 2g$. 
Then $\Omega_{0,\omega}$ is equal to the initial vector \eqref{0423_5} associated with \eqref{0503_8}. 
Therefore, 
\[
\delta_n(\underline{c}\,;q^\omega) = 
\begin{cases}
\displaystyle{
\frac{q^{g\omega}-q^{-g\omega}}{q^{g\omega}+q^{-g\omega}}\gamma_{1}(\underline{c}\,;q^\omega)
\prod_{j=1}^{J} \frac{\gamma_{2j+1}(\underline{c}\,;q^\omega)}{\gamma_{2j}(\underline{c}\,;q^\omega)}} & \text{if $n=2J+1 \geq 1$}, \\
\displaystyle{\prod_{j=0}^{J} \frac{\gamma_{2j+2}(\underline{c}\,;q^\omega)}{\gamma_{2j+1}(\underline{c}\,;q^\omega)}} & \text{if $n=2J+2 \geq 2$}, 
\end{cases}
\]
by definition \eqref{0503_9}, Theorem \ref{0503_6}, and Proposition \ref{0503_4}. 
On the other hand, 
$\tilde{\gamma}_{n}(\underline{c}\,;q^\omega)=\omega\cdot \gamma_{n}(\underline{c}\,;q^\omega)$ 
for every $1 \leq n \leq 2g$, by \eqref{def_m1}, \eqref{def_v1}, and definition \eqref{0704_1}. 
This implies 
$\delta_n(\underline{c}\,;q^\omega)=\tilde{\delta}_n(\underline{c}\,;q^\omega)$ 
for every $1 \leq n \leq 2g$, 
if we define $\tilde{\delta}_n(\underline{c}\,;q^\omega)$ by 
\[
\tilde{\delta}_n(\underline{c}\,;q^\omega) = 
\begin{cases}
\displaystyle{
\frac{1}{\omega} \frac{q^{g\omega}-q^{-g\omega}}{q^{g\omega}+q^{-g\omega}}\tilde{\gamma}_{1}(\underline{c}\,;q^\omega)
\prod_{j=1}^{J} \frac{\tilde{\gamma}_{2j+1}(\underline{c}\,;q^\omega)}{\tilde{\gamma}_{2j}(\underline{c}\,;q^\omega)}} & \text{if $n=2J+1 \geq 1$}, \\
\displaystyle{\prod_{j=0}^{J} \frac{\tilde{\gamma}_{2j+2}(\underline{c}\,;q^\omega)}{\tilde{\gamma}_{2j+1}(\underline{c}\,;q^\omega)}} & \text{if $n=2J+2 \geq 2$}.
\end{cases}
\]
Hence, for Theorem \ref{thm_5}, it is sufficient to prove that
\[
\lim_{q^\omega \searrow  1}\tilde{\delta}_n(\underline{c}\,;q^\omega)=\delta_n(\underline{c}) \quad (1 \leq n \leq 2g). 
\]
This equality follows from the formula  
\[
\lim_{q^\omega \searrow 1} \tilde{\Omega}_{1,\omega} 
= 
\begin{bmatrix}
{\mathbf a}_{1} \\
{\mathbf b}_{1}
\end{bmatrix} \quad 
\left(
\tilde{\Omega}_{1,\omega}:=
P_{2g-1}(\tilde{\gamma}_{1}(\underline{c}\,;q^\omega))^{-1}Q_{2g-1} \cdot \tilde{\Omega}_{0,\omega} \right), 
\]
by definitions of $\tilde{\gamma}_{n}(\underline{c}\,;q^\omega)$ and $\gamma_{n}(\underline{c}\,;\log q)$, 
where 
\[
{\mathbf a}_{1}
=\begin{bmatrix}
2c_0 \\
\frac{2g-1}{g}c_1 \\
\frac{2g-2}{g}c_2 \\
\vdots \\
\frac{g+1}{g}c_{g-1} \\
c_g \\
\frac{g-1}{g}c_{g-1} \\
\vdots \\
\frac{2}{g}c_2 \\
\frac{1}{g}c_1 \\
\end{bmatrix}, \quad 
{\mathbf b}_{1}
=
\begin{bmatrix}
2g c_0 \log q \\
(2g-1) c_1 \log q\\
(2g-2) c_2 \log q \\
\vdots \\
(g+1)c_{g-1} \log q \\
g c_g \log q \\
(g-1)c_{g-1} \log q \\
\vdots \\
2 c_2  \log q  \\
 c_1 \log q 
\end{bmatrix}
\] 
which are vectors of \eqref{0503_10} and \eqref{0503_11} with 
$\gamma_1=(C_{-g}+C_g)/(C_{-g}-C_g)$ associated with \eqref{0418_1}. 
Put
\[
\begin{bmatrix}
\tilde{{\mathbf a}}_{1,\omega} \\
\tilde{{\mathbf b}}_{1,\omega}
\end{bmatrix}
:= \tilde{\Omega}_{1,\omega} = 
P_{2g-1}(\tilde{\gamma}_{1}(\underline{c}\,;q^\omega))^{-1}Q_{2g-1}
\cdot \tilde{\Omega}_{0,\omega}.
\]
Then, using the formula of $P_{k}(m_{k})^{-1}Q_{k}$ in Lemma \ref{lem_203}, 
we have 
\[
\aligned
\tilde{{\mathbf a}}_{1,\omega}
& =\begin{bmatrix}
2\cosh(g\log q^\omega)c_0 \\
(\cosh((g-1)\log q^\omega)+\omega^{-1} \gamma\sinh((g-1)\log q^\omega))c_1 \\
(\cosh((g-2)\log q^\omega)+\omega^{-1} \gamma\sinh((g-2)\log q^\omega))c_2 \\
\vdots \\
(\cosh(\log q^\omega)+\omega^{-1} \gamma\sinh(\log q^\omega))c_{g-1} \\
c_g \\
(\cosh(\log q^\omega)-\omega^{-1} \gamma\sinh(\log q^\omega))c_{g-1} \\
\vdots \\
(\cosh((g-2)\log q^\omega)-\omega^{-1} \gamma\sinh((g-2)\log q^\omega))c_2 \\
(\cosh((g-1)\log q^\omega)-\omega^{-1} \gamma\sinh((g-1)\log q^\omega))c_1 
\end{bmatrix} \endaligned \]\[ \aligned
\tilde{{\mathbf b}}_{1,\omega}
& =
\begin{bmatrix}
2 \, \omega^{-1}\sinh(g\log q^\omega)c_0 \\
\gamma^{-1}(\cosh((g-1)\log q^\omega)+\omega^{-1} \gamma\sinh((g-1)\log q^\omega))c_1 \\
\gamma^{-1}(\cosh((g-2)\log q^\omega)+\omega^{-1} \gamma\sinh((g-2)\log q^\omega))c_2 \\
\vdots \\
\gamma^{-1}(\cosh(\log q^\omega)+\omega^{-1} \gamma\sinh(\log q^\omega))c_{g-1} \\
\gamma^{-1}c_g \\
\gamma^{-1}(\cosh(\log q^\omega)-\omega^{-1} \gamma\sinh(\log q^\omega))c_{g-1} \\
\vdots \\
\gamma^{-1}(\cosh((g-2)\log q^\omega)-\omega^{-1} \gamma\sinh((g-2)\log q^\omega))c_2 \\
\gamma^{-1}(\cosh((g-1)\log q^\omega)-\omega^{-1} \gamma\sinh((g-1)\log q^\omega))c_1 
\end{bmatrix}
\endaligned
\]
with
\[
\gamma=\tilde{\gamma}_{1}(\underline{c}\,;q^\omega)=\omega \frac{q^{g\omega}+q^{-g\omega}}{q^{g\omega}-q^{-g\omega}} = \omega \coth(g \log q^\omega). 
\]
Using 
\[
\aligned
\lim_{x \to 0^+}\Bigl[ \cosh((g-k)x)+\coth(gx)\sinh((g-k)x) \Bigr] &= \frac{2g-k}{g} \quad (0 \leq k \leq g), \\
\lim_{x \to 0^+}\Bigl[ \cosh((g-k)x)-\coth(gx)\sinh((g-k)x) \Bigr] &= \frac{k}{g} \quad (1 \leq k \leq g-1), \\
\lim_{x \to 0^+} \gamma^{-1}=\lim_{x \to 0^+} \frac{\log q}{x} \tanh(gx) = g \log q
\endaligned 
\]
for $x=\log q^\omega$, 
we obtain
\[
\lim_{q^\omega \searrow 1}
\begin{bmatrix}
\tilde{{\mathbf a}}_{1,\omega} \\
\tilde{{\mathbf b}}_{1,\omega}
\end{bmatrix}
=
\begin{bmatrix}
{\mathbf a}_{1} \\
{\mathbf b}_{1}
\end{bmatrix}.
\]
On the other hand, it is easy to see that 
\[
\lim_{q^\omega \searrow 1^+} \omega \frac{q^{g\omega}+q^{-g\omega}}{q^{g\omega}-q^{-g\omega}} = \frac{1}{g \log q}.
\]
Hence $\lim_{q^{\omega} \searrow  1} \delta_{n}(\underline{c}\,;q^{\omega})=\delta_{n}(\underline{c})$ for every $1 \leq n \leq 2g$. 
Because $\delta_{n}(\underline{c}\,;q^{\omega})$ is a rational function of $q^\omega$, 
we obtain the second formula of Theorem \ref{thm_5}.  \hfill $\Box$
%
%
\subsection{Remark on Theorem \ref{thm_5}.} 
%
%
We have 
\[
E_{q,\omega}(z) = A_q(z) - i \omega B_q(z) + O(\omega^2), 
\]
\[
A_{q,\omega}(z) = A_q(z) + O(\omega^2), \quad 
B_{q,\omega}(z) = \omega B_q(z) +O(\omega^3)
\]
as $\omega \to 0^+$ if $z$ lies in a compact subset of $\C$. 
Therefore, it seems that $E_{q,\omega}(z)$ is similar to $E_q(z)=A_q(z) - i B_q(z)$ for small $\omega>0$, 
but there is an obvious gap after the limit $\omega \to 0^+$ is taken. 
To resolve this gap, we consider 
\[
\tilde{E}_{q,\omega}(z) 
:= A_{q,\omega}(z) - \frac{i}{\omega} \, B_{q,\omega}(z). \\
\]
Then, 
\[
\aligned
\tilde{E}_{q,\omega}(z) 
&= A_q(z) - i\,B_q(z) + O(\omega^2) 
= E_q(z) + O(\omega^2),
\endaligned
\]
\[
\tilde{A}_{q,\omega}(z) 
:= \frac{1}{2}(\tilde{E}_{q,\omega}(z)+\tilde{E}_{q,\omega}^\sharp(z)) 
= A_{q,\omega}(z) 
= A_q(z) + O(\omega^2),
\]
\[
\tilde{B}_{q,\omega}(z)
:= \frac{i}{2}(\tilde{E}_{q,\omega}(z)-\tilde{E}_{q,\omega}^\sharp(z))
 = \frac{1}{\omega}B_{q,\omega}(z) 
 = B_q(z) +O(\omega^2)
\]
as $\omega \to 0^+$ if $z$ lies in a compact subset in $\C$. 
Hence we ``recovers'' $E_q(z)$ from $\tilde{E}_{q,\omega}(z)$ by taking the limit $\omega \to 0^+$. 
The initial vector $\tilde{\Omega}_{0,\omega}$ of \eqref{0704_1} is chosen 
to corresponds to $\tilde{A}_{q,\omega}(z)$ and $\tilde{B}_{q,\omega}(z)$. 
This is a reason Theorem \ref{thm_5} holds. 
However, despite this advantage, 
we chose $E_{q,\omega}(z)$, and not $\tilde{E}_{q,\omega}(z)$ in Section \ref{section_8_2}.  
One reason is the simple formula $E_{q,\omega}(z)=A_q(z+i\omega)$. 
By comparison, $\tilde{E}_{q,\omega}(z)$ has a slight complicated form 
\[
\tilde{E}_{q,\omega}(z) 
= \frac{A_{q}(z+i\omega)+A_{q}(z-i\omega)}{2} + \frac{A_{q}(z+i\omega)-A_{q}(z-i\omega)}{2\omega}.
\]
We do not know whether an analogue of Lemma \ref{lem_2_1} holds for $\tilde{E}_{q,\omega}(z)$. 
However, if such analogue holds, 
we may obtain results for $\tilde{E}_{q,\omega}(z)$ analogous to $E_{q,\omega}(z)$. 

\subsection{Comparison with classical results} \label{section_8_5}

A necessary and sufficient condition for all roots of $P(x) \in \C[x]$ to lie on $T$ 
is that $P(x)$ be {\bf self-inversive} (i.e., $P(x)=x^{\deg P}\overline{P(1/\bar{x})}$) 
and that all roots of the derivative $P'(x)$ lie inside or on $T$ 
(Gauss--Lucas~\cite{Lucas74}, Schur~\cite{Schur17}, Cohn~\cite{Cohn22}). If $P(x)$ is a self-inversive polynomial, 
$P'(x)$ has no zeros on $T$ except at the multiple zeros of $P(x)$ 
(\cite[Lemma (45.2)]{Marden66}). 
Therefore, a necessary and sufficient condition for all roots of the self-inversive 
polynomial $P(x) \in \C[x]$ to be simple and on $T$ is that all the roots of $P'(x)$ lie inside $T$. 
The following classical result is quite useful for checking this condition. 
\begin{theorem} \label{tak} Let $Q(x)=a_0x^n+a_1x^{n-1}+\cdots+a_n$ be a complex polynomial of degree $n$. 
Let $D_n(Q)$ be the $2n \times 2n$ matrix
\[
D_n(Q) = 
\left[
\begin{array}{cccc|cccc}
a_0 & a_1 & \cdots & \cdots & a_n & & & \\
    & a_0 & a_1 & \cdots & a_{n-1} & a_n & & \\
    &     & \ddots & \ddots & \ddots & \ddots & \ddots & \\
    &     &  & a_0 & a_1 & \cdots & \cdots & a_n \\ \hline 
\bar{a}_n & \bar{a}_{n-1} & \cdots & \cdots &  \bar{a}_0 & & & \\
    & \bar{a}_n & \bar{a}_{n-1} & \cdots & \bar{a}_1 & \bar{a}_0 & & \\
    &     & \ddots & \ddots & \ddots & \ddots & \ddots& \\
    &     &     & \bar{a}_n & \bar{a}_{n-1} & \cdots & \cdots & \bar{a}_0 
\end{array}
\right]
\]
which is  the the resultant of $Q$ and $Q^\sharp$. 
For $k=2,\dots,n$, 
given the $2k \times 2k$ matrix $D_k(Q)$, 
define the $2(k - 1) \times 2(k - 1)$ matrix $D_{k-1}(Q)$ 
by deleting the $k$-th and $2k$-th rows and columns of $D_k(Q)$. 
In particular,
\[
D_2(Q) = \left[ 
\begin{array}{cc|cc}
a_0 & a_1 & a_n &  \\
    & a_0 & a_{n-1} & a_n  \\ \hline 
\bar{a}_n & \bar{a}_{n-1} & \bar{a}_0 &  \\
    & \bar{a}_n & \bar{a}_1 & \bar{a}_0 
\end{array}
\right], \quad 
D_1(Q) = 
\begin{bmatrix}
a_0 & a_n \\ \bar{a}_n & \bar{a}_0
\end{bmatrix}.
\]
All roots of $Q(x)$ lie inside $T$ 
if and only if $\det D_k>0$ for every $1 \leq k \leq n$.  
\end{theorem}

For a proof, see 
\cite[Section 75, Problems 4, 5]{Takagi} and 
the Schur-Cohn criterion \cite[Section 43, Theorem (43,1)]{Marden66}, 
together with \cite[Section 43, Exercise 2 and Section 45, Exercise 3]{Marden66}.
\medskip

Applying Theorem \ref{tak} to $Q(x)=P_g^\prime(x)$, we find that 
all roots of $P_g(x)$ are simple and on $T$ 
if and only if $\det D_n(P_g')>0$ for every $1 \leq n \leq 2g-1$. 
For $g=2$, we have 
\[
\aligned
 \det D_1(P_2') & = (4c_0 - c_1)(4c_0 + c_1) \\
 \det D_2(P_2') & = 4(8c_0^2-2c_1^2+4c_0c_2)(8c_0^2+c_1^2-4c_0c_2) \\
 \det D_3(P_2') & = 16(2c_0+2c_1+c_2)(2c_0-2c_1+c_2)(8c_0^2+c_1^2-4c_0c_2)^2.
\endaligned
\]
On the other hand, 
\[
\aligned
\delta_2(P_2) = \frac{4c_0+c_1}{4c_0-c_1}, \quad 
\delta_3(P_2) = \frac{8c_0^2-2c_1^2+4c_0c_2}{8c_0^2+c_1^2-4c_0c_2}, \quad 
\delta_4(P_2) = \frac{2c_0+2c_1+c_2}{2c_0-2c_1+c_2}, 
\endaligned
\]
where $\delta_n(P_g)=\delta_n(\underline{c})$ for corresponding $\underline{c}$. 

As in the above examples, it is expected that 
$\delta_{n+1}(P_g)$ and $\det D_n(P_g')$ have the same sign for every $1 \leq n \leq 2g-1$ 
even if $P_g(x)$ has a root outside $T$. 
We do not touch such a problem here 
but mention that Theorems \ref{thm_1}, \ref{tak} for self-reciprocal polynomials 
are the same level from a computational point of view, because $\delta_1(P_g)=1$, in general, 
and a computation of $\det(E^+(\underline{C}) \pm E^{-}(\underline{C})J_n)$ 
is reduced to a computation of a determinant of a size $n$ matrix: 
\[
\det(E^+(\underline{C}) \pm E^{-}(\underline{C})J_n)=C_{-g}^{2g+1-n}\det([E^+(\underline{C}) \pm E^{-}(\underline{C})J_n]_{\nwarrow n}).
\]
That is, essentially, $\delta_{n+1}(P_g)$ is a determinant of a square matrix of size $n+1$. 
Therefore, it seems plausible that the criterion of Theorem \ref{thm_1} can be deduced from the known criterion using some linear algebra identities. 
However, the author does not have an idea how to realize such argument. 
\medskip

As mentioned in \cite{Marden66} and \cite{Takagi}, 
an origin of Theorem \ref{tak} is in work of Hermit and Hurwitz. 
They related the distribution of roots of polynomials 
with the signature of quadratic forms. 
Concerning the approach of the present paper, 
the reproducing kernel \eqref{lem_602_1} of the de Branges space $B(E)$ may play 
a role of quadratic forms. 

\subsection{Concluding remarks.} 

Corollary \ref{cor_602} and arguments of Subsections \ref{section_5_1}, \ref{section_5_2}, and \ref{section_8_1}, 
show that a self-reciprocal polynomial $P_g(x)$ of \eqref{def_Pg} has 
only simple zeros on $T$ if and only if 
there exists $2g$ positive real numbers $\gamma_1,\cdots,\gamma_{2g}$ such that 
\[
P_g(x) 
= \frac{P_g(1)}{2^{2g}} 
[\,1\,~\,0\,]
\prod_{n=1}^{2g}
\begin{bmatrix}
(x+1) & i\gamma_{n}(x-1)\\
-i\gamma_{n}^{-1}(x-1) & (x+1)
\end{bmatrix}
\begin{bmatrix}
1 \\ 0
\end{bmatrix},
\]
since $P_g(1)=E_q(0)=A_q(0)$. 
Compare this with the factorization  
\[
P_g(x) 
= P_g(0) \prod_{j=1}^{g}(x^2-2\lambda_j x+1) \quad (\lambda_j \in \C).
\]
As described in the introduction and in Section \ref{section_7}, 
we have at least two simple algebraic algorithms for calculating $\gamma_1,\cdots,\gamma_{2g}$ from coefficients $c_0,\cdots,c_g$, 
but it is impossible to calculate $\lambda_1,\cdots,\lambda_{g}$, since the Galois group of a general self-reciprocal polynomial $P_g(x)$  
is isomorphic to ${\frak S}_g \ltimes (\Z/2\Z)^g$. 
In addition, it is understood that the positivity of $\gamma_1,\cdots,\gamma_{2g}$ is equivalent to the positivity of a Hamiltonian, 
but a plausible meaning of $|\lambda_j|<1$ and $\lambda_i \not= \lambda_j$ ($i \not= j$) is not clear. 
\smallskip

We now comment on two important classes of self-reciprocal polynomials. 

The first one is the class of zeta functions of smooth projective curves $C/{\mathbb F}_q$ of genus $g$:
$Z_C(T)=Q_C(T)/((1-T)(1-qT))$,
where $Q_C(T)$ is a polynomial of degree $2g$ satisfying the functional equation $Q_C(T)=(q^{1/2} T)^{2g} Q_C(1/(qT))$. 
Hence $P_C(x)=Q_C(q^{-1/2}x)$ is a self-reciprocal polynomial of degree $2g$ with real coefficients. 
Weil~\cite{Weil45} proved that all roots of $P_C(x)$ lie on $T$ 
as a consequence of Castelnuovo's positivity for divisor classes on $C \times C$. 

The second one is the class of polynomials $P_A(x)$ 
attached to $n \times n$ real symmetric matrices $A=(a_{i,j})$ 
with $|a_{i,j}| \leq 1$ for every $1 \leq i<j \leq n$ (no condition is imposed on the diagonal):
$P_A(x) = \sum_{I \sqcup J =\{1,2,\cdots,n\}} x^{|I|} \prod_{i \in I, j \in J} a_{i,j}$, 
where $I \sqcup J$ means a disjoint union. 
Polynomials $P_A(x)$ are obtained as the partition function of a ferromagnetic Ising model
and are self-reciprocal polynomials of degree $n$ with real coefficients. 
The fact that all roots of any $P_A(x)$ lie on $T$ 
is known as the  Lee-Yang Circle Theorem~\cite{LeeYang52}. 
Ruelle~\cite{Ruelle71} extended this result and characterized the polynomials $P_A(x)$ in terms of multi-affine polynomials 
being symmetric under certain involution on the space of multi-affine polynomials ~\cite{Ruelle10}. 

It seems that a discovery of arithmetical, geometrical, or physical interpretation 
of the positivity of $\gamma_1,\cdots,\gamma_{2g}$ or $H(a)$ 
(for some restricted class of polynomials) is quite an interesting and important problem. 
We hope that our formula of $\gamma_1,\cdots,\gamma_{2g}$ contributes to such philosophical interpretation. 

%


\bigskip \noindent
\\
Department of Mathematics,  
Tokyo Institute of Technology \\
2-12-1 Ookayama, Meguro-ku, 
Tokyo 152-8551, JAPAN  \\
Email: {\tt msuzuki@math.titech.ac.jp}

\end{document}